\documentclass[aap]{imsart}

\RequirePackage{amsthm,amsmath,amsfonts,amssymb}
\RequirePackage[numbers,sort&compress]{natbib}
\RequirePackage[colorlinks,citecolor=blue,urlcolor=blue]{hyperref}
\RequirePackage{graphicx}

\startlocaldefs
\theoremstyle{plain}

\newtheorem{theorem}{Theorem}[section]
\newtheorem{lemma}[theorem]{Lemma}

\newtheorem{proposition}[theorem]{Proposition}

\newtheorem{definition}{Definition}[section]
\newtheorem{remark}[theorem]{Remark}

\endlocaldefs

\begin{document}

\begin{frontmatter}
\title{McKean-Vlasov SDE and SPDE with Locally Monotone Coefficients}
\runtitle{MVS(P)DEs with Locally Monotone Coefficients}

\begin{aug}
\author[A]{\fnms{Wei}~\snm{Hong}\ead[label=e1]{weihong@jsnu.edu.cn}},
\author[B]{\fnms{Shanshan}~\snm{Hu}\ead[label=e2]{shu@math.uni-bielefeld.de}}
\and
\author[A]{\fnms{Wei}~\snm{Liu}\ead[label=e3]{weiliu@jsnu.edu.cn}}
\address[A]{School of Mathematics and Statistics, Jiangsu Normal University, Xuzhou 221116, China \printead[presep={,\ }]{e1,e3}}

\address[B]{Faculty of Mathematics, Bielefeld University, 33615 Bielefeld, Germany    \printead[presep={,\ }]{e2}}
\end{aug}

\begin{abstract}
In this paper we mainly investigate  the strong and weak well-posedness of a class of McKean-Vlasov stochastic (partial) differential equations. The main existence and uniqueness results state that we only need to impose some local assumptions on the coefficients, i.e. locally monotone condition both in state variable and distribution variable, which cause some essential difficulty since the coefficients of McKean-Vlasov stochastic equations typically are   nonlocal.  Furthermore, the large deviation principle is also derived for the McKean-Vlasov stochastic equations under those weak assumptions. The wide applications of main results are illustrated by various concrete examples such as the granular media equations, plasma type models, kinetic equations, McKean-Vlasov type porous media equations and Navier-Stokes equations. In particular, we could remove or relax some typical assumptions previously imposed on those models.
\end{abstract}



\end{frontmatter}
\tableofcontents
\section{Introduction}
McKean-Vlasov stochastic differential equations (MVSDEs), also referred as distribution dependent SDEs (DDSDEs) or  mean-field SDEs in the literature, have received a great deal of attention in recent years, which originated from the seminal works \cite{M,V} by McKean and Vlasov who were initially inspired by Kac's programme in the kinetic theory (cf.~\cite{KAC}).
  These kind of models are more involved than classical SDEs and could be interpreted as the weak limit of $N$-interacting particle systems, as $N\to\infty$, in which the particles interact in a mean field way (i.e. the coefficients depend on the empirical measure of the system). The prototype of such stochastic system is in the form of
  $$dX^i(t)=\frac{1}{N}\sum_{j=1}^Nb(X^i(t),X^j(t))dt+dW^i(t),$$
  let $N\to \infty$ and one can get the decoupled SDEs that interacts with the distribution of solution, i.e.
  $$dY^i(t)=\int b(Y^i(t),y)\mathcal{L}_{Y^i(t)}(dy)dt+dW^i(t),$$
 where $\mathcal{L}_{Y^i(t)}$ is the distribution of $Y^i(t)$. Such limiting behaviour is often called the propagation of chaos  in the study of stochastic
dynamics of particle systems, we refer the reader to the classical survey by  Sznitman \cite{S1} and the papers \cite{GJS,HSS1,JW,TLN} for more background on this topic.

\subsection*{Well-posedness in finite dimensions}
In this paper, the first aim is to  consider the strong and weak well-posedness of following general MVSDEs
\begin{equation}\label{eq01}
dX(t)=b(t,X(t),\mathcal{L}_{X(t)})dt+\sigma(t,X(t),\mathcal{L}_{X(t)})dW(t),
\end{equation}
where $W(t)$ is a standard $m$-dimensional Wiener process and the coefficients $b,\sigma$ satisfy the locally monotone condition.  Before mentioning our main results, let us  briefly summarize some recent works on the MVSDEs. Note that, in the classical reference \cite{S1}, Sznitman derived the existence and uniqueness of solutions to (\ref{eq01}) by utilizing a fixed point argument on Wasserstein space while the coefficients $b$ and $\sigma$ are globally Lipschitz continuous. In order to deal with the well-posedness of (\ref{eq01}) where the coefficients satisfy globally monotone  (also called one-side Lipschitz) conditions, the technique of distribution iteration was carried out by Wang in \cite{WFY}, in addition, the exponential ergodicity and some asymptotic estimates were also studied by using the coupling argument and Harnack type inequalities for McKean-Vlasov equations. Afterwards, Huang and Wang further developed the fixed point techniques  used in \cite{S1} and systematically investigated the existence and uniqueness of solutions to (\ref{eq01}) with singular coefficients. For instance, in \cite{HW2} they considered (\ref{eq01}) where the drift term contains a linear growth term in state-distribution and a locally integrable term in time-state, while the noise term is weakly differentiable in state variable and Lipschitz continuous in distribution w.r.t.  the sum of Wasserstein and weighted variation distances, we refer to \cite{HW1,HW3,HRW,REN1,RZ21,WFY1} for the further investigations of MVSDEs with singular coefficients and   \cite{F1,HSS} for the Lyapunov type conditions.

However, to the best of authors' knowledge, there are only few results on the MVSDEs under local assumptions. It seems that the first result considering the local conditions on the coefficients was in \cite{KL} by  Kloeden and Lorenz, where they assumed locally Lipschitz w.r.t. measure and globally Lipschitz w.r.t. state variable. Erny \cite{E1} discussed the strong well-posedness of (\ref{eq01}) with the coefficients $b,\sigma$ fulfilling some locally Lipschitz continuity w.r.t. both the  state and measure variables, whereas he need to impose the uniform boundedness on  the coefficients. Afterwards, Galeati et al. in~\cite{GHM} extended the results of \cite{E1} to the assumption of exponential integrability instead of the uniform boundedness. Recently, Li et al.~\cite{LMSWY} established the strong convergence of Euler-Maruyama schemes for approximating MVSDEs by assuming globally Lipschitz w.r.t. measure and one-side locally Lipschitz type condition w.r.t. state variable with logarithmic growth on Lipschitz constants.

One simple motivating example is the following MVSDE ($p\geq 1$)
\begin{equation}\label{exa10}
dX(t)=-X(t)^3\Bigg[\int_{\mathbb{R}^d}|y|^p\mathcal{L}_{X(t)}(dy)\Bigg] dt+X(t)dW(t).
\end{equation}
It is clear that $-u^3$ satisfies the standard monotone condition (cf.~e.g.~\cite{LR1,WFY1}), however, the drift coefficient of (\ref{exa10}) is no longer monotone.
We will show that it satisfies certain ``locally monotone'' condition (cf.~Remark \ref{remark2.2} (iii) for the details).

 We want to mention that in general  the uniqueness of solutions to (\ref{eq01}) does not hold  if the drift is only locally Lipschitz (see Remark \ref{re1} in Section \ref{sec2.2} for the counterexample). In fact, the classical localization argument seems to be  unsuitable in dealing with the well-posedness of MVSDEs. If we take (\ref{eq01}) and consider the stopped process $Z(t):=X(t\wedge\tau)$ with a stopping time $\tau$, then $Z(t)$ actually solves the following equation
$$Z(t)=Z(0)+\int_0^{t\wedge\tau}b(s,Z(s),\mathcal{L}_{X(s)})d s+\int_0^{t\wedge\tau}\sigma(s,Z(s),\mathcal{L}_{X(s)})d W(s),$$
where $X(t)$ is still involved instead of only $Z(t)$.
Motivated by this,  our first goal  is  the  strong and weak existence and uniqueness of solutions to (\ref{eq01}) with locally monotone coefficients. Specifically, for the existence of (strong) solutions, we mainly assume the coefficients satisfy that for any $t,R\geq0$, $x,y\in\mathbb{R}^d$ with $|x|,|y|\leq R$ and $\mu,\nu\in\mathcal{P}_{\kappa}(\mathbb{R}^d)$,
\begin{align}\label{ass0}
&2\langle b(t,x,\mu)-b(t,y,\nu),x-y\rangle+\|\sigma(t,x,\mu)-\sigma(t,y,\nu)\|^2\nonumber\\
&\leq\big(K_t(R)+C\mu(|\cdot|^{\kappa})+C\nu(|\cdot|^{\kappa})\big)\big(|x-y|^2+\|\mu-\nu\|_{2,var}\big),
\end{align}
where $K_t(R)$ is an $\mathbb{R}_{+}$-valued function and $\|\cdot\|_{2,var}$ stands for the weighted variation distance. Furthermore, under certain additional assumptions (see $(\mathbf{A2}')$-$(\mathbf{A2}''')$ in Theorem \ref{th1}), we also prove the pathwise uniqueness for  (\ref{eq01}) and then derive the desired existence and uniqueness results.
Note that  Ren et al.~\cite{RTW} recently introduced a technique of local approximation, which is applicable to path-distribution dependent stochastic transport type equations. However, such technique is unsuitable in the case of locally monotone conditions (e.g. (\ref{ass0})). To solve this problem here we combine the technique of Euler type approximation with the martingale approach.

Compared to the existing works (e.g.~\cite{E1,GHM,Hu,KL,LMSWY,RTW}), for the existence of strong solutions, we here assume the locally monotone  condition (\ref{ass0}) and utilize the weighted variation distance to measure the distance among distributions, which is more general than that considered in the aforementioned references. In fact,  this result generalizes the classical work by Krylov and Rozovskii \cite{Kr,KR} to the McKean-Vlasov case, the difference compared to \cite{Kr,KR} is that we need to impose some growth condition on the coefficients (see $(\mathbf{A4})$ in Theorem \ref{th1}) instead of the local integrability (cf.~\cite[(3.1)]{LR1}) due to some technical restrictions. For the uniqueness of solutions, as we mentioned before, it does not hold if the drift is merely locally Lipschitz, thus we shall provide different type assumptions (i.e.~$(\mathbf{A2}')$-$(\mathbf{A2}''')$ in Theorem \ref{th1}) to guarantee it, which also helps us to extend the corresponding results to the infinite dimensional case.

Another motivation of considering MVSDEs is that it plays an important role in characterizing nonlinear Fokker-Planck-Kolmogorov (FPK for short) equations (cf.~\cite{BR1,BR2,HSS,S1,WFY,WFY1} and references therein).
  As one application of Theorem \ref{th1}, we could investigate the following granular media equations
\begin{eqnarray}\label{eq06}
\partial_tf_t=\Delta f_t+{\rm{div}}\Big\{f_t\nabla V+f_t\nabla(W\ast f_t)\Big\},
\end{eqnarray}
where the potential $V:\mathbb{R}^d \rightarrow \mathbb{R}$, the interaction functional
$W:\mathbb{R}^d \rightarrow \mathbb{R}$ and
\begin{equation*}
(W\ast f_t)(x):=\int_{\mathbb{R}^d}W(x-z)f_t(z)dz.
\end{equation*}
Such type of model arises from a conglomeration of discrete solid macroscopic particles (or grains) which are characterized by a loss of kinetic energy whenever the particles
interact, the most common loss is due to friction when the grains collide, see \cite{DSG} for more background from physics. Typically, the convexity of potential terms is
assumed in the literature, see e.g.~\cite{AG1,BCP} and references therein. Based on the main result of this work, we can drop the assumption of convexity and merely use locally Lipschitz condition with certain growth to guarantee the existence of weak solutions of PDE (\ref{eq06}).
In \cite{WFY1}, Wang also discussed granular media equations (\ref{eq06}) with Neumann boundary  under the assumptions of weak differentiability with $\|\nabla W\|_{\infty}<\infty$ and integrability $|\nabla V|\in L^p$ for some $p>d\vee2$. In this work we can allow the terms $\nabla W$ to be unbounded.

For more applications of MVSDEs such as plasma type models and kinetic equations, one can see Section \ref{sec6} below.

\subsection*{Well-posedness in infinite dimensions}
Based on the results in finite dimensions, our second goal is to establish the existence and uniqueness of solutions to a class of McKean-Vlasov  SPDEs (MVSPDEs) on Hilbert space
\begin{equation}\label{eq03}
dX(t)=A(t,X(t),\mathcal{L}_{X(t)})dt+B(t,X(t),\mathcal{L}_{X(t)})dW(t),
\end{equation}
where $W(t)$ stands for a Hilbert-space-valued cylindrical Wiener process and $A,B$ satisfy some locally monotone conditions, which  extend not only the classical result by Krylov and Rozovskii \cite{KR} but also the further  generalization by Liu and R\"{o}ckner \cite{LR2,LR13}.
The classical variational framework was first developed  by Pardoux, Krylov and Rozovskii (see \cite{KR, RS}), where they employed the well-known monotonicity method to prove the existence and uniqueness of solutions to SPDEs satisfying the monotonicity and coercivity conditions.
This classical framework has been substantially
generalized in \cite{LR2,LR13,LR1} to more general SPDEs satisfying local monotonicity and generalized coercivity conditions, which are applicable to various quasilinear and semilinear SPDEs.

As the second main contribution of this work,  we develop a  general  framework also include the  McKean-Vlasov  case (see Theorem \ref{th2}), in particular, we assume the coefficients of (\ref{eq03}) satisfy
\begin{eqnarray}\label{es0}
\!\!\!\!\!\!\!\!&&2_{V^*}\langle A(t,u,\mu)-A(t,v,\nu),u-v\rangle_V+\|B(t,u,\mu)-B(t,v,\nu)\|_{L_2(U,H)}^2
\nonumber\\
\!\!\!\!\!\!\!\!&&\leq
\big(C+\rho(v)+C\nu(\|\cdot\|_H^{\theta})\big)\|u-v\|_H^2
+C\big(1+\nu(\|\cdot\|_H^{\theta})\big)\mathbb{W}_{2,H}(\mu,\nu)^2,
\end{eqnarray}
where  $\rho$ is a nonnegative measurable and locally bounded function. The main result is applicable to
 various types of SPDEs perturbed by interaction external forces, such as  stochastic porous media equations and  stochastic 2D Navier-Stokes equations, which could be seen as the mean field limit of $N$-interacting SPDE systems (see Remark \ref{remark6.2}).

Notice that  MVSPDEs have also been investigated by many researchers. For instance, Shen et al.~\cite{SSZZ} recently investigated the large $N$ problem in the quantum field theory of $O(N)$ linear sigma model on $2$-dimensional torus $\mathbb{T}^2$. They considered the $N\to\infty$ limit  and proved that a singular MVSPDE governs the limiting dynamics. In addition, this large $N$ limit problem could also  be described as a typical mean field limit result in the theory of SPDE systems. Chiang et al.~\cite{CKS} proved the existence and uniqueness of solutions for a class of McKean-Vlasov semilinear SPDEs, and discussed the asymptotic behavior of the sequence of empirical measures determined by the solutions of interacting system of $n$-SDEs taking
values in the dual of a countably Hilbertian nuclear space, see also \cite{BKK} for the extension to a general Hilbert space.
 As the applications, the results in \cite{CKS} can be used to describe the random strings and the fluctuation of voltage potentials of interacting spatially extended neurons, which are governed by the following weakly interacting SPDE systems
 $$dX^{N,i}(t)=\Big(\Delta X^{N,i}(t)+\frac{1}{N}\sum_{j=1}^Nb_t(X^{N,i}(t),X^{N,j}(t))\Big)dt+dW^i(t),~1\leq i\leq N.$$
The latter is a more realistic model for large numbers of neurons in close proximity to each other. In  the paper \cite{ES1}, E and Shen showed the well-posedness and  propagation of chaos results for the following MVSPDE
 \begin{equation*}
   dX(t)=\big[\Delta X(t)+\Phi(t,X(t),\mathcal{L}_{X(t)})\big]dt+B(t,X(t),\mathcal{L}_{X(t)})dW(t),
   \end{equation*}
 with $\Phi,B$ satisfying globally Lipschitz condition, which  justifies the mean field approximation when the polymer system is concentrated. For more physical point of view for this type of models, one can see~\cite{O1}.
 One should note that our solution theory for MVSPDEs (\ref{eq03}) here could be directly applied to obtain the well-posedness for those MVSPDE models derived in \cite{CKS,BKK,ES1}, moreover, instead of assuming Lipschitz conditions on the coefficients in \cite{CKS,BKK,ES1},  one only need to require locally monotone condition by using our current result, see Subsection \ref{secin} for more details.
\subsection*{Large deviation principle}
Apart from the existence and uniqueness of solutions, we are also interested in investigating some asymptotic properties, for example, the Freidlin-Wentzell type large deviation principle (LDP for short) for MVSPDEs with small perturbations,
\begin{equation}\label{eq04}
dX^\varepsilon(t)=A(t,X^\varepsilon(t),\mathcal{L}_{X^\varepsilon(t)})dt+\sqrt{\varepsilon}B(t,X^\varepsilon(t),\mathcal{L}_{X^\varepsilon(t)})dW(t),~X^\varepsilon(0)=x.
\end{equation}
The LDP, as one  important topic in probability theory, mainly concerns the asymptotic behaviour of remote tails of a family of probability distributions, which has wide applications in several fields such as  statistics and information theory.
 The reader can refer to the classical monographs \cite{DZ,V1}  for the LDP theory and its applications. After the seminal work by  Freidlin and Wentzell \cite{FW},  the Freidlin-Wentzell type LDP for SDEs or SPDEs have been extensively studied in the past decades.  In particular,  the weak convergence approach developed by Budhiraja et al.~\cite{BD}
 has been proved to be very effective to study the Freidlin-Wentzell's  LDP for SPDEs, we refer the  readers  to e.g.~\cite{BDM,BPZ,CM,DE,L1,MSZ} and references therein.

However, there is much less results on LDP for McKean-Vlasov SDEs or SPDEs in the literature. Very recently, dos Reis et al.~\cite{DST} investigated the LDP for MVSDEs (\ref{eq01}) in the sense of uniform and H\"{o}lder topologies, where the coefficients satisfy the locally Lipschitz continuity but global monotonicity w.r.t.~the state variable and globally Lipschitz continuity w.r.t.~the measure (see \cite[Assumption 3.2]{DST}), by applying the argument of classical time discretization and exponential equivalence. Afterwards,  Adams et al.~\cite{ADRST} extended the results of \cite{DST} to the case of reflected MVSDEs with self-stabilising terms. By employing the weak convergence approach, Liu et al.~\cite{LSZZ}   investigated the LDP and moderate deviation principle for  MVSDEs under a similar framework as \cite{DST} but with L\'{e}vy noise, see also \cite{HLL3} for the extension to the infinite dimensional case. Inspired by \cite{ADRST,DST,HLL3,LSZZ,SY}, one natural question is whether the Freidlin-Wentzell type LDP holds for McKean-Vlasov equations under fully local assumptions or not?

Therefore, the third goal of this work is to establish the LDP for (\ref{eq04}) with locally monotone coefficients. The main strategy is based on the powerful weak convergence approach.  As an essential part of the proof, it is necessary to find the correct form of the skeleton equation.
Heuristically, if $\varepsilon\rightarrow0$ in equation (\ref{eq04}), it reduces to the following PDE
\begin{equation*}
\frac{dX^0(t)}{dt}=A(t,X^0(t),\mathcal{L}_{X^0(t)}),~X^0(0)=x.
\end{equation*}
Now we define the following skeleton equation
\begin{equation}\label{eq05}
\frac{d X^{\phi}(t)}{dt}=A(t,X^{\phi}(t),\mathcal{L}_{X^{0}(t)})+B(t,X^{\phi}(t),\mathcal{L}_{X^{0}(t)})\phi(t),~X^{\phi}(0)=x,
\end{equation}
where $\phi$ belongs to the space of square integrable functions.  We remark that $\mathcal{L}_{X^0(t)}$ (rather than  $\mathcal{L}_{X^{\phi}(t)}$) appears in the coefficients of (\ref{eq05}) and will be utilized to construct the rate function of LDP, which is crucial in the McKean-Vlasov case  and coincides also with the intuition since the distribution $\mathcal{L}_{X^{\varepsilon}(t)}$ of $X^{\varepsilon}(t)$ is not random so that the corresponding convergence is independent of the occurrence of a rare event w.r.t.~the random variable $X^{\varepsilon}(t)$. It should be pointed out that because the classical localization argument is unsuitable in the McKean-Vlasov case, some new and nontrivial analysis are needed here compared to the previous works on LDP for classical SPDEs (cf.~e.g.~\cite{L1,LTZ,RZ2,XZ}) and also for  MVS(P)DEs (cf.~\cite{ADRST,DST,HLL3,LSZZ,SY}).
As a consequence, we can directly obtain  the LDP for a large class of MVSPDE models.

The rest of manuscript is organized as follows. In Section \ref{sec3}, we study the strong and weak existence and uniqueness of solutions to MVSDEs (\ref{eq01}). In Section \ref{sec4}, we devote to investigating the existence and uniqueness of variational solutions to MVSPDEs (\ref{eq03}).
In Section \ref{sec5}, we establish the Freidlin-Wentzell type  LDP for MVSPDEs within the framework of Section \ref{sec4}.  In Section \ref{sec6} and \ref{example1} we  present many concrete examples to illustrate the wide applicability of our main results.  And we recall some lemmas in the Appendix for the reader's convenience.

Throughout this paper $C_{p}$  denotes a positive constant which may changes from line to line, where the subscript $p$ is used to emphasize that the constant depends on certain parameter.

\section{Well-posedness of McKean-Vlasov SDEs}\label{sec3}
 We first introduce some notations.
Let $|\cdot|$ and $\langle\cdot,\cdot\rangle$ be the Euclidean vector norm and scalar product, respectively. Let $\mathcal{P}(\mathbb{R}^d)$ represent the space of all probability measures on $\mathbb{R}^d$ equipped with the weak topology. Furthermore, for $p>0$, we set
$$\mathcal{P}_p(\mathbb{R}^d):=\Big\{\mu\in\mathcal{P}(\mathbb{R}^d):\mu(|\cdot|^p):=\int_{\mathbb{R}^d}|x|^p\mu(dx)<\infty\Big\}.$$
Then $\mathcal{P}_p(\mathbb{R}^d)$ is a Polish space under the $L^p$-Wasserstein distance
$$\mathbb{W}_{p}(\mu,\nu):=\inf_{\pi\in\mathfrak{C}(\mu,\nu)}\Big(\int_{\mathbb{R}^d\times \mathbb{R}^d}|x-y|^p\pi(dx,dy)\Big)^{\frac{1}{p\vee 1}},~\mu,\nu\in\mathcal{P}_p(\mathbb{R}^d),$$
where $\mathfrak{C}(\mu,\nu)$ stands for the set of all couplings for  $\mu$ and $\nu$. We also introduce a weighted variation norm,
$$\|\mu-\nu\|_{2,var}:=\sup\limits_{f\leq 1+|\cdot|^2}|\mu(f)-\nu(f)|,\;\;\;\mu,\nu\in\mathcal{P}_2(\mathbb{R}^d).$$

For any $T>0$,  let $\mathcal{C}_{T}:=C([0,T];\mathbb{R}^d)$ be the  space of all continuous function from $[0,T]$ to $\mathbb{R}^d$ equipped with the uniform norm,  i.e.
$$\|\xi\|_{T}:=\sup_{t\in[0,T]}|\xi(t)|.$$
Analogously, we define by $\mathcal{P}_T$ the space of all probability measures on $\mathcal{C}_{T}$ with the weak topology. Then
\begin{equation}\label{P1}
\mathcal{P}_{p,T}:=\Big\{\mu\in\mathcal{P}_T:\int_{\mathcal{C}_{T}}\|\xi\|_{T}^p\mu(d\xi)<\infty\Big\}
\end{equation}
is a Polish probability space under associated $L^p$-Wasserstein distance
\begin{equation}\label{P2}
\mathbb{W}_{p,T}(\mu,\nu):=\inf_{\pi\in\mathfrak{C}(\mu,\nu)}\Big(\int_{\mathcal{C}_{T}\times \mathcal{C}_{T}}\|\xi-\eta\|_{T}^p\pi(d\xi,d\eta)\Big)^{\frac{1}{p\vee1}},~~\mu,\nu\in\mathcal{P}_{p,T}.
\end{equation}
For any $t\in[0,T],~R>0,~\xi\in \mathcal{C}_{T}$, we define $\xi_t:[0,T]\to\mathbb{R}^d$ by
$$\xi_t(s):=\xi(t\wedge s),~s\in[0,T],$$
and map $\pi_t(\xi):=\xi_t$. Then the marginal distribution before time $t$ of a probability measure $\mu\in\mathcal{P}_T$ is denoted by
$$\mu_t:=\mu\circ \pi_t^{-1}.$$
Define
\begin{equation*}
\tau_R^\xi:=\inf\Big\{t\geq 0:|\xi(t)|\geq R\Big\}.
\end{equation*}
Then for any $\mu,\nu\in\mathcal{P}_{2,T}$, we can define the following ``local'' $L^2$-Wasserstein distance
\begin{equation*}
\mathbb{W}_{2,T,R}(\mu,\nu):=\inf_{\pi\in\mathfrak{C}(\mu,\nu)}\Big(\int_{\mathcal{C}_{T}\times \mathcal{C}_{T}}\|\xi_{T\wedge\tau_R^\xi\wedge\tau_R^{\eta}}-\eta_{T\wedge\tau_R^\xi\wedge\tau_R^{\eta}}\|_{T}^2\pi(d\xi,d\eta)\Big)^{\frac{1}{2}}.
\end{equation*}

\subsection{Main results}\label{sec2.2}
Let us consider the following MVSDE
\begin{equation}\label{eq1}
dX(t)=b(t,X(t),\mathcal{L}_{X(t)})dt+\sigma(t,X(t),\mathcal{L}_{X(t)})dW(t),
\end{equation}
where $W(t)$ is a standard $m$-dimensional Wiener process defined on the complete probability space $(\Omega,\mathcal{F},\{\mathcal{F}(t)\}_{t\geq 0},\mathbb{P})$ with the filtration $\{\mathcal{F}(t)\}_{t\geq 0}$ satisfying  the usual condition, and $$b:[0,\infty)\times\mathbb{R}^d\times\mathcal{P}(\mathbb{R}^d)\to \mathbb{R}^d,~ \sigma:[0,\infty)\times\mathbb{R}^d\times\mathcal{P}(\mathbb{R}^d)\to \mathbb{R}^d\otimes \mathbb{R}^m$$
are measurable maps.

\begin{definition}\label{de1}
(i)  An adapted continuous process on $\mathbb{R}^d$ is called a strong solution of (\ref{eq1}), if
$$\mathbb{E}\int_0^t\big\{|b(s,X(s),\mathcal{L}_{X(s)})|+\|\sigma(s,X(s),\mathcal{L}_{X(s)})\|^2\big\}ds<\infty,~t\geq 0,$$
and $\mathbb{P}$-a.s.
$$X(t)=X(0)+\int_0^tb(s,X(s),\mathcal{L}_{X(s)})ds+\int_0^t\sigma(s,X(s),\mathcal{L}_{X(s)})dW(s),~t\geq 0.$$

(ii) A pair $(\tilde{X}(t),\tilde{W}(t))_{t\geq 0}$ is called a weak solution to (\ref{eq1}), if there exists an $\mathbb{R}^m$-valued Wiener process $\tilde{W}(t)$ under the stochastic basis $(\tilde{\Omega},\tilde{\mathcal{F}},\{\tilde{\mathcal{F}}(t)\}_{t\geq0},\tilde{\mathbb{P}})$ such that $(\tilde{X}(t),\tilde{W}(t))$ solves (\ref{eq1}).~We say (\ref{eq1}) is weakly unique, if $(\tilde{X}(t),\tilde{W}(t))$ under the stochastic basis $(\tilde{\Omega},\tilde{\mathcal{F}},\{\tilde{\mathcal{F}}(t)\}_{t\geq0},\tilde{\mathbb{P}})$ and $(\bar{X}(t),\bar{W}(t))$ under $(\bar{\Omega},\bar{\mathcal{F}},\{\bar{\mathcal{F}}(t)\}_{t\geq0},\bar{\mathbb{P}})$ are two weak solutions to (\ref{eq1}), then $\mathcal{L}_{\tilde{X}(0)}|_{\tilde{\mathbb{P}}}=\mathcal{L}_{\bar{X}(0)}|_{\bar{\mathbb{P}}}$ implies that $\mathcal{L}_{\tilde{X}(t)}|_{\tilde{\mathbb{P}}}=\mathcal{L}_{\bar{X}(t)}|_{\bar{\mathbb{P}}}$.

\end{definition}

Assume that there exists $\kappa\geq 2$ such that the following conditions hold.

\begin{enumerate}

\item [$({\mathbf{A}}{\mathbf{1}})$]\label{A1}$($Continuity$)$ For $t\geq0$, $b(t,\cdot,\cdot),\sigma(t,\cdot,\cdot)$ is continuous on $\mathbb{R}^d\times \mathcal{P}_2(\mathbb{R}^d)$.

\item [$({\mathbf{A}}{\mathbf{2}})$]\label{A2}$($Local Weak Monotonicity$)$ There exists a constant $C>0$ such that for any $t\geq0$, $R>0$, $|x|\vee|y|\leq R$ and $\mu,\nu\in\mathcal{P}_{\kappa}(\mathbb{R}^d)$,
\begin{align*}
&2\langle b(t,x,\mu)-b(t,y,\nu),x-y\rangle+\|\sigma(t,x,\mu)-\sigma(t,y,\nu)\|^2\nonumber\\
\leq&\big(K_t(R)+C\mu(|\cdot|^{\kappa})+C\nu(|\cdot|^{\kappa})\big)\big(|x-y|^2+\|\mu-\nu\|_{2,var}\big),
\end{align*}
where $ K:[0,\infty)\times [0,\infty)\rightarrow \mathbb{R}_{+}$ satisfying for all $T,R\geq0$,
\begin{equation*}
\int_0^TK_t(R)dt<\infty.
\end{equation*}

\item [$({\mathbf{A}}{\mathbf{3}})$]\label{A3}$($Weak Coercivity$)$ For any $t\geq0$, $x\in\mathbb{R}^d$ and $\mu\in\mathcal{P}_2(\mathbb{R}^d)$,
\begin{equation}\label{a3}
\langle b(t,x,\mu),x\rangle\leq K_t(1)\big(1+|x|^2+\mu(|\cdot|^2)\big),
\end{equation}
where $K_t(1)$ is defined as in $(\mathbf{A2})$.

\item [$({\mathbf{A}}{\mathbf{4}})$]\label{A4}$($Growth$)$ For any $t\geq0$,  $x\in\mathbb{R}^d$ and $\mu\in\mathcal{P}_\kappa(\mathbb{R}^d)$,
\begin{eqnarray}
\!\!\!\!\!\!\!\!&&|b(t,x,\mu)|^2\leq K_t(1)\big(1+|x|^{\kappa}+\mu(|\cdot|^\kappa)\big),
\nonumber\\
\!\!\!\!\!\!\!\!&&
\|\sigma(t,x,\mu)\|^2\leq K_t(1)\big(1+|x|^2+\mu(|\cdot|^2)\big).\label{gro}
\end{eqnarray}
\end{enumerate}

The main result of this part is stated as follows.
\begin{theorem}\label{th1}
Assume that $(\mathbf{A1})$-$(\mathbf{A4})$ hold with
\begin{equation}\label{esK}
\sup_{t\in[0,T]}K_{t}(1)<\infty.
\end{equation}
Then for any $X(0)\in L^r(\Omega;\mathbb{R}^d)$ with $r>\kappa$,  MVSDE (\ref{eq1}) has a strong solution in the sense of Definition \ref{de1}. Moreover,
\begin{equation}\label{es2}
\mathbb{E}\Big[\sup_{t\in[0,T]}|X(t)|^{r}\Big]<\infty.
\end{equation}
Furthermore, if one of the following conditions holds:

\vspace{1mm}
\begin{enumerate}

\item [$(\mathbf{A2}')$]\label{A5}  There exist $C_0>0,\varepsilon>0$ such that for any $T,R>0$, $\xi,\eta\in \mathcal{C}_T$, $\mu,\nu\in \mathcal{P}_{2,T}$  and $t\in[0,T\wedge\tau_R^{\xi}\wedge\tau_R^{\eta}]$,
\begin{eqnarray*}
\!\!\!\!\!\!\!\!&&\langle b(t,\xi(t),\mu(t))-b(t,\eta(t),\nu(t)),\xi(t)-\eta(t)\rangle\nonumber\\
\leq~~&&\!\!\!\!\!\!\!\!
 C_R\Big\{|\xi(t)-\eta(t)|^2+\mathbb{W}_{2,T,R}(\mu_t,\nu_t)^2+C_0{\rm{e}}^{-\varepsilon C_R}(1\wedge\mathbb{W}_{2}(\mu(t),\nu(t))^2)\Big\},
\end{eqnarray*}
and
\begin{eqnarray*}
\!\!\!\!\!\!\!\!&&\|\sigma(t,\xi(t),\mu(t))-\sigma(t,\eta(t),\nu(t))\|^2
\nonumber\\
\leq~~&&\!\!\!\!\!\!\!\!
 C_R\Big\{|\xi(t)-\eta(t)|^2+\mathbb{W}_{2,T,R}(\mu_t,\nu_t)^2+C_0{\rm{e}}^{-\varepsilon C_R}(1\wedge\mathbb{W}_{2}(\mu(t),\nu(t))^2)\Big\}.
\end{eqnarray*}

\item [$(\mathbf{A2}'')$]\label{A6} There exists $C>0$ such that for any $t\geq0$, $R>0$, $|x|\vee|y|\leq R$ and $\mu,\nu\in\mathcal{P}_{\kappa}(\mathbb{R}^d)$ with $\kappa$ defined in $(\mathbf{A2})$,
\begin{align}\label{a4}
&2\langle b(t,x,\mu)-b(t,y,\nu),x-y\rangle+\|\sigma(t,x,\mu)-\sigma(t,y,\nu)\|^2\nonumber\\
\leq&\big(K_t(R)+C\mu(|\cdot|^{\kappa})+C\nu(|\cdot|^{\kappa})\big)\big(|x-y|^2+\mathbb{W}_2(\mu,\nu)^2\big).
\end{align}
Moreover, there exists a non-decreasing function $f:\mathbb{R}\to\mathbb{R}$  such that  for any  $\mathbb{E} {\rm{e}}^{f(|X(0)|)}<\infty$,
\begin{eqnarray}
\!\!\!\!\!\!\!\!&&\sup_{t\in[0,T]}\mathbb{E} {\rm{e}}^{f(|X(t)|)}<\infty,\label{eqe1}
\\
\!\!\!\!\!\!\!\!&&\lim\limits_{R\to \infty}\Big(f(R)-\int_0^TK_t(R)dt\Big)=\infty.\label{eqe2}
\end{eqnarray}

\item [$(\mathbf{A2}''')$]\label{A7} There exists $C>0$ such that for any $t\geq 0$, $\mu,\nu\in\mathcal{P}_{\kappa}(\mathbb{R}^d)$ and any coupling $\pi\in\mathfrak{C}(\mu,\nu)$,
\begin{eqnarray*}
\!\!\!\!\!\!\!\!&&\int_{\mathbb{R}^d\times\mathbb{R}^d}2\langle b(t,x,\mu)-b(t,y,\nu),x-y\rangle+\|\sigma(t,x,\mu)-\sigma(t,y,\nu)\|^2 \pi(dx,dy)
\nonumber\\
\leq~~\!\!\!\!\!\!\!\!&&
 C\int_{\mathbb{R}^d\times\mathbb{R}^d}\big(1+\mu(|\cdot|^{\kappa})+\nu(|\cdot|^{\kappa})\big)|x-y|^2\pi(dx,dy).
\end{eqnarray*}
\end{enumerate}
Then MVSDE (\ref{eq1}) has a unique strong/weak solution in the sense of Definition \ref{de1} provided satisfying (\ref{es2}).
\end{theorem}
\begin{remark}\label{re1}
It should be pointed out that in order to guarantee the uniqueness of solutions to (\ref{eq1}),  some additional assumption is necessary. In fact, we consider the following MVSDEs
\begin{equation}\label{eq53}
X(t)=X(0)+\int_{0}^{t}B(X(s),\mathbb{E}[\bar{b}(X(s))])ds+\int_{0}^{t}\Sigma(X(s),\mathbb{E}[\bar{\sigma}(X(s))])dW(s).
\end{equation}
Scheutzow \cite{SCHEUTZOW} has proved that when $\Sigma=0$ and either of functions $B$ or $\bar{b}$ is merely locally Lipschitz continuous, the uniqueness of solutions to (\ref{eq53}), in general, does not hold. Therefore, it is reasonable  to impose some extra structure (e.g.~$(\mathbf{A2}')$-$(\mathbf{A2}''')$) on the coefficients to obtain a unique solution.
\end{remark}
\begin{remark}\label{remark2.2}
We shall give some comments for the above assumptions.

\vspace{1mm}
(i) Note that by \cite[Theorem 6.15]{villani}, there exists a constant $c > 0$ such that for $p>0$,
$$ \|\mu-\nu\|_{var}+\mathbb{W}_p(\mu,\nu)^{1\vee p}\leq c\|\mu-\nu\|_{p,var},\;\;\;\mu,\nu\in \mathcal{P}_{p}(\mathbb{R}^d).$$
Therefore, the condition $(\mathbf{A2}'')$  is stronger than $(\mathbf{A2})$, whereas the conditions $(\mathbf{A2}')$ and $(\mathbf{A2}''')$ are not comparable  to $(\mathbf{A2})$.

(ii) Unlike the classical SDE case, the main challenge in applying local assumptions to the McKean-Vlasov equations is that the measure dependence is inherently non-local. In  {$(\mathbf{A2}')$},  we show that the dependence on the distribution of the coefficients is asymptotically determined by the distribution of the path. On the other hand, we exploit sufficient moment control in {$(\mathbf{A2}'')$} which is also a natural remedy. As for {$(\mathbf{A2}''')$}, we utilize the integrated condition which is inspired by \cite{HSS}.

(iii) As stated in the introduction, it is easy to check that the following example ($p\geq 1$)
\begin{equation}\label{example1.1}
dX(t)=-X(t)^3\Bigg[\int_{\mathbb{R}^d}|y|^p\mathcal{L}_{X(t)}(dy)\Bigg] dt+ X(t)dW(t)
\end{equation}
 satisfies $(\mathbf{A1})$-$(\mathbf{A4})$ and $(\mathbf{A2}''')$.
  Therefore, by Theorem \ref{th1}, (\ref{example1.1}) admits a unique strong/weak solution. For more concrete examples or applications about the conditions above, one can see Remarks \ref{re2}-\ref{re3} in Section \ref{sec6}.

(iv) In contrast to the existing results (cf.~e.g.~\cite{E1,GHM,Hu,RTW}),  we here assume the locally monotone conditions rather than locally Lipschitz (see for example, \cite[Proposition 3.29]{GHM} or \cite[Proposition 2.1]{RTW}), which extend the classical result by Krylov and Rozovskii \cite{KR} to the McKean-Vlasov case and then help us to establish the corresponding result also in infinite dimensional case, see Theorem \ref{th2} below for details.
\end{remark}

In the sequel,  we first construct the Euler type approximation and prove that these approximating equations are well-posed and the distributions of approximating solutions $\mathcal{L}_{X^{(n)}}$  is  tight, which allows us to find a limit of subsequence of $\mathcal{L}_{X^{(n)}}$ and implies the existence of strong solution by taking advantage of the modified Yamada-Watanabe theorem (see Lemma \ref{lem2}).
Combining this with the pathwise uniqueness of (\ref{eq1}), we can get the strong (weak) well-posedness.

\subsection{Construction of approximating equations}
For any $T\geq0$ and integer $n\geq1$, let $T_n=\frac{T}{n}$, $t_{k}^n=kT_n$ for $k=0,1,\ldots,n$. For any $t\in[0,t_{1}^n]$, let us consider the following SDE
\begin{equation}\label{eq2}
dX^{(n)}(t)=b(t,X^{(n)}(t),\mu^{(n)}(0))dt+\sigma(t,X^{(n)}(t),\mu^{(n)}(0))dW(t),~X^{(n)}(0)=X(0),
\end{equation}
where $\mu^{(n)}(0):=\mathcal{L}_{X^{(n)}(0)}$.

Note that (\ref{eq2}) is a classical SDE (not distribution dependent), thus in view of the conditions $(\mathbf{A1})$-$(\mathbf{A4})$, we see the coefficients in (\ref{eq2}) satisfy the conditions of \cite[Theorem 3.1.1]{LR1}, thus (\ref{eq2}) admits a unique solution on the time interval $[0,t_{1}^n]$. Moreover, by  $(\mathbf{A3})$  we can show that for any $r>\kappa$,
\begin{equation*}
\mathbb{E}\Big[\sup_{t\in[0,t_{1}^n]}|X^{(n)}(t)|^{r}\Big]\leq C_{r,T}(1+\mathbb{E}|X(0)|^{r}),
\end{equation*}
whose proof is  similar to that of Lemma \ref{pro1} below, we omit the details here.

Inductively, for any $k=0,1,\ldots,n-1$ and $t\in(t_{k}^n,t_{k+1}^n]$, we consider
\begin{equation}\label{eq4}
dX^{(n)}(t)=b(t,X^{(n)}(t),\mu^{(n)}(t_{k}^n))dt+\sigma(t,X^{(n)}(t),\mu^{(n)}(t_{k}^n))dW(t)
\end{equation}
with initial value $X^{(n)}(t_{k}^n)$, where $\mu^{(n)}(t_{k}^n):=\mathcal{L}_{X^{(n)}(t_{k}^n)}$. Analogously, (\ref{eq4}) admits a unique solution with the estimate
\begin{equation*}
\mathbb{E}\Big[\sup_{t\in[t_{k}^n,t_{k+1}^n]}|X^{(n)}(t)|^{r}\Big]\leq C_r(1+\mathbb{E}|X^{(n)}(t_{k}^n)|^{r}).
\end{equation*}
Let $\chi_n(t):=t^n_k$ for any $t\in( t^n_k,t^n_{k+1}]$, $k=0,1,\ldots,n-1$.  Then for any $t\in[0,T]$, we introduce the following approximating equation
\begin{equation}\label{eqapp}
dX^{(n)}(t)=b(t,X^{(n)}(t),\mu^{(n)}(\chi_n(t)))dt+\sigma(t,X^{(n)}(t),\mu^{(n)}(\chi_n(t)))dW(t),
\end{equation}
with initial value $X^{(n)}(0)=X(0)$,
which has a unique solution satisfying the following preliminary estimate
\begin{equation}\label{es5}
\mathbb{E}\Big[\sup_{t\in[0,T]}|X^{(n)}(t)|^{r}\Big]\leq\sum_{k=0}^{n-1}\mathbb{E}\Big[\sup_{t\in[t^n_k,t^n_{k+1}]}|X^{(n)}(t)|^{r}\Big]\leq C(n)<\infty.
\end{equation}

The following two lemmas illustrate the uniform estimate and time H\"{o}lder continuity of the approximating solution $X^{(n)}(t)$ to (\ref{eqapp}), which play a significant role in proving the tightness of $X^{(n)}(t)$.
\begin{lemma}\label{pro1}
Assume that the conditions $(\mathbf{A1})$-$(\mathbf{A4})$  hold. For any $T>0$, $X(0)\in L^r(\Omega;\mathbb{R}^d)$ with $r>\kappa$,  there is $C_{r,T}>0$  such that
\begin{equation}\label{es6}
\sup_{n\geq 1}\mathbb{E}\Big[\sup_{t\in[0,T]}|X^{(n)}(t)|^r\Big]\leq C_{r,T}.
\end{equation}
\end{lemma}
\begin{proof}
Applying It\^{o}'s formula, we have
\begin{eqnarray*}
|X^{(n)}(t)|^r=~~\!\!\!\!\!\!\!\!&&|X(0)|^r+\frac{r(r-2)}{2}   \int_0^t|X^{(n)}(s)|^{r-4}|\sigma(s,X^{(n)}(s),\mu^{(n)}(\chi_n(s)))^*X^{(n)}(s)|^2ds
\nonumber \\
 \!\!\!\!\!\!\!\!&&+\frac{r}{2}\int_0^t|X^{(n)}(s)|^{r-2}\big(2\langle b(s,X^{(n)}(s),\mu^{(n)}(\chi_n(s))),X^{(n)}(s)\rangle
 \nonumber \\
 \!\!\!\!\!\!\!\!&&
 ~~~~~~~~~~~~~~~~~~~~~~~~~~~~+\|\sigma(s,X^{(n)}(s),\mu^{(n)}(\chi_n(s)))\|^2\big)ds
\nonumber \\
 \!\!\!\!\!\!\!\!&&+r\int_0^t|X^{(n)}(s)|^{r-2}\langle X^{(n)}(s),\sigma(s,X^{(n)}(s),\mu^{(n)}(\chi_n(s)))dW(s)\rangle
\nonumber \\
 \leq~~\!\!\!\!\!\!\!\!&&|X(0)|^r+C_r\int_0^tK_s(1)|X^{(n)}(s)|^{r-2}(1+|X^{(n)}(s)|^2+\mathbb{E}|X^{(n)}(\chi_n(s))|^2)ds
 \nonumber \\
 \!\!\!\!\!\!\!\!&&+r\int_0^t|X^{(n)}(s)|^{r-2}\langle X^{(n)}(s),\sigma(s,X^{(n)}(s),\mu^{(n)}(\chi_n(s)))dW(s)\rangle
\nonumber \\
 \leq~~\!\!\!\!\!\!\!\!&&|X(0)|^r+C_r\int_0^tK_s(1)(1+|X^{(n)}(s)|^r+\mathbb{E}|X^{(n)}(\chi_n(s))|^r)ds
 \nonumber \\
 \!\!\!\!\!\!\!\!&&+r\int_0^t|X^{(n)}(s)|^{r-2}\langle X^{(n)}(s),\sigma(s,X^{(n)}(s),\mu^{(n)}(\chi_n(s)))dW(s)\rangle,
\end{eqnarray*}
where we used  (\ref{a3}) and (\ref{gro}) in the first inequality and Jensen's inequality, Young's inequality in the last step.

Due to (\ref{es5}), by Burkholder-Davis-Gundy's inequality and Jensen's inequality, it holds that
\begin{eqnarray*}
\mathbb{E}\Big[\sup_{t\in[0,T]}|X^{(n)}(t)|^r\Big]\leq~~\!\!\!\!\!\!\!\!&&C_r\Big(\mathbb{E}|X(0)|^r+\int_0^TK_s(1)ds\Big)+C_r\int_0^TK_s(1)\mathbb{E}\Big[\sup_{u\in[0,s]}|X^{(n)}(u)|^r\Big]ds
\nonumber \\
 \!\!\!\!\!\!\!\!&&+\frac{1}{2}\mathbb{E}\Big[\sup_{t\in[0,T]}|X^{(n)}(t)|^r\Big],
\end{eqnarray*}
which implies the uniform estimate (\ref{es6}) by using Gronwall's lemma.
\end{proof}

\begin{remark}
We remark that (\ref{es5}) is an important argument that allows to not use stopping times to localize the
process (because (\ref{es5}) guarantees two things: the expecations are finite and the local
martingale is a real martingale), which is a crucial point in the study of McKean-Vlasov equations.
\end{remark}

\begin{lemma}\label{pro2}
Assume that the conditions $(\mathbf{A1})$-$(\mathbf{A4})$ and (\ref{esK}) hold. For any  $X(0)\in L^r(\Omega;\mathbb{R}^d)$ with $r>\kappa$, there is $C_{r,T}>0$ independent of $n$ such that for any $t,s\in[0,T]$ and $2\leq q\leq \frac{2r}{\kappa}$,
\begin{equation*}
\sup_{n\geq 1}\mathbb{E}|X^{(n)}(t)-X^{(n)}(s)|^{q}\leq C_{r,T}|t-s|^{\frac{q}{2}}.
\end{equation*}
\end{lemma}
\begin{proof}
By $(\mathbf{A4})$,  there is $C_{r,T}>0$ such that for any $t,s\in[0,T]$,
\begin{eqnarray*}
\!\!\!\!\!\!\!\!&&\mathbb{E}|X^{(n)}(t)-X^{(n)}(s)|^q\nonumber \\
\leq~~\!\!\!\!\!\!\!\!&& C_T|t-s|^{q-1}\mathbb{E}\int_{s}^{t}|b(u,X^{(n)}(u),\mu^{(n)}(\chi_n(u)))|^{q} du
\nonumber \\
\!\!\!\!\!\!\!\!&&
+C_T\mathbb{E}\Big|\int_{s}^{t}\sigma(u,X^{(n)}(u),\mu^{(n)}(\chi_n(u)))dW(u)\Big|^q
\nonumber \\
\leq~~  \!\!\!\!\!\!\!\!&&C_T|t-s|^{q-1}\int_{s}^{t}\Big(1+\mathbb{E}\big(\sup_{u\in[0,T]}|X^{(n)}(u)|^{\frac{\kappa q}{2}}\big)+\mathbb{E}\big(\sup_{u\in[0,T]}|X^{(n)}(\chi_n(u))|^{\frac{\kappa q}{2}}\big)\Big)du
\nonumber \\
\!\!\!\!\!\!\!\!&&+C_T|t-s|^{\frac{q-2}{2}}\int_{s}^{t}\Big(1+\mathbb{E}\big(\sup_{u\in[0,T]}
|X^{(n)}(u)|^{q}\big)+\mathbb{E}\big(\sup_{u\in[0,T]}|X^{(n)}(\chi_n(u))|^{q}\big)\Big)du
\nonumber \\
\leq~~\!\!\!\!\!\!\!\!&&C_{r,T}|t-s|^{\frac{q}{2}},
\end{eqnarray*}
where we used (\ref{esK}) in the second inequality and (\ref{es6}) in the last step.
\end{proof}

\subsection{Existence of strong solutions}

We now prove the existence of (strong) solutions to (\ref{eq1}) by the following proposition.

  \begin{proposition}\label{pro3}
Assume that  $(\mathbf{A1})$-$(\mathbf{A4})$ and (\ref{esK}) hold. For any $X(0)\in L^r(\Omega;\mathbb{R}^d)$ with $r>\kappa$, there exists a strong solution $\{X(t)\} _{t\in[0,T]}$ to (\ref{eq1}) in the sense of Definition \ref{de1}. Moreover,
\begin{equation}\label{es19}
\mathbb{E}\Big[\sup_{t\in[0,T]}|X(t)|^{r}\Big]<\infty.
\end{equation}
\end{proposition}

  \begin{proof}
In view of Lemmas \ref{pro1} and \ref{pro2}, we can prove that $(X^{(n)})_{n\geq 1}$ satisfies the conditions (i) and (ii) of Lemma \ref{lem1} in Appendix,
thus $(X^{(n)})_{n\geq 1}$ is tight. Furthermore, using Ulam's tightness theorem and the fact that the product of two compact sets is compact yields that
 $$(X^{(n)},W)_{n\geq 1}~\text{is tight in}~C([0,T];\mathbb{R}^d)\times C([0,T];\mathbb{R}^m).$$

Therefore, applying the Skorokhod representation theorem (see Lemma \ref{lem22} in Appendix), we can deduce that there exists a subsequence of $(X^{(n)},W)_{n\geq 1}$, denoted again by $(X^{(n)},W)_{n\geq 1}$, a probability space $(\tilde{\Omega},\tilde{\mathcal{F}},\tilde{\mathbb{P}})$, and, on this space, $C([0,T];\mathbb{R}^d)\times C([0,T];\mathbb{R}^m)$-valued random variables $(\tilde{X}^{(n)},\tilde{W}^{(n)})$ coincide in law with $(X^{(n)},W)$   in $(\Omega,\mathcal{F},\mathbb{P})$ such that
\begin{eqnarray}
\!\!\!\!\!\!\!\!&&\tilde{X}^{(n)}\to \tilde{X}~\text{in}~C([0,T];\mathbb{R}^d),~\tilde{\mathbb{P}}\text{-a.s.},~\text{as}~n\to\infty,\label{es8}
 \\
\!\!\!\!\!\!\!\!&&
\tilde{W}^{(n)}\to \tilde{W}~\text{in}~C([0,T];\mathbb{R}^m),~\tilde{\mathbb{P}}\text{-a.s.},~\text{as}~n\to\infty.\nonumber
\end{eqnarray}

Let $\tilde{\mathcal{F}}^{(n)}(t)$ (resp.~$\tilde{\mathcal{F}}(t)$) be  the completion of the $\sigma$-algebra generated by $\{\tilde{X}^{(n)}(s),\tilde{W}^{(n)}(s):s\leq t\}$ (resp.~$\{\tilde{X}(s),\tilde{W}(s):s\leq t\}$), then it is easy to see that $\tilde{X}^{(n)}(t)$ is $\tilde{\mathcal{F}}^{(n)}(t)$-adapted and continuous, $\tilde{W}^{(n)}(t)$ is a standard $m$-dimensional Wiener process on $(\tilde{\Omega},\tilde{\mathcal{F}},\{\tilde{\mathcal{F}}^{(n)}(t)\}_{t\in[0,T]},\tilde{\mathbb{P}})$.

Note that, by the uniform estimate (\ref{es6}) and the equivalence of the laws of $\tilde{X}^{(n)}$ and $X^{(n)}$, one can also obtain that for any $T>0$ and $r>\kappa$,  there is a constant $C_{r,T}>0$ such that
\begin{equation}\label{es7}
\sup_{n\geq 1}\tilde{\mathbb{E}}\Big[\sup_{t\in[0,T]}|\tilde{X}^{(n)}(t)|^r\Big]\leq C_{r,T}.
\end{equation}
 Taking into account the convergence (\ref{es8}) and the uniform integrability (\ref{es7}), by Vitali's convergence theorem (see Lemma \ref{lem9} in Appendix), we have
\begin{equation}\label{es9}
\tilde{\mathbb{E}}\Big[\sup_{t\in[0,T]}|\tilde{X}^{(n)}(t)-\tilde{X}(t)|^{2}\Big]\to 0,~\text{as}~n\to\infty.
\end{equation}
Moreover, by (\ref{es8}), (\ref{es7}) and Fatou's lemma, it is clear that
 \begin{equation}\label{es10}
\tilde{\mathbb{E}}\Big[\sup_{t\in[0,T]}|\tilde{X}(t)|^{r}\Big]\leq C_{r,T}.
\end{equation}

Let the process $\tilde{M}^{(n)}(t)$ with trajectories in $C([0,T];\mathbb{R}^d)$ be
$$\tilde{M}^{(n)}(t)=\tilde{X}^{(n)}(t)-\tilde{X}^{(n)}(0)-\int_0^tb(s,\tilde{X}^{(n)}(s),\tilde{\mu}^{(n)}(\chi_n(s)))ds,$$
where we denote $\tilde{\mu}^{(n)}(t)=\mathcal{L}_{\tilde{X}^{(n)}(t)}$, which is a square integrable martingale w.r.t.~the filtration $\tilde{\mathcal{F}}^{(n)}(t)$ with quadratic variation
$$\langle \tilde{M}^{(n)}\rangle_t=\int_0^t \sigma\big(s,\tilde{X}^{(n)}(s),\tilde{\mu}^{(n)}(\chi_n(s))\big)\sigma^*\big(s,\tilde{X}^{(n)}(s),\tilde{\mu}^{(n)}(\chi_n(s))\big)ds.$$
Indeed, for all $s\leq t\in [0,T]$ and bounded continuous functions $\Phi(\cdot)$ on $C([0,T];\mathbb{R}^d)\times C([0,T];\mathbb{R}^m)$, since $\tilde{X}^{(n)}$ and $X^{(n)}$ have the same distribution, it follows that
\begin{eqnarray}\label{mar1}
\tilde{\mathbb{E}}\Big[\big(\tilde{M}^{(n)}(t)-\tilde{M}^{(n)}(s)\big)\Phi((\tilde{X}^{(n)},\tilde{W}^{(n)})|_{[0,s]})\Big]=0
\end{eqnarray}
and
\begin{eqnarray}\label{mar2}
\!\!\!\!\!\!\!\!&&\tilde{\mathbb{E}}\Big[\Big(\tilde{M}^{(n)}(t)\otimes\tilde{M}^{(n)}(t)-\tilde{M}^{(n)}(s)\otimes\tilde{M}^{(n)}(s)
\nonumber \\
\!\!\!\!\!\!\!\!&&
-\int_s^t a\big(r,\tilde{X}^{(n)}(r),\tilde{\mu}^{(n)}(\chi_n(r))\big)dr\Big)\Phi((\tilde{X}^{(n)},\tilde{W}^{(n)})|_{[0,s]})\Big]=0,
\end{eqnarray}
where $a:=\sigma\sigma^*$.

In the sequel, we will take the limit in (\ref{mar1}) and (\ref{mar2}).
Denote $\tilde{\mu}(t)=\mathcal{L}_{\tilde{X}(t)}$. Note that by Lemma \ref{pro2} and the convergence (\ref{es9}),
\begin{eqnarray}\label{es13}
\lim_{n\to\infty}\sup_{s\in[0,T]}\mathbb{W}_2(\tilde{\mu}^{(n)}(\chi_n(s)),\tilde{\mu}(s))^2\leq~~ \!\!\!\!\!\!\!\!&&C\lim_{n\to\infty}\sup_{s\in[0,T]}\tilde{\mathbb{E}}|\tilde{X}^{(n)}(s)-\tilde{X}^{(n)}(\chi_n(s))|^2
\nonumber \\
\!\!\!\!\!\!\!\!&&
+C\lim_{n\to\infty}\sup_{s\in[0,T]}\tilde{\mathbb{E}}|\tilde{X}^{(n)}(s)-\tilde{X}(s)|^2
\nonumber \\
\leq~~\!\!\!\!\!\!\!\!&&C\lim_{n\to\infty}T_n
=0.
\end{eqnarray}
Hence, according to the continuity of $b,\sigma$ (see condition $(\mathbf{A1})$) and (\ref{es13}), it is obvious that for any $s\in[0,T]$,
\begin{eqnarray*}
\!\!\!\!\!\!\!\!&&|b(s,\tilde{X}^{(n)}(s),\tilde{\mu}^{(n)}(\chi_n(s)))-b(s,\tilde{X}(s),\tilde{\mu}(s))|\to0,~\tilde{\mathbb{P}}\text{-a.s.},~\text{as}~n\to\infty,
\nonumber \\
\!\!\!\!\!\!\!\!&&\|\sigma(s,\tilde{X}^{(n)}(s),\tilde{\mu}^{(n)}(\chi_n(s)))-\sigma(s,\tilde{X}(s),\tilde{\mu}(s))\|\to0,~\tilde{\mathbb{P}}\text{-a.s.},~\text{as}~n\to\infty.
\end{eqnarray*}
Then by (\ref{es8}), (\ref{es7}), (\ref{es10}), (\ref{es13}) and the dominated convergence theorem, one can infer that for any $t\in[0,T]$,
\begin{eqnarray}
\!\!\!\!\!\!\!\!&&\int_0^t|b(s,\tilde{X}^{(n)}(s),\tilde{\mu}^{(n)}(\chi_n(s)))-b(s,\tilde{X}(s),\tilde{\mu}(s))|^2ds\to 0,~\tilde{\mathbb{P}}\text{-a.s.},~\text{as}~n\to\infty,\label{es66}
 \\
\!\!\!\!\!\!\!\!&&\int_0^t\|\sigma(s,\tilde{X}^{(n)}(s),\tilde{\mu}^{(n)}(\chi_n(s)))-\sigma(s,\tilde{X}(s),\tilde{\mu}(s))\|^2ds\to 0,~\tilde{\mathbb{P}}\text{-a.s.},~\text{as}~n\to\infty.\label{es17}
\end{eqnarray}
Moreover, by condition $(\mathbf{A4})$, (\ref{es7}) and (\ref{es10}), we deduce that
\begin{eqnarray}\label{es18}
~~~&&\tilde{\mathbb{E}}\Big|\int_0^t\|
\sigma(s,\tilde{X}^{(n)}(s),\tilde{\mu}^{(n)}(\chi_n(s)))- \sigma(s,\tilde{X}(s),\tilde{\mu}(s))\|^2ds\Big|^{r'}
\nonumber \\
~~~&&\leq C_T\tilde{\mathbb{E}}\Big[\int_0^tK_s(1)\Big(1+|\tilde{X}^{(n)}(s)|^{2r'}+\tilde{\mathbb{E}}|\tilde{X}^{(n)}(\chi_n(s))|^{2r'}
\nonumber \\
~~~&&~~~~~~~~~~~~~~~~~~
+|\tilde{X}(s)|^{2r'}+\tilde{\mathbb{E}}|\tilde{X}(s)|^{2r'}\Big)ds\Big]
\nonumber \\
~~~&&\leq C_T,
\end{eqnarray}
and
\begin{eqnarray}\label{es65}
~~~&&\tilde{\mathbb{E}}\Big|\int_0^t|
b(s,\tilde{X}^{(n)}(s),\tilde{\mu}^{(n)}(\chi_n(s)))- b(s,\tilde{X}(s),\tilde{\mu}(s))|^2ds\Big|^{r'}
\nonumber \\
~~~&&\leq C_T\tilde{\mathbb{E}}\Big[\int_0^tK_s(1)\Big(1+|\tilde{X}^{(n)}(s)|^{\kappa r'}+\tilde{\mathbb{E}}|\tilde{X}^{(n)}(\chi_n(s))|^{\kappa r'}
\nonumber \\
~~~&&~~~~~~~~~~~~~~~~~~
+|\tilde{X}(s)|^{\kappa r'}+\tilde{\mathbb{E}}|\tilde{X}(s)|^{\kappa r'}\Big)ds\Big]
\nonumber \\
~~~&&\leq C_T,
\end{eqnarray}
where $1<r'<\frac{r}{\kappa}$.

Consequently, from (\ref{es66})-(\ref{es65}) and the Vitali's convergence theorem, it follows that
\begin{eqnarray*}
\!\!\!\!\!\!\!\!&&\lim_{n\to\infty}\tilde{\mathbb{E}}\Big[\int_0^t\|
\sigma(s,\tilde{X}^{(n)}(s),\tilde{\mu}^{(n)}(\chi_n(s)))- \sigma(s,\tilde{X}(s),\tilde{\mu}(s))\|^2ds\Big]
=0,
\nonumber \\
\!\!\!\!\!\!\!\!&&\lim_{n\to\infty}\tilde{\mathbb{E}}\Big[\int_0^t|
b(s,\tilde{X}^{(n)}(s),\tilde{\mu}^{(n)}(\chi_n(s)))- b(s,\tilde{X}(s),\tilde{\mu}(s))|^2ds\Big]
=0.
\end{eqnarray*}
Then we deduce that for all $s\leq t\in [0,T]$ and bounded continuous functions $\Phi(\cdot)$ on $C([0,T];\mathbb{R}^d)\times C([0,T];\mathbb{R}^m)$,
\begin{eqnarray}\label{mar3}
\tilde{\mathbb{E}}\Big[\big(\tilde{M}(t)-\tilde{M}(s)\big)\Phi((\tilde{X},\tilde{W})|_{[0,s]})\Big]=0
\end{eqnarray}
and
\begin{eqnarray}\label{mar4}
\!\!\!\!\!\!\!\!&&\tilde{\mathbb{E}}\Big[\Big(\tilde{M}(t)\otimes\tilde{M}(t)-\tilde{M}(s)\otimes\tilde{M}(s)
\nonumber \\
\!\!\!\!\!\!\!\!&&~~~
-\int_s^t \sigma\big(r,\tilde{X}(r),\tilde{\mu}(r)\big)\sigma^*\big(r,\tilde{X}(r),\tilde{\mu}(r)\big)dr\Big)\Phi((\tilde{X},\tilde{W})|_{[0,s]})\Big]=0,
\end{eqnarray}
where $\tilde{M}(t)$ is defined by
$$\tilde{M}(t):=\tilde{X}(t)-\tilde{X}(0)-\int_0^tb(s,\tilde{X}(s),\tilde{\mu}(s))ds.$$
Therefore, by (\ref{mar3}) and (\ref{mar4}), process $\tilde{M}(t)$ is a square integrable $\tilde{\mathcal{F}}(t)$-martingale in $\mathbb{R}^d$
with quadratic variation
$$\langle \tilde{M}\rangle_t=\int_0^t \sigma\big(s,\tilde{X}(s),\tilde{\mu}(s)\big)\sigma^*\big(s,\tilde{X}(s),\tilde{\mu}(s)\big)ds.$$
Consequently, we conclude that (\ref{eq1}) admits a weak solution by the martingale representation theorem.

Now let $\mu(t)=\mathcal{L}_{X(t)}$, we consider the decoupled SDE
\begin{equation}\label{eq7}
dX(t)=b^{\mu}(t,X(t))dt+\sigma^{\mu}(t,X(t))dW(t),
\end{equation}
where $b^{\mu}(t,X(t)):=b(t,X(t),\mu(t))$, $\sigma^{\mu}(t,X(t)):=\sigma(t,X(t),\mu(t))$. Under the assumptions $(\mathbf{A1})$-$(\mathbf{A4})$,  using the classical result \cite[Theorem 3.1.1]{LR1}, it is easy to show that (\ref{eq7}) has strong uniqueness. Then in view of the modified Yamada-Watanabe  theorem (see Lemma \ref{lem2}), the MVSDE (\ref{eq1}) has a strong solution.

Finally, the estimate (\ref{es19}) follows from (\ref{es10}).  Hence the proof is complete.
\end{proof}

\begin{remark}
We would like to give some comments on our method used in Proposition \ref{pro3} and explain the main differences compared with  \cite{E1,GHM,KR}.

 \vspace{1mm}
(i) As stated in Remark \ref{remark2.2} (iv), this work generalize  the classical result of Krylov and Rozovskii \cite{KR} (see also Section 3 in \cite{LR1}) to the McKean-Vlasov case, where the authors in \cite{KR} considered an Euler approximation of classical SDEs and then utilize the Picard iteration argument to prove the existence of solution. Compared with \cite{KR}, our strategy is to establish an Euler type approximation (only for distributions) of MVSDEs (\ref{eq1}), and then utilize the tightness argument and martingale approach to prove the existence of solution, which is quite different from \cite{KR}.

 \vspace{1mm}
(ii) The author in \cite{E1} use a Banach-Picard iteration to construct an approximating sequence $X^{[n]}$ of MVSDEs (\ref{eq1}) whereas we here utilize an Euler type scheme (i.e.~(\ref{eqapp})). Based on these, we investigate the tightness of $(X^{(n)},W)_{n}$ instead of $(X^{[n]},X^{[n+1]})_n$ in \cite{E1}, then we can obtain the convergence of $\tilde{\mu}^{(n)}=\mathcal{L}_{\tilde{X}^{(n)}}$  with direct computation here instead of Arzel\`{a}-Ascoli's theorem as in \cite{E1}. Finally, both papers use a generalized version of Yamada and Watanabe result (which is also along
the lines of the proof of Proposition 3.7 of \cite{GHM}).
Clearly, the main advantage of our method  is that we can work with more general assumptions on the coefficients.
\end{remark}

\subsection{Pathwise uniqueness}
In this subsection, we deal with the pathwise uniqueness of (\ref{eq1}). Then we are in the position to derive the strong and weak well-posedness of (\ref{eq1}).

\begin{proposition}\label{pro4}
Under the assumptions in Theorem \ref{th1}, the pathwise uniqueness holds for
solutions in the sense of Definition \ref{de1} provided satisfying (\ref{es2}).
\end{proposition}
\begin{proof}
First,  let $X(t)$, $Y(t)$ be two solutions to (\ref{eq1}) in the sense of Definition \ref{de1}, which satisfy (\ref{es2}). Set
$$\tau_R:=\tau_R^X\wedge\tau_R^Y=\inf\big\{t\in[0,\infty):|X(t)|\vee|Y(t)|\geq R\big\},~R>0,$$
with the convention $\inf{\emptyset}=\infty$. It is easy to see that $\tau_R\uparrow T$ as $R\uparrow\infty$.

Clearly, $Z(t):=X(t)-Y(t)$ satisfies the equation
\begin{eqnarray*}
Z(t)=~~\!\!\!\!\!\!\!\!&&\int_0^t\big(b(s,X(s),\mathcal{L}_{X(s)})-b(s,Y(s),\mathcal{L}_{Y(s)})\big)ds
\nonumber \\
\!\!\!\!\!\!\!\!&&
+\int_0^t\big(\sigma(s,X(s),\mathcal{L}_{X(s)})-\sigma(s,Y(s),\mathcal{L}_{Y(s)})\big)dW(s).
\end{eqnarray*}
\textbf{Case 1:} Assume that the condition $(\mathbf{A2}')$ holds. By It\^{o}'s formula, it follows that for any $t\in[0,T]$,
\begin{eqnarray}\label{eqca1}
\!\!\!\!\!\!\!\!&&\mathbb{E}\Big[\sup_{s\in[0,t\wedge\tau_R]}|Z(s)|^2\Big]
\nonumber \\
=~~\!\!\!\!\!\!\!\!&&\mathbb{E}\Big[\sup_{s\in[0,t\wedge\tau_R]}\int_0^s\Big(2\langle b(u,X(u),\mathcal{L}_{X(u)})-b(u,Y(u),\mathcal{L}_{Y(u)}), Z(u)\rangle
\nonumber \\
\!\!\!\!\!\!\!\!&&
~~~~~+\|\sigma(u,X(u),\mathcal{L}_{X(u)})-\sigma(u,Y(u),\mathcal{L}_{Y(u)})\|^2\Big) du\Big]
\nonumber \\
\!\!\!\!\!\!\!\!&&+2\mathbb{E}\Big[\sup_{s\in[0,t\wedge\tau_R]}\int_0^s\langle Z(u),\big(\sigma(u,X(u),\mathcal{L}_{X(u)})-\sigma(u,Y(u),\mathcal{L}_{Y(u)})\big)dW(u)\rangle\Big].
\nonumber \\
=:~~\!\!\!\!\!\!\!\!&&~(\text{I})+(\text{II}).
\end{eqnarray}
Applying Burkholder-Davis-Gundy's inequality and Young's inequality, we obtain
\begin{eqnarray}\label{eqca2}
(\text{II})\leq~~ \!\!\!\!\!\!\!\!&&C\mathbb{E}\Bigg(\int_0^{t\wedge\tau_R}|Z(s)|^2\|\sigma(s,X(s),\mathcal{L}_{X(s)})-\sigma(s,Y(s),\mathcal{L}_{Y(s)})\|^2ds\Bigg)^{\frac{1}{2}}
\nonumber \\
\leq~~\!\!\!\!\!\!\!\!&&C\mathbb{E}\Bigg(\sup_{s\in[0,t\wedge\tau_R]}|Z(s)|^2\cdot\int_0^{t\wedge\tau_R}\|\sigma(s,X(s),\mathcal{L}_{X(s)})-\sigma(s,Y(s),\mathcal{L}_{Y(s)})\|^2ds\Bigg)^{\frac{1}{2}}
\nonumber \\
\leq~~\!\!\!\!\!\!\!\!&& C\mathbb{E}\Bigg(\int_0^{t\wedge\tau_R}\|\sigma(s,X(s),\mathcal{L}_{X(s)})-\sigma(s,Y(s),\mathcal{L}_{Y(s)})\|^2ds\Bigg)
\nonumber \\
\!\!\!\!\!\!\!\!&&
+\frac{1}{2}\mathbb{E}\Big[\sup_{s\in[0,t\wedge\tau_R]}|Z(s)|^2\Big].
\end{eqnarray}
Combining $(\mathbf{A2}')$, (\ref{eqca1}) and (\ref{eqca2}), we have
\begin{eqnarray*}
\mathbb{E}\Big[\sup_{s\in[0,t\wedge\tau_R]}|Z(s)|^2\Big]
\leq~~\!\!\!\!\!\!\!\!&&C_R\mathbb{E}\Big[\int_0^{t\wedge\tau_R}\Big(|Z(s)|^2+\mathbb{W}_{2,T,R}(\mathcal{L}_{X_s},\mathcal{L}_{Y_s})^2
\nonumber \\
\!\!\!\!\!\!\!\!&&
~~~~~
+C_0e^{-\varepsilon C_R}(1\wedge\mathbb{W}_{2}(\mathcal{L}_{X(s)},\mathcal{L}_{Y(s)})^2)\Big)ds\Big]
\nonumber \\
\!\!\!\!\!\!\!\!&&
+\frac{1}{2}\mathbb{E}\Big[\sup_{s\in[0,t\wedge\tau_R]}|Z(s)|^2\Big].
\end{eqnarray*}

Note that, in view of the definition of $\mathbb{W}_{2,T,R}(\mathcal{L}_{X_s},\mathcal{L}_{Y_s})$, one has
$$\mathbb{W}_{2,T,R}(\mathcal{L}_{X_s},\mathcal{L}_{Y_s})^2\leq \mathbb{E}\|X_{s\wedge\tau_R}-Y_{s\wedge\tau_R}\|_{T}^2=\mathbb{E}\Big[\sup_{r\in[0,s\wedge\tau_R]}|Z(r)|^2\Big].$$
Thus we deduce that
$$\mathbb{E}\Big[\sup_{s\in[0,t]}|Z(s\wedge\tau_R)|^2\Big]\leq C_R\int_0^t\mathbb{E}\Big[\sup_{r\in[0,s]}|Z(r\wedge\tau_R)|^2\Big]ds+C_0C_Re^{-\varepsilon C_R}t.$$
From Fatou's lemma and Gronwall's inequality, we can get
$$\mathbb{E}\Big[\sup_{s\in[0,t]}|Z(s)|^2\Big]\leq \liminf_{R\to\infty}\mathbb{E}\Big[\sup_{s\in[0,t]}|Z(s\wedge\tau_R)|^2\Big]\leq C_0t\liminf_{R\to\infty}C_Re^{-C_R(\varepsilon-t)}.$$
Then we can choose $t_0\in(0,\varepsilon\wedge T)$ such that the pathwise uniqueness holds on interval $[0,t_0)$. Since $t_0$ is independent of initial value, we can prove that for any $t\in[t_0,2t_0\wedge T)$,
$$\mathbb{E}\Big[\sup_{s\in[t_0,2t_0]}|Z(s)|^2\Big]\leq \liminf_{R\to\infty}\mathbb{E}\Big[\sup_{s\in[t_0,2t_0]}|Z(s\wedge\tau_R)|^2\Big]\leq C_0t\liminf_{R\to\infty}C_Re^{-C_R(\varepsilon-(t-t_0))}=0.$$
Repeating the same procedure for finite times, we show the pathwise uniqueness holds up to time $T$.\\

\textbf{Case 2:} Assume that $(\mathbf{A2}'')$ holds. First, by It\^{o}'s formula and $(\mathbf{A2}'')$ we have that for any $t\in[0,T]$,
\begin{eqnarray}\label{eqc1}
\!\!\!\!\!\!\!\!&&\mathbb{E}|Z(t)|^2
\nonumber \\
\leq~~ \!\!\!\!\!\!\!\!&& \mathbb{E}\Big[\int_0^{t}\mathbf{1}_{\{|X(s)|\leq R,~|Y(s)|\leq R\}}\big(K_s(R)+C\mathbb{E}|X(s)|^{\kappa}+C\mathbb{E}|Y(s)|^{\kappa}\big)
\nonumber \\
\!\!\!\!\!\!\!\!&&~~
\cdot\big(|Z(s)|^2+\mathbb{E}|Z(s)|^2\big)ds\Big]
\nonumber \\
\!\!\!\!\!\!\!\!&&
+\mathbb{E}\Big[\int_0^t\Big(\mathbf{1}_{\{|X(s)|\leq R,~|Y(s)|> R\}}+\mathbf{1}_{\{|X(s)|> R,~|Y(s)|\leq R\}}+\mathbf{1}_{\{|X(s)|> R,~|Y(s)|> R\}}\Big)
\nonumber \\
\!\!\!\!\!\!\!\!&&~~
\cdot\Big(2\langle b(s,X(s),\mathcal{L}_{X(s)})-b(s,Y(s),\mathcal{L}_{Y(s)}), Z(s)\rangle
\nonumber \\
\!\!\!\!\!\!\!\!&&
~~~~~+\|\sigma(s,X(s),\mathcal{L}_{X(s)})-\sigma(s,Y(s),\mathcal{L}_{Y(s)})\|^2\Big) du\Big]
\nonumber \\
\leq~~ \!\!\!\!\!\!\!\!&& \int_0^{t}\big(2K_s(R)+C\mathbb{E}|X(s)|^{\kappa}+C\mathbb{E}|Y(s)|^{\kappa}\big)
\mathbb{E}|Z(s)|^2ds+\text{I}(R),
\end{eqnarray}
where the term
\begin{eqnarray*}
\text{I}(R):=~~~\!\!\!\!\!\!\!\!&&C\mathbb{E}\Big[\int_0^t\Big(\mathbf{1}_{\{|X(s)|>R\}}+\mathbf{1}_{\{|Y(s)|> R\}}\Big)
\Big|2\langle b(s,X(s),\mathcal{L}_{X(s)})-b(s,Y(s),\mathcal{L}_{Y(s)}), Z(s)\rangle
\nonumber \\
\!\!\!\!\!\!\!\!&&
~~~~~+\|\sigma(s,X(s),\mathcal{L}_{X(s)})-\sigma(s,Y(s),\mathcal{L}_{Y(s)})\|^2\Big| du\Big].
\end{eqnarray*}
In view of condition $(\mathbf{A4})$, we deduce that
\begin{eqnarray}\label{eqc2}
\text{I}(R)\leq ~~~\!\!\!\!\!\!\!\!&&C\mathbb{E}\Big[\int_0^t\Big(\mathbf{1}_{\{|X(s)|>R\}}+\mathbf{1}_{\{|Y(s)|> R\}}\Big)
\nonumber \\
\!\!\!\!\!\!\!\!&&\cdot
\big(1+|X(s)|^{\kappa}+|Y(s)|^{\kappa}+\mathbb{E}|X(s)|^{\kappa}+\mathbb{E}|Y(s)|^{\kappa}\big)\big(|X(s)|+|Y(s)|\big)ds\Big]
\nonumber \\
\!\!\!\!\!\!\!\!&&+C\mathbb{E}\Big[\int_0^t\Big(\mathbf{1}_{\{|X(s)|>R\}}+\mathbf{1}_{\{|Y(s)|> R\}}\Big)
\nonumber \\
\!\!\!\!\!\!\!\!&&\cdot
\big(1+|X(s)|^2+|Y(s)|^2+\mathbb{E}|X(s)|^{2}+\mathbb{E}|Y(s)|^{2}\big)ds\Big]
\nonumber \\
\leq~~~\!\!\!\!\!\!\!\!&&C\mathbb{E}\Big[\int_0^t\Big(\mathbf{1}_{\{|X(s)|>R\}}+\mathbf{1}_{\{|Y(s)|> R\}}\Big)
\nonumber \\
\!\!\!\!\!\!\!\!&&\cdot
\big(1+|X(s)|^{\kappa+1}+|Y(s)|^{\kappa+1}+\mathbb{E}|X(s)|^{\kappa+1}+\mathbb{E}|Y(s)|^{\kappa+1}\big)ds\Big]
\nonumber \\
\leq~~~\!\!\!\!\!\!\!\!&&C\Big(\int_0^t\big(\mathbb{P}(|X(s)|>R)+\mathbb{P}(|Y(s)|>R)\big)ds\Big)^{\frac{1}{r}}
\nonumber \\
\!\!\!\!\!\!\!\!&&\cdot\Big(\int_0^t\big(1+\mathbb{E}|X(s)|^{\frac{(\kappa+1)r}{r-1}}+\mathbb{E}|Y(s)|^{\frac{(\kappa+1)r}{r-1}}\big)ds\Big)^{\frac{r-1}{r}}
\nonumber \\
\leq~~~\!\!\!\!\!\!\!\!&&C_T\Big(\int_0^t\big(\mathbb{P}(|X(s)|>R)
+\mathbb{P}(|Y(s)|>R)\big)ds\Big)^{\frac{1}{r}},
\end{eqnarray}
for some $r>1$.

Given a non-decreasing function $f:\mathbb{R}\to\mathbb{R}$ satisfying (\ref{eqe1}) and (\ref{eqe2}).  By Chebyshev's inequality, we have
\begin{equation}\label{eqc3}
\mathbb{P}(|X(s)|>R)+\mathbb{P}(|Y(s)|>R)\leq Ce^{-2rf(R)}\Big(\mathbb{E}e^{2rf(|X(s)|)}+\mathbb{E}e^{2rf(|Y(s)|)}\Big).
\end{equation}
Combining (\ref{eqc1})-(\ref{eqc3}), by Gronwall's lemma and (\ref{eqe1}) we can show that
\begin{eqnarray}\label{es21}
 \mathbb{E}|Z(t)|^2
\leq~~\!\!\!\!\!\!\!\!&& C_T\Big(\sup_{s\in[0,T]}\big(\mathbb{E} e^{2rf(|X(s)|)}+\mathbb{E} e^{2rf(|Y(s)|)}\big)\Big)^{\frac{1}{ r}}e^{-2\big(f(R)-\int_0^TK_s(R)ds\big)}
\nonumber \\
\leq~~~\!\!\!\!\!\!\!\!&&C_Te^{-2\big(f(R)-\int_0^TK_s(R)ds\big)}.
\end{eqnarray}
 Taking $R\to \infty$ on both sides of (\ref{es21}), by (\ref{eqe2}) we can get that for any $t\in[0,T]$,
$$\mathbb{E}|Z(t)|^2=0,$$
which together with the path continuity implies the pathwise uniqueness.      \\

\textbf{Case 3:}
Assume that the condition $(\mathbf{A2}''')$ holds. Define
\begin{equation}\label{eqph1}
\phi(t):=C(1+\mathbb{E}|X(t)|^{\kappa}+\mathbb{E}|Y(t)|^{\kappa}),
\end{equation}
where the constant $C$ is the same as in $(\mathbf{A2}''')$.

Applying It\^{o}'s formula and product rule, we can get that for any $t\in[0,T]$,
\begin{eqnarray*}
\!\!\!\!\!\!\!\!&&\exp\Big\{-\int_0^t\phi(s)ds\Big\}|X(t)-Y(t)|^2
\nonumber\\
\leq~~\!\!\!\!\!\!\!\!&&\int_0^t\exp\Big\{-\int_0^s\phi(r)dr\Big\}
\Big(\|\sigma(s,X(s),\mathcal{L}_{X(s)})-\sigma(s,Y(s),\mathcal{L}_{Y(s)})\|^2
\nonumber\\
\!\!\!\!\!\!\!\!&&~~~~+2\langle b(s,X(s),\mathcal{L}_{X(s)})-b(s,Y(s),\mathcal{L}_{Y(s)}),X(s)-Y(s)\rangle
-\phi(s)|X(s)-Y(s)|^2\Big)ds
\nonumber\\
\!\!\!\!\!\!\!\!&&+2\int_0^t\exp\Big\{-\int_0^s\phi(r)dr\Big\}
\langle X(s)-Y(s),\big(\sigma(s,X(s),\mathcal{L}_{X(s)})-\sigma(s,Y(s),\mathcal{L}_{Y(s)})\big)dW(s)\rangle.
\end{eqnarray*}
By (\ref{es2}) and taking expectation on both sides of the above inequality, we derive that
\begin{eqnarray*}
\!\!\!\!\!\!\!\!&&\mathbb{E}\Bigg\{\exp\Big\{-\int_0^t\phi(s)ds\Big\}|X(t)-Y(t)|^2\Bigg\}
\nonumber\\
\leq~~\!\!\!\!\!\!\!\!&&\int_0^t\exp\Big\{-\int_0^s\phi(r)dr\Big\}
\mathbb{E}\Big(\|\sigma(s,X(s),\mathcal{L}_{X(s)})-\sigma(s,Y(s),\mathcal{L}_{Y(s)})\|^2
\nonumber\\
\!\!\!\!\!\!\!\!&&~~~~+2\langle b(s,X(s),\mathcal{L}_{X(s)})-b(s,Y(s),\mathcal{L}_{Y(s)}),X(s)-Y(s)\rangle
-\phi(s)|X(s)-Y(s)|^2\Big)ds
\nonumber\\
=~~\!\!\!\!\!\!\!\!&&\int_0^t\exp\Big\{-\int_0^s\phi(r)dr\Big\}
\Big(\int_{\mathbb{R}^d\times\mathbb{R}^d}\|\sigma(s,x,\mathcal{L}_{X(s)})-\sigma(s,y,\mathcal{L}_{Y(s)})\|^2
\nonumber\\
\!\!\!\!\!\!\!\!&&~~~~+2\langle b(s,x,\mathcal{L}_{X(s)})-b(s,y,\mathcal{L}_{Y(s)}),x-y\rangle \pi(dx,dy)
\nonumber\\
\!\!\!\!\!\!\!\!&&~~~~
-\int_{\mathbb{R}^d\times\mathbb{R}^d}\phi(s)|x-y|^2\pi(dx,dy)\Big)ds
\nonumber\\
\leq~~\!\!\!\!\!\!\!\!&&\int_0^t\exp\Big\{-\int_0^s\phi(r)dr\Big\}
\Big(\int_{\mathbb{R}^d\times\mathbb{R}^d}\phi(s)|x-y|^2 \pi(dx,dy)
\nonumber\\
\!\!\!\!\!\!\!\!&&~~~~
-\int_{\mathbb{R}^d\times\mathbb{R}^d}\phi(s)|x-y|^2\pi(dx,dy)\Big)ds
\nonumber\\
=~~\!\!\!\!\!\!\!\!&&~~ 0,
\end{eqnarray*}
where $\pi\in\mathfrak{C}(\mathcal{L}_{X(s)},\mathcal{L}_{Y(s)})$.

Consequently, by (\ref{es2}) we deduce that $X(t)=Y(t),~\mathbb{P}\text{-a.s.},~t\in[0,T]$, thus the pathwise uniqueness follows from the path continuity.
\end{proof}

\vspace{3mm}
\textbf{Proof of Theorem \ref{th1}:} Combining with Propositions \ref{pro3}, \ref{pro4} and the modified Yamada-Watanabe theorem yields that MVSDE (\ref{eq1}) is strongly and weakly well-posed.  Now the proof is complete.

\section{Well-posedness of McKean-Vlasov SPDEs}\label{sec4}
The aim of this section is to extend the existence and uniqueness result to the case of  MVSPDE, which is applicable to various concrete SPDE models with interaction external force.
To this end, we first introduce some functional spaces and necessary notations.

Let $(U,\langle\cdot,\cdot\rangle_U)$ and $(H, \langle\cdot,\cdot\rangle_H) $ be  separable Hilbert spaces, and $H^*$ the dual space of $H$. Let $V$ denote a reflexive Banach space such that the embedding $V\subset H$ is continuous and dense. Identifying $H$ with its dual space by the Riesz isomorphism, then we have the so-called Gelfand triple
$$V\subset H(\cong H^*)\subset V^*.$$
The dualization between $V$ and $V^*$ is denoted by $_{V^*}\langle\cdot,\cdot\rangle_V$. Moreover, it is easy to see that $_{V^*}\langle\cdot,\cdot\rangle_V|_{H\times V}=\langle\cdot,\cdot\rangle_H$.

Let $\mathcal{P}(H)$ represent the space of all probability measures on $H$ equipped with the weak topology. Furthermore, we set for any $p>0$,
$$\mathcal{P}_p(H):=\Big\{\mu\in\mathcal{P}(H):\mu(\|\cdot\|_{H}^p):=\int_H\|\xi\|_H^p\mu(d\xi)<\infty\Big\}.$$
Then $\mathcal{P}_p(H)$ is a Polish space under the $L^p$-Wasserstein distance
$$\mathbb{W}_{p,H}(\mu,\nu):=\inf_{\pi\in\mathfrak{C}(\mu,\nu)}\Big(\int_{H\times H}\|\xi-\eta\|_H^p\pi(d\xi,d\eta)\Big)^{\frac{1}{p\vee1}},~\mu,\nu\in\mathcal{P}_p(H),$$
here $\mathfrak{C}(\mu,\nu)$ stands for the set of all couplings for  $\mu$ and $\nu$.
We also recall $\mathcal{P}_{\theta,T}(H)$, $\mathbb{W}_{2,T,R,H}$ are defined by (\ref{P1}), (\ref{P2}), respectively, with $H$ replacing $\mathbb{R}^d$, and
$$\tau_R^u:=\inf\Big\{t\in[0,T]:\|u(t)\|_{H}+\int_0^t\|u(s)\|_{V}^{\alpha}ds\geq R\Big\},$$
where the constant $\alpha$ will be determined later.
\subsection{Main results}\label{secin}
For the measurable maps
$$
A:[0,T]\times V\times\mathcal{P}(H)\rightarrow V^*,~~B:[0,T]\times V\times\mathcal{P}(H)\rightarrow L_2(U,H),
$$
we are interested in the following McKean-Vlasov stochastic evolution equation on  $H$,
\begin{equation}\label{eqSPDE}
dX(t)=A(t,X(t),\mathcal{L}_{X(t)})dt+B(t,X(t),\mathcal{L}_{X(t)})dW(t),
\end{equation}
where $\{W(t)\}_{t\in [0,T]}$ is an $U$-valued cylindrical Wiener process defined on a complete filtered probability space $\left(\Omega,\mathcal{F},\{\mathcal{F}(t)\}_{t\in[0,T]},\mathbb{P}\right)$.

To investigate the existence and uniqueness of solutions to MVSPDE (\ref{eqSPDE}), we assume that there are some constants $\alpha>1$, $C,\delta>0$, $\theta\geq 2$, $\beta\geq 0$ and a function $f(\cdot)\in L^1([0,T];[0,\infty))$ such that the following conditions hold for all $t\in[0,T]$.

\begin{enumerate}

\item [$({\mathbf{H}}{\mathbf{1}})$]\label{H1}
 $($Demicontinuity$)$ For any $v\in V$, the map
\begin{equation*}
V\times\mathcal{P}_2(H)\ni(u,\mu)\mapsto_{V^*}\langle A(t,u,\mu),v\rangle_V
\end{equation*}
is continuous.

\item [$({\mathbf{H}}{\mathbf{2}})$]\label{H2}
 $($Local Monotonicity$)$ For any $u,v\in V$, $\mu\in\mathcal{P}_{ 2 }(H)$, $\nu\in\mathcal{P}_{ \theta }(H)$,
\begin{eqnarray*}
\!\!\!\!\!\!\!\!&&2_{V^*}\langle A(t,u,\mu)-A(t,v,\nu),u-v\rangle_V+\|B(t,u,\mu)-B(t,v,\nu)\|_{L_2(U,H)}^2
\nonumber\\
\!\!\!\!\!\!\!\!&&\leq
\big(C+\rho(v)+C\nu(\|\cdot\|_H^{\theta})\big)\|u-v\|_H^2
+C\big(1+\nu(\|\cdot\|_H^{\theta})\big)\mathbb{W}_{2,H}(\mu,\nu)^2,
\end{eqnarray*}
where $\rho:V\to [0,\infty)$ is a measurable function and locally bounded in $V$.

\item [$({\mathbf{H}}{\mathbf{3}})$]\label{H3}
 $($Coercivity$)$ For any $u\in V$, $\mu\in\mathcal{P}_2(H)$,
\begin{equation*}
2_{V^*}\langle A(t,u,\mu),u\rangle_V+\delta\|u\|_V^\alpha\leq C\|u\|_H^2+C\mu(\|\cdot\|_H^2)+f(t).
\end{equation*}

\item [$({\mathbf{H}}{\mathbf{4}})$]\label{H4}
 $($Growth$)$ For any $u\in V$, $\mu\in\mathcal{P}_{\theta}(H)$,
\begin{equation*}
\|A(t,u,\mu)\|_{V^*}^{\frac{\alpha}{\alpha-1}}\leq \big(f(t)+C\|u\|_V^{\alpha}+C \mu(\|\cdot\|_H^{\theta} )\big)\big(1+\|u\|_{H}^{\beta}+\mu(\|\cdot\|_H^{\theta})\big)
\end{equation*}
and
\begin{equation*}
\|B(t,u,\mu)\|_{L_2(U,H)}^2\leq f(t)+C\|u\|_H^2+C\mu(\|\cdot\|_H^2).
\end{equation*}
\end{enumerate}

We recall the definition of variational solution to  (\ref{eqSPDE}).
\begin{definition}\label{de2}
We say a continuous $H$-valued $\{\mathcal{F}(t)\}_{t\in[0,T]}$-adapted process $\{X(t)\}_{t\in[0,T]}$ is a solution of (\ref{eqSPDE}), if for its $dt\times \mathbb{P}$-equivalent class $\hat{X}$
\begin{equation*}
\hat{X}\in L^\alpha\big([0,T]\times\Omega,dt\times\mathbb{P};V\big)\cap L^2\big([0,T]\times\Omega,dt\times\mathbb{P};H\big),
\end{equation*}
where $\alpha$ is the same as  in $(\mathbf{H3})$ and $\mathbb{P}$-a.s.
\begin{equation*}
X(t)=X(0)+\int_0^t A(s,\bar{X}(s),\mathcal{L}_{\bar{X}(s)})ds+\int_0^t B(s,\bar{X}(s),\mathcal{L}_{\bar{X}(s)})dW(s),~t\in[0,T],
\end{equation*}
where $\bar{X}$ is an $V$-valued progressively measurable $dt\times\mathbb{P}$-version of $\hat{X}$.
\end{definition}

Now we show the existence and uniqueness result to  MVSPDE (\ref{eqSPDE}).
\begin{theorem}\label{th2}
Suppose that the embedding $V\subset H$ is compact, $(\mathbf{H1})$-$(\mathbf{H4})$ hold for
$\sup_{t\in[0,T]}f(t)<\infty$
 and
\begin{equation}\label{es22}
\rho(v)\leq C(1+\|v\|_V^{\alpha})(1+\|v\|_H^{\beta}),~v\in V.
\end{equation}
Then for any $X(0)\in L^p(\Omega,\mathcal{F}(0),\mathbb{P};H)$ with $p\in(\eta,\infty)\cap[\beta+2,\infty)$, where
$\eta:=\frac{2(\alpha+\beta)(\alpha-1)}{\alpha}\vee \theta$,
MVSPDE (\ref{eqSPDE}) has a  solution in the sense of Definition \ref{de2}. Moreover,
\begin{equation}\label{es37}
\mathbb{E}\Big[\sup_{t\in[0,T]}\|X(t)\|_H^p\Big]+\mathbb{E}\int_0^T\|X(t)\|_V^{\alpha}dt+\mathbb{E}\int_0^T\|X(t)\|_H^{p-2}\|X(t)\|_V^{\alpha}dt<\infty.
\end{equation}

Furthermore, if one of the following conditions holds
\begin{enumerate}
\item [$({\mathbf{H}}{\mathbf{2}}')$]\label{H21}
 There exists $C>0$ such that for any $R>0$, $u,v\in \mathcal{C}_T\cap L^{\alpha}([0,T];V)$, $\mu,\nu\in \mathcal{P}_{\theta,T}(H)$  and $t\in[0,T\wedge\tau_R^{u}\wedge\tau_R^{v}]$,
\begin{eqnarray*}
\!\!\!\!\!\!\!\!&&_{V^*}\langle A(t,u(t),\mu(t))-A(t,v(t),\nu(t)),u(t)-v(t)\rangle_V
\nonumber\\
\leq~\!\!\!\!\!\!\!\!&&~~
\phi(t)\Big(\|u(t)-v(t)\|_H^2+\mathbb{W}_{2,T,R,H}(\mu_t,\nu_t)^2\Big)
\end{eqnarray*}
and
\begin{eqnarray*}
\!\!\!\!\!\!\!\!&&\|B(t,u(t),\mu(t))-B(t,v(t),\nu(t))\|_{L_2(U,H)}^2
\nonumber\\
\leq~\!\!\!\!\!\!\!\!&&~~
\phi(t)\Big(\|u(t)-v(t)\|_H^2+\mathbb{W}_{2,T,R,H}(\mu_t,\nu_t)^2\Big),
\end{eqnarray*}
where $\phi(t):=C+\rho_1(u(t))+\rho_2(v(t))+C\mu(t)(\|\cdot\|_H^{\theta})+C\nu(t)(\|\cdot\|_H^{\theta})$,
$\rho_1,\rho_2$ are  measurable functions and locally bounded in $V$, which satisfy (\ref{es22}).

\vspace{2mm}
\item [$({\mathbf{H}}{\mathbf{2}}'')$]\label{H22}
There exists $C>0$ such that for any $u,v\in V$, $\mu,\nu\in\mathcal{P}_{\theta}(H)$,
\begin{eqnarray*}
\!\!\!\!\!\!\!\!&&2_{V^*}\langle A(t,u,\mu)-A(t,v,\nu),u-v\rangle_V+\|B(t,u,\mu)-B(t,v,\nu)\|_{L_2(U,H)}^2
\nonumber\\
\!\!\!\!\!\!\!\!&&\leq
 C\big(1+\mu(\|\cdot\|_H^{\theta})+\nu(\|\cdot\|_H^{\theta})\big)\big(\|u-v\|_H^2+\mathbb{W}_{2,H}(\mu,\nu)^2\big),
\end{eqnarray*}
\end{enumerate}
then MVSPDE (\ref{eqSPDE}) has a unique solution in the sense of Definition \ref{de2} provided satisfying (\ref{es37}).

\end{theorem}

\begin{remark}
(i) The proof of the existence of solutions to (\ref{eqSPDE}) is given in Subsection \ref{sec3.3} and the uniqueness is given in Subsection \ref{sec3.4}.

(ii) Note that the condition $(\mathbf{H2})$ is a generalization of $(\mathbf{A2})$, which is stronger than $(\mathbf{A2})$ if $V=H=\mathbb{R}^d$. The main reason is that we utilize the so-called ``monotone method'' (cf.~\cite[Section 5]{LR1}) in the proof of the existence of solutions to MVSPDEs, it seems very difficult to employ this method under the form of condition $(\mathbf{A2})$. We mention that  $(\mathbf{H2})$ is suitable for  applications of confinement/interaction potentials in the infinite-dimensional case (see Remark \ref{remark6.2} (ii)).
\end{remark}
\subsection{Construction of approximating equations}
Choosing $\{e_1,e_2,\cdots\}\subset V$ as an orthonormal basis (ONB) on $H$. Consider the maps
$$\Pi^n:V^{*}\rightarrow H^n:=\text{span}\{e_1,e_2,\cdots\,e_n\},~n\in\mathbb{N},$$
by
$$\Pi^n x:=\sum\limits_{i=1}^{n}{}_{V^*}\langle x,e_i\rangle_{V}e_i,~x\in V^*.$$
It is straightforward that if we restrict $\Pi^n$ to $H$, denoted by $\Pi^n|_{H}$, then it is an orthogonal projection onto $H^n$ on $H$. Denote by $\{g_1,g_2,\cdots\}$ the ONB of $U$. Let
\begin{equation*}
W^{(n)}(t):=\widetilde{\Pi}^nW(t)=\sum\limits_{i=1}^{n}\langle W(t),g_i\rangle_{U}g_i,~n\in\mathbb{N},
\end{equation*}
where $\widetilde{\Pi}^n$ is an orthonormal projection onto $U^n:=\text{span}\{g_1,g_2,\cdots,g_n\}$ on $U$.

For any $n\in\mathbb{N}$, we consider the following stochastic equation on $H^n$
\begin{equation}\label{eqf}
dX^{(n)}(t)=\Pi^nA(t,X^{(n)}(t),\mathcal{L}_{X^{(n)}(t)})dt
+\Pi^nB(t,X^{(n)}(t),\mathcal{L}_{X^{(n)}(t)})dW^{(n)}(t),
\end{equation}
with initial value $X^{(n)}(0)=\Pi^nX(0)$.  Note that the coefficients of equation (\ref{eqf}) satisfy $(\mathbf{A1})$-$(\mathbf{A4})$.  Hence, according to Theorem \ref{th1}, for any $n\in\mathbb{N}$, it has a  strong solution that satisfies
\begin{equation}\label{eqpri}
\mathbb{E}\Big[\sup_{t\in[0,T]}\|X^{(n)}(t)\|_H^p\Big]\leq C(n)<\infty,
\end{equation}
where $p$ is defined in Theorem \ref{th2} and $C(n)$ is some constant that depends on $n$.

For the sake of convenience, we define the following notations
\begin{eqnarray*}
\!\!\!\!\!\!\!\!&&\mathbb{J}=L^2([0,T]\times\Omega;L_2(U,H));
\nonumber\\\vspace{2mm}
\!\!\!\!\!\!\!\!&&\mathbb{K}=L^\alpha([0,T]\times\Omega;V);
\nonumber\\\vspace{2mm}
\!\!\!\!\!\!\!\!&&\mathbb{K}^*=L^{\frac{\alpha}{\alpha-1}}([0,T]\times\Omega;V^*).
\end{eqnarray*}

\begin{lemma}\label{lem3}
Suppose that the assumptions in Theorem \ref{th2} hold, then for any $T>0$ there exists a constant $C_T>0$ such that for any $n\in\mathbb{N}$,
$$\|X^{(n)}\|_{\mathbb{K}}+\mathbb{E}\Big[\sup_{t\in[0,T]}\|X^{(n)}(t)\|_H^2\Big]\leq C_T\Big(1+\mathbb{E}\|X(0)\|_H^2\Big).$$
Furthermore,  for any $n\in\mathbb{N}$,
$$\|B(\cdot,X^{(n)}(\cdot),\mathcal{L}_{X^{(n)}(\cdot)})\|_{\mathbb{J}}\leq C_T\Big(1+\mathbb{E}\|X(0)\|_H^2\Big).$$
\end{lemma}
\begin{proof}
The proof is standard and follows from It\^{o}'s formula,  one could refer to \cite{LR2} or \cite[Theorem 2.1]{HL1}, thus we omit the details.
\end{proof}

\vspace{1mm}
Besides Lemma \ref{lem3}, we can also get the following a priori estimates.
\begin{lemma}\label{lem4}
Suppose that the assumptions in Theorem \ref{th2} hold, then for any $T>0$ there exists a constant $C_T>0$ such that for any $n\in\mathbb{N}$,
\begin{equation*}
\mathbb{E}\Big[\sup_{t\in[0,T]}\|X^{(n)}(t)\|_H^p\Big]+\mathbb{E}\int_0^T\|X^{(n)}(t)\|_H^{p-2}\|X^{(n)}(t)\|_V^{\alpha}dt
\leq C_T\Big(1+\mathbb{E}\|X(0)\|_H^p\Big).
\end{equation*}
Furthermore,  for any $n\in\mathbb{N}$,
$$\|A(\cdot,X^{(n)}(\cdot),\mathcal{L}_{X^{(n)}(\cdot)})\|_{\mathbb{K}^*}\leq C_T\Big(1+\mathbb{E}\|X(0)\|_H^p\Big).$$
\end{lemma}
\begin{proof}
Applying It\^{o}'s formula, we can get that for any $t\in[0,T]$,
\begin{eqnarray}\label{es24}
\|X^{(n)}(t)\|_H^p
=~~\!\!\!\!\!\!\!\!&&\|X^{(n)}(0)\|_H^p+\frac{p(p-2)}{2}  \int_0^t\|X^{(n)}(s)\|_H^{p-4}
\nonumber \\
\!\!\!\!\!\!\!\!&&
~~~~~~~~~~~~~~~~~~~~~~\cdot\|\big(\Pi^nB(s,X^{(n)}(s),\mathcal{L}_{X^{(n)}(s)})\tilde{\Pi}^n\big)^*X^{(n)}(s)\|_U^2ds
\nonumber \\
\!\!\!\!\!\!\!\!&&+\frac{p}{2}\int_0^t\|X^{(n)}(s)\|_H^{p-2}\big(2{}_{V^*}\langle A(s,X^{(n)}(s),\mathcal{L}_{X^{(n)}(s)}),X^{(n)}(s)\rangle_V
\nonumber \\
\!\!\!\!\!\!\!\!&&
~~~~~~~~~~~~~~~~~~~~~~~~~~~~+\|\Pi^nB(s,X^{(n)}(s),\mathcal{L}_{X^{(n)}(s)})\tilde{\Pi}^n\|_{L_2(U,H)}^2\big)ds
\nonumber \\
\!\!\!\!\!\!\!\!&&+p\int_0^t\|X^{(n)}(s)\|_H^{p-2}\langle X^{(n)}(s),\Pi^nB(s,X^{(n)}(s),\mathcal{L}_{X^{(n)}(s)})dW^{(n)}(s)\rangle_H
\nonumber \\
\leq~~\!\!\!\!\!\!\!\!&&\|X^{(n)}(0)\|_H^p-\frac{p\delta}{2}\int_0^t\|X^{(n)}(s)\|_H^{p-2}\|X^{(n)}(s)\|_V^{\alpha}ds
\nonumber \\
\!\!\!\!\!\!\!\!&&
+C\int_0^t\big(\|X^{(n)}(s)\|_H^p+\mathbb{E}\|X^{(n)}(s)\|_H^p+f^{\frac{p}{2}}(s)\big)ds
\nonumber \\
\!\!\!\!\!\!\!\!&&+p\int_0^t\|X^{(n)}(s)\|_H^{p-2}\langle X^{(n)}(s),\Pi^nB(s,X^{(n)}(s),\mathcal{L}_{X^{(n)}(s)})dW^{(n)}(s)\rangle_H.
\end{eqnarray}
Using  Burkholder-Davis-Gundy's inequality, we obtain
\begin{eqnarray}\label{es25}
\!\!\!\!\!\!\!\!&&\mathbb{E}\Big[\sup_{t\in[0,T]}\Big|\int_0^t\|X^{(n)}(s)\|_H^{p-2}\langle X^{(n)}(s),\Pi^nB(s,X^{(n)}(s),\mathcal{L}_{X^{(n)}(s)})dW^{(n)}(s)\rangle_H\Big|\Big]
\nonumber \\
\!\!\!\!\!\!\!\!&&\leq~C\mathbb{E}\Big[\int_0^{T}\|X^{(n)}(s)\|_H^{2p-2}\|B(s,X^{(n)}(s),\mathcal{L}_{X^{(n)}(s)})\|_{L_2(U,H)}^2ds\Big]^{\frac{1}{2}}
\nonumber \\
\!\!\!\!\!\!\!\!&&\leq~C\mathbb{E}\Big[\int_0^{T}\big(\|X^{(n)}(t)\|_H^p+\mathbb{E}\|X^{(n)}(t)\|_H^p+f^{\frac{p}{2}}(t)\big)dt\Big]
\nonumber \\
\!\!\!\!\!\!\!\!&&~~~
+\frac{1}{2}\mathbb{E}\Big[\sup_{t\in[0,T]} \|X^{(n)}(t)\|_H^p\Big].
\end{eqnarray}
Combining (\ref{eqpri})-(\ref{es25}),  we have
\begin{eqnarray}\label{es26}
\!\!\!\!\!\!\!\!&&\mathbb{E}\Big[\sup_{t\in[0,T]}\|X^{(n)}(t)\|_H^p\Big]+\frac{p\delta}{2}\mathbb{E}\int_0^{T}\|X^{(n)}(t)\|_H^{p-2}\|X^{(n)}(t)\|_V^{\alpha}dt
\nonumber \\
\!\!\!\!\!\!\!\!&&\leq \mathbb{E}\|X(0)\|_H^p+C\mathbb{E}\int_0^{T}\|X^{(n)}(t)\|_H^pdt+C\int_0^{T}f^{\frac{p}{2}}(t)dt.
\end{eqnarray}
Then by Gronwall's lemma, it leads to
$$\mathbb{E}\Big[\sup_{t\in[0,T]}\|X^{(n)}(t)\|_H^p\Big]+\mathbb{E}\int_0^{T}\|X^{(n)}(t)\|_H^{p-2}\|X^{(n)}(t)\|_V^{\alpha}dt\leq C_T\big(1+\mathbb{E}\|X(0)\|_H^p\big).$$
Furthermore, due to $(\mathbf{H4})$ and $p\in(\eta,\infty)\cap[\beta+2,\infty)$, we deduce that
$$\|A(\cdot,X^{(n)}(\cdot),\mathcal{L}_{X^{(n)}(\cdot)})\|_{\mathbb{K}^*}\leq C_T.$$
Hence, we complete the proof.
\end{proof}

\begin{lemma}\label{lem5}
Suppose that the assumptions in Theorem \ref{th2} hold. Then $\{X^{(n)}\}_{n\in\mathbb{N}}$ is tight in the space $C([0,T];V^*)\cap L^{\alpha}([0,T];H)$.
\end{lemma}
\begin{proof}
Since the embedding $V\subset H$ is compact,  the proof is similar to that of Lemma 2.12 in \cite{RSZ1}, thus we omit the details.
\end{proof}
\subsection{Existence of variational solutions}\label{sec3.3}
In the following, we intend to show the existence of (variational) solutions to (\ref{eqSPDE}).

\begin{proposition}\label{pro5}
Suppose that all assumptions in Theorem \ref{th2}  hold. Then for any $X(0)\in L^p(\Omega,\mathcal{F}_0,\mathbb{P};H)$ with $p\in(\eta,\infty)\cap[\beta+2,\infty)$,  there exists a (variational) solution to (\ref{eqSPDE}) in the sense of Definition \ref{de2} and satisfies
\begin{equation*}
\mathbb{E}\Big[\sup_{t\in[0,T]}\|X(t)\|_H^p\Big]+\mathbb{E}\int_0^T\|X(t)\|_V^{\alpha}dt+\mathbb{E}\int_0^T\|X(t)\|_H^{p-2}\|X(t)\|_V^{\alpha}dt<\infty.
\end{equation*}
\end{proposition}

\begin{proof}
Set
$$\Upsilon:=[C([0,T];V^*)\cap L^{\alpha}([0,T];H)]\times C([0,T];U_1),$$
where $U_1$ is a Hilbert space such that the embedding $U\subset U_1$ is Hilbert-Schmidt.
By Lemma \ref{lem5} and the generalized Skorohod representation theorem, there exists a  filtered probability space $(\tilde{\Omega},\tilde{\mathcal{F}},\tilde{\mathcal{F}}(t),\tilde{\mathbb{P}})$ and $\Upsilon$-valued random vectors $\{(\tilde{X}^{(n)},\tilde{W}^{(n)})\}_{n\in\mathbb{N}}$ and $(\tilde{X},\tilde{W})$ such that
\begin{eqnarray}
\!\!\!\!\!\!\!\!&&(i)~\tilde{W}^{(n)}=\tilde{W}~\text{for any}~ n\in\mathbb{N},~~\tilde{\mathbb{P}}\text{-a.s.};
\nonumber\\
\!\!\!\!\!\!\!\!&&(ii)~\mathcal{L}_{(\tilde{X}^{(n)},\tilde{W}^{(n)})}=\mathcal{L}_{(\tilde{X},\tilde{W})};\label{conver0}
\\
\!\!\!\!\!\!\!\!&&(iii)~\|\tilde{X}^{(n)}-\tilde{X}\|_{L^{\alpha}([0,T];H)}+\|\tilde{X}^{(n)}-\tilde{X}\|_{C([0,T];V^*)}\to 0,~~\tilde{\mathbb{P}}\text{-a.s.}. \label{conver1}
\end{eqnarray}

In the following, we will show that $(\tilde{X},\tilde{W})$ is a solution of (\ref{eqSPDE}). Note that by Lemmas \ref{lem3}, \ref{lem4} and (\ref{conver0}),  we have
\begin{equation}\label{es20}
\sup_{n\in\mathbb{N}}\Big\{\tilde{\mathbb{E}}\Big[\sup_{t\in[0,T]}\|\tilde{X}^{(n)}(t)\|_H^p\Big]+\mathbb{E}\int_0^T\|\tilde{X}^{(n)}(t)\|_V^{\alpha}dt+\mathbb{E}\int_0^T\|\tilde{X}^{(n)}(t)\|_H^{p-2}\|\tilde{X}^{(n)}(t)\|_V^{\alpha}dt\Big\}<\infty,
\end{equation}
then by the  lower semicontinuity, we deduce that
\begin{equation}\label{es38}
\tilde{\mathbb{E}}\Big[\sup_{t\in[0,T]}\|\tilde{X}(t)\|_H^p\Big]+\mathbb{E}\int_0^T\|\tilde{X}(t)\|_V^{\alpha}dt+\mathbb{E}\int_0^T\|\tilde{X}(t)\|_H^{p-2}\|\tilde{X}(t)\|_V^{\alpha}dt<\infty.
\end{equation}

As a consequence of (\ref{es20}) and Lemma \ref{lem4}, there exists a subsequences of $n$, denoted again by $n$, such that for $n\rightarrow\infty$,
\begin{eqnarray*}
\!\!\!\!\!\!\!\!&&(i)~\tilde{X}^{(n)}\rightarrow\bar{X}~\text{weakly in~} \mathbb{K}~\text{and weakly star in}~L^p(\Omega;L^{\infty}([0,T];H));
\nonumber\\
\!\!\!\!\!\!\!\!&&(ii)~Y^{(n)}:=A(\cdot,\tilde{X}^{(n)}(\cdot),\mathcal{L}_{\tilde{X}^{(n)}(\cdot)})\rightarrow Y~\text{weakly in~} \mathbb{K}^*;
\nonumber\\
\!\!\!\!\!\!\!\!&&(iii)~Z^{(n)}:=B(\cdot,\tilde{X}^{(n)}(\cdot),\mathcal{L}_{\tilde{X}^{(n)}(\cdot)})\rightarrow Z~\text{weakly in~} \mathbb{J}.
\end{eqnarray*}
From $(iii)$ we can infer that
$$\int_0^{\cdot}\Pi^n B(s,\tilde{X}^{(n)}(s),\mathcal{L}_{\tilde{X}^{(n)}(s)})d\tilde{W}(s)\to\int_0^{\cdot} Z(s)d\tilde{W}(s)$$
weakly star in $L^{\infty}([0,T];L^2(\Omega;H))$.

Let us define
$$X(t):=X(0)+\int_0^tY(s)ds+\int_0^tZ(s)dW(s),~t\in[0,T].$$
Then by a standard argument (cf.~\cite{HLL1} or \cite{RSZ1}), it is straightforward that $X=\bar{X}=\tilde{X}$, $dt\times \tilde{\mathbb{P}}$-a.e..
From now on, we will work on the  filtered probability space $(\tilde{\Omega},\tilde{\mathcal{F}},\tilde{\mathcal{F}}(t),\tilde{\mathbb{P}})$. However, without loss of generality,  we drop all the superscripts $\tilde{}$ to simplify the notations, for example, we write $\tilde{X}^{(n)}$ as $X^{(n)}$. By \cite[Theorem 4.2.5]{LR1}, $X$ is an $H$-valued continuous $(\mathcal{F}(t))$-adapted process.

Now it is sufficient to prove that
\begin{equation}\label{es27}
A(\cdot,X(\cdot),\mathcal{L}_{X(\cdot)})=Y,~B(\cdot,X(\cdot),\mathcal{L}_{X(\cdot)})=Z,~dt\times\mathbb{P}\text{-a.e.}.
\end{equation}

In order to prove (\ref{es27}), we first introduce the following set
$$\mathbb{S}:=\Big\{\phi:\phi~\text{is}~V\text{-valued}~(\mathcal{F}_t)\text{-adapted process such that}~\mathbb{E}\int_0^T\rho(\phi(s))ds<\infty\Big\}.$$
For any $\phi\in \mathbb{K}\cap\mathbb{S}\cap L^p(\Omega;L^{\infty}([0,T];H))$,
\begin{eqnarray}\label{es28}
\!\!\!\!\!\!\!\!&&\mathbb{E}\Big(e^{-\int_0^t(C+C\mathbb{E}\|\phi(s)\|_H^{\theta}+\rho(\phi(s)))ds}\|X^{(n)}(t)\|_H^2\Big)-\mathbb{E}\|X^{(n)}(0)\|_H^2
\nonumber\\
\leq~~\!\!\!\!\!\!\!\!&&\mathbb{E}\Big[\int_0^te^{-\int_0^s(C+C\mathbb{E}\|\phi(r)\|_H^{\theta}+\rho(\phi(r)))dr}\Big(\|B(s,X^{(n)}(s),\mathcal{L}_{X^{(n)}(s)})\|_{L_2(U,H)}^2
\nonumber\\
\!\!\!\!\!\!\!\!&&
~~~~~~~~~+2{}_{V^*}\langle A(s,X^{(n)}(s),\mathcal{L}_{X^{(n)}(s)}),X^{(n)}(s)\rangle_V
\nonumber\\
\!\!\!\!\!\!\!\!&&
~~~~~~~~~
-(C+C\mathbb{E}\|\phi(s)\|_H^{\theta}+\rho(\phi(s)))\|X^{(n)}(s)\|_H^2\Big)ds\Big]
\nonumber\\
\leq~~\!\!\!\!\!\!\!\!&&\mathbb{E}\Big[\int_0^te^{-\int_0^s(C+C\mathbb{E}\|\phi(r)\|_H^{\theta}+\rho(\phi(r)))dr}
\nonumber\\
\!\!\!\!\!\!\!\!&&~~~~~~~\cdot
\Big(\|B(s,X^{(n)}(s),\mathcal{L}_{X^{(n)}(s)})-B(s,\phi(s),\mathcal{L}_{\phi(s)})\|_{L_2(U,H)}^2
\nonumber\\
\!\!\!\!\!\!\!\!&&~~~~~~~+2{}_{V^*}\langle A(s,X^{(n)}(s),\mathcal{L}_{X^{(n)}(s)})-A(s,\phi(s),\mathcal{L}_{\phi(s)}),X^{(n)}(s)-\phi(s)\rangle_V
\nonumber\\
\!\!\!\!\!\!\!\!&&~~~~~~~-(C+C\mathbb{E}\|\phi(s)\|_H^{\theta}+\rho(\phi(s)))\|X^{(n)}(s)-\phi(s)\|_H^2\Big)ds\Big]
\nonumber\\
\!\!\!\!\!\!\!\!&&+\mathbb{E}\Big[\int_0^te^{-\int_0^s(C+C\mathbb{E}\|\phi(r)\|_H^{\theta}+\rho(\phi(r)))dr}
\nonumber\\
\!\!\!\!\!\!\!\!&&~~~~~~~\cdot
\Big(2{}_{V^*}\langle A(s,X^{(n)}(s),\mathcal{L}_{X^{(n)}(s)})-A(s,\phi(s),\mathcal{L}_{\phi(s)}),\phi(s)\rangle_V
\nonumber\\
\!\!\!\!\!\!\!\!&&~~~~~~~+2{}_{V^*}\langle A(s,\phi(s),\mathcal{L}_{\phi(s)}),X^{(n)}(s)\rangle_V-\|B(s,\phi(s),\mathcal{L}_{\phi(s)})\|_{L_2(U,H)}^2
\nonumber\\
\!\!\!\!\!\!\!\!&&~~~~~~~+2\langle B(s,X^{(n)}(s),\mathcal{L}_{X^{(n)}(s)}),B(s,\phi(s),\mathcal{L}_{\phi(s)})\rangle_{L_2(U,H)}
\nonumber\\
\!\!\!\!\!\!\!\!&&~~~~~~~
-2(C+C\mathbb{E}\|\phi(s)\|_H^{\theta}+\rho(\phi(s)))\langle X^{(n)}(s),\phi(s)\rangle_H
\nonumber\\
\!\!\!\!\!\!\!\!&&~~~~~~~+(C+C\mathbb{E}\|\phi(s)\|_H^{\theta}+\rho(\phi(s)))\|\phi(s)\|_H^2\Big)ds\Big]
\nonumber\\
=:~~\!\!\!\!\!\!\!\!&&~(\text{I})+(\text{II}),
\end{eqnarray}
where the constant $C>0$ is the same as in the condition $(\mathbf{H2})$.

Note that by $(\mathbf{H2})$ and the definition of $\phi$, the first integral on the right hand side of (\ref{es28}) can be controlled as follows
\begin{eqnarray}\label{es29}
(\text{I})\leq~~\!\!\!\!\!\!\!\!&&C\mathbb{E}\Big[\int_0^t(1+\mathbb{E}\|\phi(s)\|_H^{\theta})
\mathbb{E}\|X^{(n)}(s)-\phi(s)\|_H^2ds\Big]
\nonumber\\
\leq~~\!\!\!\!\!\!\!\!&& C\mathbb{E}\int_0^t
\|X^{(n)}(s)-\phi(s)\|_H^2ds.
\end{eqnarray}
In view of (\ref{es28}) and (\ref{es29}), due to the lower semicontinuity (see e.g.~\cite[(3.13)]{HLL1}), for any $\psi\in L^{\infty}([0,T];[0,\infty))$,
\begin{eqnarray}\label{es31}
\!\!\!\!\!\!\!\!&&
\mathbb{E}\Big[\int_0^T\psi_t\Big(e^{-\int_0^t(C+C\mathbb{E}\|\phi(s)\|_H^{\theta}+\rho(\phi(s)))ds}\|X(t)\|_H^2-\|X(0)\|_H^2\Big)dt\Big]
\nonumber\\
\leq~~\!\!\!\!\!\!\!\!&&\liminf_{n\to\infty}\mathbb{E}\Big[\int_0^T\psi_t\Big(e^{-\int_0^t(C+C\mathbb{E}\|\phi(s)\|_H^{\theta}+\rho(\phi(s)))ds}\|X^{(n)}(t)\|_H^2-\|X^{(n)}(0)\|_H^2\Big)dt\Big]
\nonumber\\
\leq~~\!\!\!\!\!\!\!\!&&C\liminf_{n\to\infty}\mathbb{E}\int_0^T
\|X^{(n)}(s)-\phi(s)\|_H^2ds+\mathbb{E}\Big[\int_0^T\psi_t\Big(\int_0^te^{-\int_0^s(C+2\mathbb{E}\|\phi(r)\|_H^{\theta}+\rho(\phi(r)))dr}
\nonumber\\
\!\!\!\!\!\!\!\!&&~~~~\cdot
\big(2{}_{V^*}\langle Y(s)-A(s,\phi(s),\mathcal{L}_{\phi(s)}),\phi(s)\rangle_V+2{}_{V^*}\langle A(s,\phi(s),\mathcal{L}_{\phi(s)}),X(s)\rangle_V
\nonumber\\
\!\!\!\!\!\!\!\!&&~~~~-\|B(s,\phi(s),\mathcal{L}_{\phi(s)})\|_{L_2(U,H)}^2+2\langle Z(s),B(s,\phi(s),\mathcal{L}_{\phi(s)})\rangle_{L_2(U,H)}
\nonumber\\
\!\!\!\!\!\!\!\!&&~~~~-2(C+C\mathbb{E}\|\phi(s)\|_H^{\theta}+\rho(\phi(s)))\langle X(s),\phi(s)\rangle_H
\nonumber\\
\!\!\!\!\!\!\!\!&&~~~~+(C+C\mathbb{E}\|\phi(s)\|_H^{\theta}+\rho(\phi(s)))\|\phi(s)\|_H^2\big)ds\Big)dt\Big].
\end{eqnarray}
On the other hand, by It\^{o}'s formula we know that for any $\phi\in \mathbb{K}\cap\mathbb{S}\cap L^p(\Omega;L^{\infty}([0,T];H))$,
\begin{eqnarray}\label{es30}
\!\!\!\!\!\!\!\!&&\mathbb{E}\Big[e^{-\int_0^t(C+C\mathbb{E}\|\phi(s)\|_H^{\theta}+\rho(\phi(s)))ds}\|X(t)\|_H^2-\|X(0)\|_H^2\Big]
\nonumber\\
=~~\!\!\!\!\!\!\!\!&&\mathbb{E}\Big[\int_0^te^{-\int_0^s(C+C\mathbb{E}\|\phi(r)\|_H^{\theta}+\rho(\phi(r)))dr}\big(2{}_{V^*}\langle Y(s),X(s)\rangle_V+\|Z(s)\|_{L_2(U,H)}^2
\nonumber\\
\!\!\!\!\!\!\!\!&&~~~~~~~~ -(C+C\mathbb{E}\|\phi(s)\|_H^{\theta}+\rho(\phi(s)))\|X(s)\|_H^2\big)ds\Big].
\end{eqnarray}
Now substituting (\ref{es30}) into (\ref{es31}), it turns out that
\begin{eqnarray}\label{es32}
\!\!\!\!\!\!\!\!&&\mathbb{E}\Big[\int_0^T\psi_t\Big(\int_0^te^{-\int_0^s(C+C\mathbb{E}\|\phi(r)\|_H^{\theta}+\rho(\phi(r)))dr}\big(\|B(s,\phi(s),\mathcal{L}_{\phi(s)})-Z(s)\|_{L_2(U,H)}^2
\nonumber\\
\!\!\!\!\!\!\!\!&&~~~~+2{}_{V^*}\langle Y(s)-A(s,\phi(s),\mathcal{L}_{\phi(s)}),X(s)-\phi(s)\rangle_V
\nonumber\\
\!\!\!\!\!\!\!\!&&
~~~~-(C+C\mathbb{E}\|\phi(s)\|_H^{\theta}+\rho(\phi(s)))\|X(s)-\phi(s)\|_H^2\big)ds\Big)dt\Big]
\nonumber\\
\!\!\!\!\!\!\!\!&&
\leq C\liminf_{n\to\infty}\mathbb{E}\int_0^T
\|X^{(n)}(s)-\phi(s)\|_H^2ds.
\end{eqnarray}
We claim that
\begin{equation}\label{es35}
\lim_{n\to\infty}\mathbb{E}\int_0^T
\|X^{(n)}(s)-X(s)\|_H^2ds=0,
\end{equation}
here selecting a subsequence if necessary.

In fact, by the convergence (\ref{conver1}), we can find a subsequence still denoted by $\{X^{(n)}\}_{n\in\mathbb{N}}$ such that
$$\lim_{n\to\infty} \|X^{(n)}(t)-X(t)\|_H=0,~dt\times \mathbb{P}\text{-a.e.},$$
then by (\ref{es20}), (\ref{es38}) and Vitali's convergence theorem we deduce that the claim holds.

Due to (\ref{es38}), it is straightforward that
$$X\in\mathbb{K}\cap\mathbb{S}\cap L^p(\Omega;L^{\infty}([0,T];H)).$$
Now by taking $\phi=X$ we can deduce from (\ref{es32}) and (\ref{es35}) that
$$Z=B(\cdot,X(\cdot),\mathcal{L}_{X(\cdot)}).$$
Moreover, letting $\phi=X-\varepsilon\tilde{\phi}v$ for any $\varepsilon>0$, $v\in V$ and $\tilde{\phi}\in L^\infty([0,T]\times\Omega;\mathbb{R})$,  it follows that
$$\mathbb{W}_{2,H}(\mathcal{L}_{X(s)},\mathcal{L}_{\phi(s)})^2\leq\mathbb{E}\|\varepsilon\tilde{\phi}(s)v\|_{H}^2\leq \varepsilon\|\tilde{\phi}\|_{\infty}^2\|v\|_{H}^2\downarrow0,~\text{as}~\varepsilon\downarrow0.$$
Then let $\varepsilon\to 0$ and by dominated convergence theorem,  $(\mathbf{H1})$ and the arbitrariness of $\psi$ and $\tilde{\phi}$, we infer that
$$Y=A(\cdot,X(\cdot),\mathcal{L}_{X(\cdot)}).$$
Thus we can conclude that $(X,W)$ is a weak solution of (\ref{eqSPDE}). Finally, similar to the proof of Proposition \ref{pro3}, by the modified Yamada-Watanabe theorem and the condition $(\mathbf{H2})$ we can get that Proposition \ref{pro5} follows.
\end{proof}

\subsection{Pathwise uniqueness}\label{sec3.4}
In this subsection, we devote to proving the pathwise uniqueness of solutions  to MVSPDE (\ref{eqSPDE}),  then combining with Proposition \ref{pro5} it implies Theorem \ref{th2}.

\begin{proposition}\label{pro6}
Suppose that all assumptions in Theorem \ref{th2} hold, then the pathwise uniqueness holds for solutions in the sense of Definition \ref{de2} provided satisfying (\ref{es37}).
\end{proposition}

\begin{proof}
Let $X(t),Y(t)$ be two solutions of (\ref{eqSPDE}) in the sense of Definition \ref{de2}, which satisfy (\ref{es37}), with initial random variable $\xi$.
Recall
\begin{eqnarray}
\!\!\!\!\!\!\!\!&&\mathbb{E}\Big[\sup_{t\in[0,T]}\|X(t)\|_H^p\Big]+\mathbb{E}\int_0^T\|X(t)\|_V^{\alpha}dt+\mathbb{E}\int_0^T\|X(t)\|_H^{p-2}\|X(t)\|_V^{\alpha}dt<\infty,\label{es53}
\\
\!\!\!\!\!\!\!\!&&\mathbb{E}\Big[\sup_{t\in[0,T]}\|Y(t)\|_H^p\Big]+\mathbb{E}\int_0^T\|Y(t)\|_V^{\alpha}dt+\mathbb{E}\int_0^T\|Y(t)\|_H^{p-2}\|Y(t)\|_V^{\alpha}dt<\infty.\label{es54}
\end{eqnarray}

\textbf{Case 1:} Assume that the condition $(\mathbf{H2}')$ holds. Define a stopping time
\begin{eqnarray*}
\tau_R:=~\!\!\!\!\!\!\!\!&&~~~\tau_R^{X}\wedge\tau_R^{Y}
\\
=~\!\!\!\!\!\!\!\!&&~\inf\Bigg\{t\in[0,T]:\Big\{\|X(t)\|_{H}+\int_0^t\|X(s)\|_{V}^{\alpha}ds\Big\}\vee\Big\{\|Y(t)\|_{H}+\int_0^t\|Y(s)\|_{V}^{\alpha}ds\Big\}\geq R\Bigg\}.
\end{eqnarray*}
Set
$$\phi(t):=C+C\mathbb{E}\|X(t)\|_{H}^{\theta}+C\mathbb{E}\|Y(t)\|_{H}^{\theta}+\rho_1(X(t))
+\rho_2(Y(t)).$$
Applying It\^{o}'s formula, by (\ref{es53}) and (\ref{es54}) we can get that for any $t\in[0,T]$,
\begin{eqnarray}\label{es111}
\!\!\!\!\!\!\!\!&&\|X(t\wedge\tau_R)-Y(t\wedge\tau_R)\|_{H}^2
\nonumber\\
\leq~\!\!\!\!\!\!\!\!&&~~\int_0^{t\wedge\tau_R}
\Big(2{}_{{V}^*}\langle A(s,X(s),\mathcal{L}_{X(s)})-A(s,Y(s),\mathcal{L}_{Y(s)}),X(s)-Y(s)\rangle_{V}
\nonumber\\
\!\!\!\!\!\!\!\!&&~~~~
+\|B(s,X(s),\mathcal{L}_{X(s)})-B(s,Y(s),\mathcal{L}_{Y(s)})\|_{L_2(U;H)}^2\Big)ds+2|\mathcal{M}(t\wedge\tau_R)|
\nonumber\\
\leq~\!\!\!\!\!\!\!\!&&~~\int_0^{t\wedge\tau_R}\phi(s)\big(\|X(s)-Y(s)\|_{H}^2+\mathbb{W}_{2,T,R,H}(\mathcal{L}_{X_s},\mathcal{L}_{Y_s})^2\big)ds
+2|\mathcal{M}(t\wedge\tau_R)|
\nonumber\\
\leq~\!\!\!\!\!\!\!\!&&~~~C\int_0^{t\wedge\tau_R}\phi(s)\Big\{\|X(s)-Y(s)\|_{H}^2
\nonumber\\
\!\!\!\!\!\!\!\!&&~~~~
+\mathbb{E}\Big[\sup_{r\in[0,s]}\|X(r\wedge\tau_R)-Y(r\wedge\tau_R)\|_{H}^2\Big]\Big\}ds+2|\mathcal{M}(t\wedge\tau_R)|,
\end{eqnarray}
where $\mathcal{M}(t)$ is a local martingale given by
$$\mathcal{M}(t):=\int_0^t\langle X(s)-Y(s), \big(B(s,X(s),\mathcal{L}_{X(s)})-B(s,Y(s),\mathcal{L}_{Y(s)})\big)dW_s\rangle_{H}.$$
By B-D-G's inequality, we have
\begin{eqnarray*}
\!\!\!\!\!\!\!\!&&~~\mathbb{E}\Big[\sup_{t\in[0,T]}|\mathcal{M}(t\wedge\tau_R)|\Big]
\nonumber\\
\leq~~\!\!\!\!\!\!\!\!&&~~
\mathbb{E}\Big[\int_0^{T\wedge\tau_R}\|B(t,X(t),\mathcal{L}_{X(t)})-B(t,Y(t),\mathcal{L}_{Y(t)})\|_{L_2(U,H)}^2\|X(t)-Y(t)\|_{H}^2dt\Big]^{\frac{1}{2}}
\nonumber\\
\leq~~\!\!\!\!\!\!\!\!&&~~\frac{1}{2}\mathbb{E}\Big[\sup_{t\in[0,T]}\|X(t\wedge\tau_R)-Y(t\wedge\tau_R)\|_{H}^2\Big]
\nonumber\\
\!\!\!\!\!\!\!\!&&~~+C\mathbb{E}\Big[\int_0^{T\wedge\tau_R}\big(\|X(t)-Y(t)\|_{H}^2+\mathbb{W}_{2,T,R,H}(\mathcal{L}_{X_t},\mathcal{L}_{Y_t})^2\big)dt\Big].
\end{eqnarray*}
Now taking $\sup_{t\in[0,T]}$ and expectation on both sides of (\ref{es111}), we have
\begin{eqnarray*}
\mathbb{E}\Big[\sup_{t\in[0,T\wedge\tau_R]}\|X(t)-Y(t)\|_{H}^2\Big]
\leq~~\!\!\!\!\!\!\!\!&&~~~C\mathbb{E}\int_0^{T\wedge\tau_R}\phi(t)\|X(t)-Y(t)\|_{H}^2dt
\nonumber\\
\!\!\!\!\!\!\!\!&&~~~
+C\mathbb{E}\int_0^{T}\phi(t)\mathbb{E}\Big[\sup_{s\in[0,t\wedge\tau_R]}\|X(s)-Y(s)\|_{H}^2\Big]dt.
\end{eqnarray*}
Combining stochastic Gronwall's lemma (cf.~\cite[Lemma A.1]{LLX}), (\ref{es53}) and (\ref{es54}), it follows that
\begin{eqnarray}\label{es23}
\!\!\!\!\!\!\!\!&&~~~\mathbb{E}\Big[\sup_{t\in[0,T\wedge\tau_R]}\|X(t)-Y(t)\|_{H}^2\Big]
\nonumber\\
\!\!\!\!\!\!\!\!&&~~~
\leq C_R\int_0^{T}\mathbb{E}\phi(t)\cdot\mathbb{E}\Big[\sup_{s\in[0,t\wedge\tau_R]}\|X(s)-Y(s)\|_{H}^2\Big]dt.
\end{eqnarray}
By (\ref{es22}), (\ref{es53}) and (\ref{es54}) we know
$$\int_0^{T}\mathbb{E}\phi(t)dt=\mathbb{E}\int_0^{T}\phi(t)dt<\infty,$$
thus (\ref{es23}) and Gronwall's lemma imply
\begin{equation*}\label{es114}
\mathbb{E}\Big[\sup_{t\in[0,T\wedge\tau_R]}\|X(t)-Y(t)\|_{H}^2\Big]\leq 0.
\end{equation*}
Finally, Fatou's lemma  implies that
\begin{equation*}
\mathbb{E}\Big[\sup_{t\in[0,T]}\|X(t)-Y(t)\|_{H}^2\Big]\leq \liminf_{R\to\infty}\mathbb{E}\Big[\sup_{t\in[0,T\wedge\tau_R]}\|X(t)-Y(t)\|_{H}^2\Big]=0.
\end{equation*}
We complete the proof.

\vspace{2mm}
\textbf{Case 2:} Assume that the condition $(\mathbf{H2}'')$ holds. Applying It\^{o}'s formula and the product rule, we can get that for any $t\in[0,T]$,
\begin{eqnarray*}
\!\!\!\!\!\!\!\!&&e^{-\int_0^t2C(1+\mathbb{E}\|X(s)\|_H^{\theta}+\mathbb{E}\|Y(s)\|_H^{\theta})ds}\|X(t)-Y(t)\|_H^2
\nonumber\\
\leq~~\!\!\!\!\!\!\!\!&&\int_0^te^{-\int_0^s2C(1+\mathbb{E}\|X(r)\|_H^{\theta}+\mathbb{E}\|Y(r)\|_H^{\theta})dr}
\nonumber\\
\!\!\!\!\!\!\!\!&&~~~~\cdot
\Big(2{}_{V^*}\langle A(s,X(s),\mathcal{L}_{X(s)})-A(s,Y(s),\mathcal{L}_{Y(s)}),X(s)-Y(s)\rangle_V
\nonumber\\
\!\!\!\!\!\!\!\!&&~~~~
+\|B(s,X(s),\mathcal{L}_{X(s)})-B(s,Y(s),\mathcal{L}_{Y(s)})\|_{L_2(U,H)}^2
\nonumber\\
\!\!\!\!\!\!\!\!&&~~~~
-2C(1+\mathbb{E}\|X(s)\|_H^{\theta}+\mathbb{E}\|Y(s)\|_H^{\theta})\|X(s)-Y(s)\|_H^2\Big)ds
\nonumber\\
\!\!\!\!\!\!\!\!&&+2\int_0^te^{-\int_0^s2C(1+\mathbb{E}\|X(r)\|_H^{\theta}+\mathbb{E}\|Y(r)\|_H^{\theta})dr}
\nonumber\\
\!\!\!\!\!\!\!\!&&~~~~\cdot
\langle X(s)-Y(s),\big(B(s,X(s),\mathcal{L}_{X(s)})-B(s,Y(s),\mathcal{L}_{Y(s)})\big)dW(s)\rangle_H,
\end{eqnarray*}
where the constant $C>0$ is the same as in the condition $(\mathbf{H2}'')$.

Then taking  expectation on both sides of the above inequality, by $(\mathbf{H2}'')$, (\ref{es53}) and (\ref{es54}) we have
\begin{equation*}
\mathbb{E}\Big\{e^{-\int_0^t2C(1+\mathbb{E}\|X(s)\|_H^{\theta}+\mathbb{E}\|Y(s)\|_H^{\theta})ds}\|X(t)-Y(t)\|_H^2\Big\}
\leq0.
\end{equation*}

Consequently, we deduce that $X(t)=Y(t),~\mathbb{P}\text{-a.s.},~t\in[0,T]$, thus the pathwise uniqueness follows from the path continuity on $H$.
\end{proof}

\vspace{2mm}
\textbf{Proof of Theorem \ref{th2}:} Combining Propositions \ref{pro5} and \ref{pro6}, (\ref{eqSPDE}) admits a unique solution in the sense of Definition \ref{de2} provided satisfying (\ref{es37}).
The proof is complete.

\section{Large deviation principle}\label{sec5}
In this section, the main aim is to investigate the small noise asymptotics (more precisely,  LDP) for the following MVSPDEs
\begin{equation}\label{eqsm}
dX^\varepsilon(t)=A(t,X^\varepsilon(t),\mathcal{L}_{X^\varepsilon(t)})dt+\sqrt{\varepsilon}B(t,X^\varepsilon(t),\mathcal{L}_{X^\varepsilon(t)})dW(t),~X^\varepsilon(0)=x\in H,
\end{equation}
 where  $\varepsilon>0$, $W(t)$ is an $U$-valued cylindrical Wiener process (the path
of $W$ take values in  $C([0,T];U_1)$, where $U_1$ is another
Hilbert space in which the embedding $U\subset U_1$ is
Hilbert--Schmidt) defined on $\left(\Omega,\mathcal{F},\{\mathcal{F}(t)\}_{t\geq 0},\mathbb{P}\right)$, the coefficients $A,B$ satisfy the assumptions in Theorem \ref{th2} throughout  this section. In order to obtain the LDP of (\ref{eqsm}), we also need the following time H\"{o}lder continuity of $B$.
\begin{enumerate}

\item [$({\mathbf{H}}{\mathbf{5}})$]\label{H5}
 $($Time H\"{o}lder continuity$)$ There exist $C,\gamma>0$  such that for any $u\in V$, $\mu\in\mathcal{P}_2(H)$ and $t,s\in[0,T]$,
\begin{equation*}
\|B(t,u,\mu)-B(s,u,\mu)\|_{L_2(U,H)}\leq C\big(1+\|u\|_H+\sqrt{\mu(\|\cdot\|_H^2)}\big)|t-s|^{\gamma}.
\end{equation*}
\end{enumerate}
We remark that if $B$ is time homogeneous, then  $(\mathbf{H5})$ is automatically satisfied. We recall some definitions in the theory of large deviations. Let $\{X^\varepsilon\}_{\varepsilon>0}$ denote a family of
random variables defined on a probability space
$(\Omega,\mathcal{F},\mathbb{P})$ and taking values in a Polish
space $\mathcal{E}$. The theory of large deviations  is
concerned with events $A$ for which the probability $\mathbb{P}(X^\varepsilon\in A)$  converges to zero exponentially fast
as $\varepsilon\to 0$.  The exponential decay rate of such probabilities is typically expressed by a  rate function.

\begin{definition}(Rate function) A function $I: \mathcal{E}\to [0,\infty]$ is called
a rate function if $I$ is lower semicontinuous. Moreover, a rate function $I$
is called a {\it good rate function} if  the level set $\{x\in \mathcal{E}: I(x)\le
K\}$ is compact for each constant $K<\infty$.
\end{definition}

\begin{definition}(Large deviation principle) The random variable family
 $\{X^\varepsilon\}_{\varepsilon>0}$ is said to satisfy
 the LDP on $\mathcal{E}$ with rate function
 $I$ if  the following lower and upper bound conditions hold,

(i) (Lower bound) for any open set $G\subset \mathcal{E}$,
$$\liminf_{\varepsilon\to 0}
   \varepsilon \log \mathbb{P}(X^{\varepsilon}\in G)\geq -\inf_{x\in G}I(x).$$

(ii) (Upper bound) for any closed set $F\subset \mathcal{E}$,
$$ \limsup_{\varepsilon\to 0}
   \varepsilon \log \mathbb{P}(X^{\varepsilon}\in F)\leq
  -\inf_{x\in F} I(x).
$$
\end{definition}

The Laplace principle is defined as follows (cf. \cite{DE,DZ}).

\begin{definition}(Laplace principle) The family $\{X^\varepsilon\}_{\varepsilon>0}$ is
said to satisfy the Laplace principle on $\mathcal{E}$ with a rate function
$I$ if for each bounded continuous real-valued function $h$ defined
on $\mathcal{E}$, we have
$$\lim_{\varepsilon\to 0}\varepsilon \log \mathbb{E}\left\lbrace
 \exp\left[-\frac{1}{\varepsilon} h(X^{\varepsilon})\right]\right\rbrace
= -\inf_{x\in E}\left\{h(x)+I(x)\right\}.$$
\end{definition}
It is known that if $\mathcal{E}$ is a Polish space and $I$ is a good rate function,  then the
LDP and Laplace principle are equivalent  because of the Varadhan's lemma \cite{V1} and Bryc's converse \cite{DZ}.

We are now in the position to establish the LDP for MVSPDE (\ref{eqsm}).
\begin{theorem}\label{th3}
Suppose that  all assumptions in Theorem \ref{th2} and $(\mathbf{H5})$
 hold. Then $\{X^\varepsilon\}_{\varepsilon>0}$
satisfies the LDP in $C([0,T]; H)$ with the
good rate function $I$ given by
\begin{equation*}
I(f):=\inf_{\left\{\phi\in L^2([0,T]; U):\  f=\mathcal{G}^0(\int_0^\cdot
\phi(s)ds)\right\}}\left\lbrace\frac{1}{2}
\int_0^T\|\phi(s)\|_U^2ds \right\rbrace,
\end{equation*}
where the measurable map $\mathcal{G}^0$ is defined by (\ref{eqsk}) below.
\end{theorem}

\begin{remark}
(i) To prove Theorem \ref{th3}, we want to utilize the well-known weak convergence approach developed by Budhiraja et al. (cf.~\cite{BD,BDM}).  The first step of  weak convergence approach is to find a measurable map such that the solution can be represented by a functional of Wiener process.  However, the classical Yamada-Watanabe theorem  does not hold in the McKean-Vlasov case  (cf.~\cite{WFY}). In order to overcome this difficulty, the decoupled method will be used to construct the aforementioned measurable map.

(ii) Compared to the existing works \cite{DST,HLL3,LSZZ} on the LDP for the MVSDE/SPDEs, the result of Theorem \ref{th3} is even new in the finite-dimensional case (i.e.~$V=H=V^*=\mathbb{R}^d$). For instance,  the authors in \cite{DST,LSZZ} considered the MVSDEs with locally Lipschitz but globally monotone drift (see \cite[Assumption 3.2]{DST} or \cite[(A1)]{LSZZ}), whereas here  we only impose  fully local assumption on the coefficients (i.e.~$(\mathbf{H2})$ and $(\mathbf{H2}')$).
\end{remark}

\subsection{Weak convergence approach}

Let
$$\mathcal{A}:=\left\lbrace \phi: \phi\  \text{is  $U$-valued
 $\mathcal{F}_t$-predictable process and}\
  \int_0^T\|\phi(s)\|^2_Uds<\infty \  \mathbb{P}\text{-a.s.}\right\rbrace, $$
  and
$$S_M:=\left\lbrace \phi\in L^2([0,T], U):
\int_0^T\|\phi(s)\|^2_{U}  ds\leq M
 \right\rbrace.$$
It is known that $S_M$ endowed with the weak topology is a Polish space. Let
$$\mathcal{A}_M:=\left\{\phi\in\mathcal{A}: \phi(\cdot)\in S_M, ~\mathbb{P}\text{-a.s.}\right\}.$$

  Now we state a sufficient condition introduced recently in \cite{MSZ} for the Laplace
 principle of $X^\varepsilon$, which is a modified form of the classical weak convergence criterion developed by Budhiraja et al.~\cite{BD,BDM}, and is more convenient to use in the current framework.

For any $\varepsilon>0$, suppose that $\mathcal{G}^\varepsilon: C([0,T]; U_1)\rightarrow
\mathcal{E}$ is a measurable map.

\vspace{1mm}
\textbf{Condition (A)}: There exists a measurable map $\mathcal{G}^0: C([0,T];
U_1)\rightarrow \mathcal{E}$ such that the following two conditions hold:

(i) Let $\{\phi^\varepsilon: \varepsilon>0\}\subset \mathcal{A}_M$ for
any $M<\infty$. For any $\delta>0$,
$$\lim_{\varepsilon\to 0}\mathbb{P}\Big(d\Big(\mathcal{G}^\varepsilon\big(\sqrt{\varepsilon}W_{\cdot}+\int_0^{\cdot}\phi^\varepsilon(s)ds\big),\mathcal{G}^0\big(\int_0^{\cdot}\phi^\varepsilon(s)ds\big)\Big)>\delta\Big)=0, $$
where $d(\cdot,\cdot)$ denotes the metric in the path space $\mathcal{E}$.

(ii) Let $\{\phi^n: n\in\mathbb{N}\}\subset S_M$ for any $M<\infty$ such that $\phi^n$ converges to element $\phi$ in $S_M$ as $n\to\infty$, then
$\mathcal{G}^0\big(\int_0^{\cdot}\phi^n(s)ds\big)$ converges to $\mathcal{G}^0\big(\int_0^{\cdot}\phi(s)ds\big)$ in the space $\mathcal{E}$.

\begin{lemma}\label{lem8}(\cite[Theorem 3.2]{MSZ})  If
$X^\varepsilon=\mathcal{G}^\varepsilon(\sqrt{\varepsilon}W_\cdot )$ and \textbf{Condition (A)}
hold, then $\{X^\varepsilon\}_{\varepsilon>0}$ satisfies the Laplace
principle (hence LDP) in $\mathcal{E}$ with the good
rate function $I$ given by
\begin{equation}\label{rf}
I(f):=\inf_{\left\{\phi\in L^2([0,T]; U):\  f=\mathcal{G}^0(\int_0^\cdot
\phi(s)ds)\right\}}\left\lbrace\frac{1}{2}
\int_0^T\|\phi(s)\|_U^2ds \right\rbrace,
\end{equation}
with the convention $\inf \emptyset:=\infty$.
\end{lemma}

Let us first explain the main idea of the proof of Freidlin-Wentzell's LDP for MVSPDE (\ref{eqsm}). Intuitively, when parameter $\varepsilon$ tends to $0$ in (\ref{eqsm}), the noise term vanishes and it reduces to the following limiting equation
\begin{equation}\label{eq8}
\frac{dX^0(t)}{dt}=A(t,X^0(t),\mathcal{L}_{X^0(t)}),~X^0(0)=x\in H,
\end{equation}
where the solution $X^0(t)$ is a deterministic path and its law $\mathcal{L}_{X^0(t)}=\delta_{X^0(t)}$. By Theorem \ref{th2} we can deduce that (\ref{eq8}) admits a unique strong solution  satisfying $X^0\in C([0,T];H)$.

Let $\mu^\varepsilon(t)=\mathcal{L}_{X^\varepsilon(t)}$,  it is easy to see that $X^\varepsilon(t)$ also solves the following decoupled SDE (not distribution dependent)
\begin{equation}\label{eqde}
dX^\varepsilon(t)=A^{\mu^\varepsilon}(t,X^\varepsilon(t))dt+\sqrt{\varepsilon}B^{\mu^\varepsilon}(t,X^\varepsilon(t))dW(t),
\end{equation}
where we denote $A^{\mu}(\cdot,\cdot)=A(\cdot,\cdot,\mu)$, $B^{\mu}(\cdot,\cdot)=B(\cdot,\cdot,\mu)$ for any $\mu\in\mathcal{P}(H)$.  By the assumptions of Theorem \ref{th3}, (\ref{eqde}) admits a unique strong solution by \cite{LR2}, meanwhile, we can get that $X^\varepsilon\in C([0,T];H)$, $\mathbb{P}$-a.s.. Hence, thanks to the infinite-dimensional version of Yamada-Watanabe theorem \cite{RSZ}, there exists a measurable map $\mathcal{G}_{\mu^\varepsilon}:C([0,T];U_1)\to \mathcal{E}$, where we denote $\mathcal{E}=C([0,T];H)$ throughout this section, such that the solution of (\ref{eqde}) could be represented by
\begin{equation*}
X^\varepsilon=\mathcal{G}_{\mu^\varepsilon}(\sqrt{\varepsilon}W({\cdot})).
\end{equation*}
For convenience  we denote $\mathcal{G}^\varepsilon=\mathcal{G}_{\mu^\varepsilon}$.  Then for any $\phi^\varepsilon\in\mathcal{A}_M$, we consider the process
$$X^{\varepsilon,\phi^\varepsilon}=\mathcal{G}^\varepsilon\Big(\sqrt{\varepsilon}W({\cdot})+\int_0^{\cdot}\phi^\varepsilon(s)ds\Big),$$
then  it solves the following control problem
\begin{eqnarray}\label{eqc}
dX^{\varepsilon,\phi^\varepsilon}(t)=~~\!\!\!\!\!\!\!\!&&A(t,X^{\varepsilon,\phi^\varepsilon}(t),\mathcal{L}_{X^{\varepsilon}(t)})dt+B(t,X^{\varepsilon,\phi^\varepsilon}(t),\mathcal{L}_{X^{\varepsilon}(t)})\phi^\varepsilon(t)dt
\nonumber\\
\!\!\!\!\!\!\!\!&&
+\sqrt{\varepsilon}B(t,X^{\varepsilon,\phi^\varepsilon}(t),\mathcal{L}_{X^{\varepsilon}(t)})dW(t),~X^{\varepsilon,\phi^\varepsilon}(0)=x,
\end{eqnarray}
where  $\mathcal{L}_{X^{\varepsilon}(t)}$ is the law of solution of (\ref{eqsm}). By a standard argument, we know that (\ref{eqc}) admits a unique solution following from the Girsanov's transformation.

Now we introduce the skeleton equation associated with (\ref{eqsm}) as follows,
\begin{equation}\label{eqsk}
\frac{d}{dt}\bar{X}^{\phi}(t)=A(t,\bar{X}^{\phi}(t),\mathcal{L}_{X^{0}(t)})+B(t,\bar{X}^{\phi}(t),\mathcal{L}_{X^{0}(t)})\phi(t),~\bar{X}^{\phi}(0)=x,
\end{equation}
where $\phi\in L^2([0,T];U)$, $\mathcal{L}_{X^{0}(t)}$ is the Dirac measure of the solution of (\ref{eq8}). The existence and uniqueness of solutions to (\ref{eqsk}) will be proved later, which implies that there exists a  map $\mathcal{G}^0:C([0,T];U_1)\to \mathcal{E}$ such that
\begin{equation*}\label{g1}
\mathcal{G}^0(h):=\left\{ \begin{aligned}
&\bar{X}^{\phi},~~\text{if}~h=\int_0^{\cdot}\phi(s)ds~\text{for some}~\phi\in L^2([0,T];U);\\
&0,~~~~\text{otherwise}.
\end{aligned} \right.
\end{equation*}
In the sequel, we aim to prove that the aforementioned measurable maps $\mathcal{G}^\varepsilon$ and $\mathcal{G}^0$ satisfy  \textbf{Condition (A)}.

\subsection{A priori estimates}
In this subsection, we derive the existence and uniqueness of solutions to the skeleton equation (\ref{eqsk}). Some necessary a priori estimates are also obtained.

\begin{lemma}
Suppose that all assumptions in Theorem  \ref{th3} hold. For every $x\in H$ and $\phi\in L^2([0,T];U)$, there exists a unique solution $\{\bar{X}^{\phi}(t)\}_{t\in[0, T]}$ to (\ref{eqsk}). Moreover, there exists $C_{T,M,x}>0$,
\begin{equation}\label{es33}
\sup_{\phi\in S_M}\Big\{\sup_{t\in[0,T]}\|\bar{X}^{\phi}(t)\|_{H}^2+\int_0^T\|\bar{X}^{\phi}(t)\|_{V}^{\alpha}dt\Big\}\leq C_{T,M,x}.
\end{equation}
Furthermore, for any $p\in(\eta,\infty)\cap[\beta+2,\infty)$, we have
\begin{equation}\label{es51}
\sup_{\phi\in S_M}\Big\{\sup_{t\in[0,T]}\|\bar{X}^{\phi}(t)\|_{H}^p+\int_0^T\|\bar{X}^{\phi}(t)\|_{H}^{p-2}\|\bar{X}^{\phi}(t)\|_{V}^{\alpha}dt\Big\}\leq C_{T,M,x}.
\end{equation}
\end{lemma}

\begin{proof}
Let $\mu^0(t)=\mathcal{L}_{X^{0}(t)}$. Note that (\ref{eqsk}) is equivalent to
\begin{equation}\label{eq11}
\frac{d}{dt}\bar{X}^{\phi}(t)=A^{\mu^0}(t,\bar{X}^{\phi}(t))+B^{\mu^0}(t,\bar{X}^{\phi}(t))\phi(t),~\bar{X}^{\phi}(0)=x,
\end{equation}
where $A^{\mu^0},B^{\mu^0}$ is defined as in (\ref{eqde}). It is easy to infer that $A^{\mu^0},B^{\mu^0}$  satisfy the assumptions in \cite[Theorem 2.1]{LTZ}, then the existence and uniqueness of solutions to (\ref{eq11}) follows directly from \cite[Lemma 3.1]{LTZ},  hence (\ref{eqsk}) also admits a unique solution. It suffices to prove the uniform estimates  (\ref{es33}) and (\ref{es51}).

By the  integration by parts formula and $(\mathbf{H3})$-$(\mathbf{H4})$, we deduce that
\begin{eqnarray*}
\frac{d}{dt}\|\bar{X}^{\phi}(t)\|_{H}^2=~~\!\!\!\!\!\!\!\!&&2{}_{V^*}\langle A(t,\bar{X}^{\phi}(t),\mathcal{L}_{X^{0}(t)}),\bar{X}^{\phi}(t)\rangle_{V}+2\langle B(t,\bar{X}^{\phi}(t),\mathcal{L}_{X^{0}(t)})\phi(t),\bar{X}^{\phi}(t)\rangle_{H}
\nonumber\\
\leq~~\!\!\!\!\!\!\!\!&&-\delta\|\bar{X}^{\phi}(t)\|_{V}^{\alpha}+C(1+\|\phi(t)\|_U^2)\|\bar{X}^{\phi}(t)\|_{H}^2+C\big(1+\mathcal{L}_{X^{0}(t)}(\|\cdot\|_H^2)\big).
\end{eqnarray*}
Note that $X^0\in C([0,T];H)$ and $\mathcal{L}_{X^{0}(t)}(\|\cdot\|_H^2)=\|X^0(t)\|_H^2$. For any $\phi\in S_M$,  by applying Gronwall' lemma we have
\begin{eqnarray*}
\!\!\!\!\!\!\!\!&&\sup_{t\in[0,T]}\|\bar{X}^{\phi}(t)\|_{H}^2+\delta\int_0^T\|\bar{X}^{\phi}(t)\|_{V}^{\alpha}dt
\nonumber\\
\leq~~\!\!\!\!\!\!\!\!&&C_T\exp\Big\{\int_0^T\big(1+\|\phi(t)\|_U^2\big)dt\Big\}\big(1+\|x\|_{H}^2+\sup_{t\in[0,T]}\|X^0(t)\|_H^2\big)
\nonumber\\
\leq~~\!\!\!\!\!\!\!\!&&C_{T,M}\big(1+\|x\|_{H}^2+\sup_{t\in[0,T]}\|X^0(t)\|_H^2\big).
\end{eqnarray*}
In order to prove (\ref{es51}),  note that by $(\mathbf{H3})$ and $X^0\in C([0,T];H)$,
\begin{eqnarray*}
\|\bar{X}^{\phi}(t)\|_H^p
=~~\!\!\!\!\!\!\!\!&&\|x\|_H^p
+p\int_0^t\|\bar{X}^{\phi}(s)\|_H^{p-2}{}_{V^*}\langle A(s,\bar{X}^{\phi}(s),\mathcal{L}_{X^{0}(s)}),\bar{X}^{\phi}(s)\rangle_Vds
\nonumber \\
\!\!\!\!\!\!\!\!&&+p\int_0^t\|\bar{X}^{\phi}(s)\|_H^{p-2}\langle B(s,\bar{X}^{\phi}(s),\mathcal{L}_{X^{0}(s)})\phi(s),\bar{X}^{\phi}(s)\rangle_Hds
\nonumber \\
\leq~~\!\!\!\!\!\!\!\!&&C(1+\|x\|_H^p)-\frac{p\delta}{2}\int_0^t\|\bar{X}^{\phi}(s)\|_H^{p-2}\|\bar{X}^{\phi}(s)\|_V^{\alpha}ds
\nonumber \\
\!\!\!\!\!\!\!\!&&
+C_p\int_0^t\|\bar{X}^{\phi}(s)\|_H^{p-2}\big(1+\|\bar{X}^{\phi}(s)\|_H^2+\|X^{0}(s)\|_H^2\big)ds
\nonumber \\
\!\!\!\!\!\!\!\!&&
+C_p\int_0^t\|\bar{X}^{\phi}(s)\|_H^{p-1}\big(1+\|\bar{X}^{\phi}(s)\|_H+\|X^{0}(s)\|_H\big)\|\phi(s)\|_Uds
\nonumber \\
\leq~~\!\!\!\!\!\!\!\!&&C_T(1+\|x\|_H^p)+C_p\int_0^t\|\bar{X}^{\phi}(s)\|_H^{p}\|\phi(s)\|_Uds+C_{p,T}\Big(\int_0^t\|\phi(s)\|_U^2ds\Big)^{\frac{1}{2}}.
\end{eqnarray*}
Then due to the definition of $S_M$ and Gronwall's lemma, it is easy to show that
\begin{equation*}
\sup_{\phi\in S_M}\Big\{\sup_{t\in[0,T]}\|\bar{X}^{\phi}(t)\|_{H}^p+\int_0^T\|\bar{X}^{\phi}(t)\|_{H}^{p-2}\|\bar{X}^{\phi}(t)\|_{V}^{\alpha}dt\Big\}\leq C_{T,M}(1+\|x\|_H^p).
\end{equation*}
The proof is complete.
\end{proof}

The following lemma will play an important role in proving  \textbf{Condition (A)}(i).
\begin{lemma}\label{lem6}
There exists a constant $C_T>0$ such that for any $\varepsilon>0$,
 $$\mathbb{E}\Big[\sup_{t\in[0,T]}\|X^{\varepsilon}(t)-X^0(t)\|_H^2\Big]\leq C_T\varepsilon.$$
\end{lemma}
\begin{proof}
Let $Z^\varepsilon(t)=X^{\varepsilon}(t)-X^0(t)$. Using It\^{o}'s formula gives that
\begin{eqnarray}\label{es34}
\|Z^\varepsilon(t)\|_H^2=~~\!\!\!\!\!\!\!\!&&2\int_0^t{}_{V^*}\langle A(s,X^{\varepsilon}(s),\mathcal{L}_{X^{\varepsilon}(s)})-A(s,X^0(s),\mathcal{L}_{X^0(s)}),Z^\varepsilon(s)\rangle_Vds
\nonumber\\
\!\!\!\!\!\!\!\!&&~~+2\sqrt{\varepsilon}\int_0^t\langle B(s,X^{\varepsilon}(s),\mathcal{L}_{X^{\varepsilon}(s)})dW(s),Z^\varepsilon(s)\rangle_H
\nonumber\\
\!\!\!\!\!\!\!\!&&
~~+\varepsilon\int_0^t\|B(s,X^{\varepsilon}(s),\mathcal{L}_{X^{\varepsilon}(s)})\|_{L_2(U,H)}^2ds
\nonumber\\
=:~~\!\!\!\!\!\!\!\!&&\sum_{i=1}^3 \mathcal{I}_i(t).
\end{eqnarray}
By $(\mathbf{H2})$, for $\mathcal{I}_1(t)$ we have
\begin{eqnarray*}
\mathbb{E}\Big[\sup_{t\in[0,T]}\mathcal{I}_1(t)\Big]\leq~~\!\!\!\!\!\!\!\!&& \mathbb{E}\int_0^T\big(C+\rho(X^0(t))+C\|X^0(t)\|_H^{\theta}\big)\|Z^\varepsilon(t)\|_H^2dt
\nonumber\\
\!\!\!\!\!\!\!\!&&
~~~~+C\int_0^T(1+\|X^0(t)\|_H^{\theta})\mathbb{E}\|Z^\varepsilon(t)\|_H^2dt
\nonumber\\
\leq~~\!\!\!\!\!\!\!\!&&C\int_0^T\big(1+\rho(X^0(t))+\|X^0(t)\|_H^{\theta}\big)\mathbb{E}\|Z^\varepsilon(t)\|_H^2dt,
\end{eqnarray*}
where we used the fact that $X^0(t)$ is deterministic. In addition, we recall (\ref{es37}) that $X^0(t),X^\varepsilon(t)$ satisfy the following energy estimates, respectively,
\begin{equation}\label{es36}
\sup_{t\in[0,T]}\|X^0(t)\|_H+\int_0^T\|X^0(t)\|_V^{\alpha}dt<\infty
\end{equation}
and
\begin{equation}\label{es42}
\mathbb{E}\Big[\sup_{t\in[0,T]}\|X^\varepsilon(t)\|_H^2\Big]+\mathbb{E}\int_0^T\|X^\varepsilon(t)\|_V^{\alpha}dt<\infty.
\end{equation}

As for $\mathcal{I}_3(t)$, we have the following control
\begin{eqnarray*}
\mathbb{E}\Big[\sup_{t\in[0,T]}\mathcal{I}_3(t)\Big]\leq~~\!\!\!\!\!\!\!\!&&C\varepsilon\mathbb{E}\int_0^T(1+\|X^{\varepsilon}(t)\|_H^2+\mathbb{E}\|X^{\varepsilon}(t)\|_H^2)dt
\nonumber\\
\leq~~\!\!\!\!\!\!\!\!&&C_T\varepsilon.
\end{eqnarray*}
For the term $\mathcal{I}_2(t)$, using Burkholder-Davis-Gundy's inequality we have
\begin{eqnarray}\label{es39}
\mathbb{E}\Big[\sup_{t\in[0,T]}\mathcal{I}_2(t)\Big]\leq~~\!\!\!\!\!\!\!\!&&
 C\sqrt{\varepsilon}\mathbb{E}\Big[\int_0^T\|B(t,X^{\varepsilon}(t),\mathcal{L}_{X^{\varepsilon}(t)})\|_{L_2(U,H)}^2\|Z^\varepsilon(t)\|_H^2dt\Big]^{\frac{1}{2}}
\nonumber\\
\leq~~\!\!\!\!\!\!\!\!&&\frac{1}{2}\mathbb{E}\Big[\sup_{t\in[0,T]}\|Z^\varepsilon(t)\|_H^2\Big]+C_T\varepsilon.
\end{eqnarray}
In view of (\ref{es34})-(\ref{es39}), by Gronwall's lemma and (\ref{es22}) we obtain that
\begin{eqnarray*}
\mathbb{E}\Big[\sup_{t\in[0,T]}\|Z^{\varepsilon}(t)\|_H^2\Big]
\leq~~ \!\!\!\!\!\!\!\!&&C_T\varepsilon\exp\Big\{\int_0^T\big(1+\rho(X^0(t))+\|X^0(t)\|_H^{\theta}\big)dt\Big\}
\nonumber\\
\leq~~\!\!\!\!\!\!\!\!&&C_T\varepsilon.
\end{eqnarray*}
Then we conclude the desired estimate.
\end{proof}

\vspace{3mm}
To prove \textbf{Condition (A)} (ii), the following lemma is needed for the estimate of integral of time
increment to the solution of skeleton equation (\ref{eqsk}).
\begin{lemma}\label{lem7}
For any $x\in H$ and $\phi\in S_M$, there exists a constant $C_{T,M,x}>0$ such that
$$\int_0^{T}\|\bar{X}^{\phi}(t)-\bar{X}^{\phi}(t(\Delta))\|_{H}^2dt\leq C_{T,M,x}\Delta,$$
where $\Delta>0$ is a small enough constant and $t(\Delta):=[\frac{t}{\Delta}]\Delta$  (here $[s]$ denote the largest integer smaller than $s$).
\end{lemma}
\begin{proof}
It is easy to see that
\begin{eqnarray}\label{14}
\!\!\!\!\!\!\!\!&&\int_0^{T}\|\bar{X}^{\phi}(t)-\bar{X}^{\phi}(t(\Delta))\|_{H}^2dt
\nonumber\\
\leq~~\!\!\!\!\!\!\!\!&&\int_0^{\Delta}\|\bar{X}^{\phi}(t)-x\|_{H}^2dt+
\int_{\Delta}^{T}\|\bar{X}^{\phi}(t)-\bar{X}^{\phi}(t(\Delta))\|_{H}^2dt
\nonumber\\
\leq~~\!\!\!\!\!\!\!\!&&2\Delta\|x\|_{H}^2+2\Delta\sup_{\phi\in S_M}\Big\{\sup_{t\in[0,T]}\|\bar{X}^{\phi}(t)\|_H^2\Big\}+2\int_{\Delta}^{T}\|\bar{X}^{\phi}(t)-\bar{X}^{\phi}(t-\Delta)\|_{H}^2dt
\nonumber\\
\!\!\!\!\!\!\!\!&&+2\int_{\Delta}^{T}\|\bar{X}^{\phi}(t(\Delta))-\bar{X}^{\phi}(t-\Delta)\|_{H}^2dt
\nonumber\\
\leq~~\!\!\!\!\!\!\!\!&&C_{T,M}\Delta(1+\|x\|_{H}^2)+2\int_{\Delta}^{T}\|\bar{X}^{\phi}(t)-\bar{X}^{\phi}(t-\Delta)\|_{H}^2dt
\nonumber\\
\!\!\!\!\!\!\!\!&&+2\int_{\Delta}^{T}\|\bar{X}^{\phi}(t(\Delta))-\bar{X}^{\phi}(t-\Delta)\|_{H}^2dt.
\end{eqnarray}
For the first integral term on the right side of (\ref{14}), the integration by parts formula yields that
\begin{eqnarray}\label{15}
\!\!\!\!\!\!\!\!&&\|\bar{X}^{\phi}(t)-\bar{X}^{\phi}(t-\Delta)\|_{H}^2
\nonumber\\
=~~\!\!\!\!\!\!\!\!&&2\int_{t-\Delta}^{t}{}_{V^*}\langle A(s,\bar{X}^{\phi}(s),\mathcal{L}_{X^0(s)}),\bar{X}^{\phi}(s)-\bar{X}^{\phi}(t-\Delta)\rangle_Vds
\nonumber\\
\!\!\!\!\!\!\!\!&&+2\int_{t-\Delta}^{t}\langle B(s,\bar{X}^{\phi}(s),\mathcal{L}_{X^0(s)})\phi(s),\bar{X}^{\phi}(s)-\bar{X}^{\phi}(t-\Delta)\rangle_Hds
\nonumber\\
=:~~\!\!\!\!\!\!\!\!&&\mathcal{R}_1(t)+\mathcal{R}_2(t).
\end{eqnarray}
Now we  estimate $\int_\Delta^{T}\mathcal{R}_i(t)dt$, $i=1,2$, respectively.  By $(\mathbf{H4})$ and H\"{o}lder's inequality we infer that
\begin{eqnarray}\label{16}
\!\!\!\!\!\!\!\!&&\int_\Delta^{T}\mathcal{R}_1(t)dt
\nonumber\\
\leq~~\!\!\!\!\!\!\!\!&&2\Big[\int_\Delta^{T}\int_{t-\Delta}^{t}\|A(s,\bar{X}^{\phi}(s),\mathcal{L}_{X^0(s)})\|_{V^*}^{\frac{\alpha}{\alpha-1}}dsdt\Big]^{\frac{\alpha-1}{\alpha}}
\nonumber\\
\!\!\!\!\!\!\!\!&&~~\cdot
\Big[\int_\Delta^{T}\int_{t-\Delta}^{t}\|\bar{X}^{\phi}(s)-\bar{X}^{\phi}(t-\Delta)\|_{V}^{\alpha}dsdt\Big]^{\frac{1}{\alpha}}
\nonumber\\
\leq~~\!\!\!\!\!\!\!\!&&C\Big[\Delta\int_0^{T}\big(1+\|\bar{X}^{\phi}(s)\|_V^{\alpha}+\|X^0(s)\|_H^{\theta}\big)\big(1+\|\bar{X}^{\phi}(s)\|_H^{\beta}+\|X^0(s)\|_H^{\theta}\big)ds\Big]^{\frac{\alpha-1}{\alpha}}
\nonumber\\
\!\!\!\!\!\!\!\!&&~~\cdot
\Big[\Delta\int_0^{T}\|\bar{X}^{\phi}(s)\|_V^{\alpha}ds\Big]^{\frac{1}{\alpha}}
\nonumber\\
\leq~~\!\!\!\!\!\!\!\!&&C_{T,M}\Delta,
\end{eqnarray}
where we used the estimates (\ref{es51}) and (\ref{es36}) in the last step.

Analogously, by $(\mathbf{H4})$ we have
\begin{eqnarray}\label{17}
\!\!\!\!\!\!\!\!&&\int_\Delta^{T}\mathcal{R}_2(t)dt
\nonumber\\
\leq~~\!\!\!\!\!\!\!\!&&2\Big[\int_\Delta^{T}\int_{t-\Delta}^{t}\|B(s,\bar{X}^{\phi}(s),\mathcal{L}_{X^0(s)})\|_{L_2(U,H)}^2\|\phi(s)\|_U^2dsdt\Big]^{\frac{1}{2}}
\nonumber\\
\!\!\!\!\!\!\!\!&&~~\cdot
\Big[\int_\Delta^{T}\int_{t-\Delta}^{t}\|\bar{X}^{\phi}(s)-\bar{X}^{\phi}(t-\Delta)\|_H^2dsdt\Big]^{\frac{1}{2}}
\nonumber\\
\leq~~\!\!\!\!\!\!\!\!&&C\Big[\Delta\int_0^{T}\big(1+\|\bar{X}^{\phi}(s)\|_H^2+\|X^0(s)\|_H^2\big)\|\phi(s)\|_U^2ds\Big]^{\frac{1}{2}}
\Big[\Delta\int_0^{T}\|\bar{X}^{\phi}(s)\|_H^2ds\Big]^{\frac{1}{2}}
\nonumber\\
\leq~~\!\!\!\!\!\!\!\!&&C_T\Delta\Big[\Big(1+\sup_{\phi\in S_M}\Big\{\sup_{t\in[0,T]}\|\bar{X}^{\phi}(t)\|_H^2\Big\}+\sup_{t\in[0,T]}\|X^0(t)\|_H^2\Big)\int_0^{T}\|\phi(s)\|_U^2ds\Big]^{\frac{1}{2}}
\nonumber\\
\!\!\!\!\!\!\!\!&&~~\cdot\Big[\int_0^{T}\|\bar{X}^{\phi}(s)\|_H^2ds\Big]^{\frac{1}{2}}
\nonumber\\
\leq~~\!\!\!\!\!\!\!\!&&C_{T,M}\Delta.
\end{eqnarray}
Combining (\ref{16})-(\ref{17}) with (\ref{15}), it implies that
\begin{equation}\label{18}
\int_\Delta^{T}\|\bar{X}^{\phi}(t)-\bar{X}^{\phi}(t-\Delta)\|_{H}^2dt\leq C_{T,M}\Delta.
\end{equation}
By a similar calculation as the proof of (\ref{18}), it is easy to deduce that
\begin{equation}\label{19}
\int_\Delta^{T}\|\bar{X}^{\phi}(t(\Delta))-\bar{X}^{\phi}(t-\Delta)\|_{H}^2dt\leq C_{T,M}\Delta.
\end{equation}
Finally, combining (\ref{18})-(\ref{19}) with (\ref{14}), we get the desired estimate.
\end{proof}
\subsection{Proof of LDP}
In this subsection, we devote to proving  that $\{X^\varepsilon\}_{\varepsilon>0}$ satisfies the LDP in $C([0,T]; H)$. In fact, it suffices to  prove \textbf{Condition (A)} (i) and (ii). The proof of \textbf{Condition (A)} (i) will be given in Proposition \ref{pro7} and \textbf{Condition (A)} (ii) will be established in Proposition \ref{pro8}.

\begin{proposition}\label{pro7}
Suppose that all assumptions in Theorem \ref{th3} hold. Let $\{\phi^\varepsilon: \varepsilon>0\}\subset \mathcal{A}_M$ for
any $M<\infty$. Then for any $\delta>0$ we have
$$\lim_{\varepsilon\to 0}\mathbb{P}\Big(\sup_{t\in[0,T]}\|X^{\varepsilon,\phi^\varepsilon}(t)-\bar{X}^{\phi^\varepsilon}(t)\|_H>\delta\Big)=0,$$
where $\bar{X}^{\phi^\varepsilon}$ is the solution of (\ref{eqsk}) with $\phi^\varepsilon$ replacing $\phi$.
\end{proposition}

\begin{proof}
Recall that $Y^{\varepsilon}(t):=X^{\varepsilon,\phi^\varepsilon}(t)-\bar{X}^{\phi^\varepsilon}(t)$ satisfies
\begin{eqnarray*}
\left\{ \begin{aligned}
dY^{\varepsilon}(t)=&\big[A(t,X^{\varepsilon,\phi^\varepsilon}(t),\mathcal{L}_{X^{\varepsilon}(t)})-A(t,\bar{X}^{\phi^\varepsilon}(t),\mathcal{L}_{X^0(t)})\big]dt\\
&+\big[B(t,X^{\varepsilon,\phi^\varepsilon}(t),\mathcal{L}_{X^{\varepsilon}(t)})-B(t,\bar{X}^{\phi^\varepsilon}(t),\mathcal{L}_{X^0(t)})\big]\phi^\varepsilon(t)dt\\
&+\sqrt{\varepsilon}B(t,X^{\varepsilon,\phi^\varepsilon}(t),\mathcal{L}_{X^{\varepsilon}(t)})dW(t),\\
Y^{\varepsilon}(0)=0&.
\end{aligned}\right.
\end{eqnarray*}
Using It\^{o}'s formula, we deduce that
\begin{eqnarray}\label{es40}
\|Y^{\varepsilon}(t)\|_H^2=~~\!\!\!\!\!\!\!\!&&2\int_0^t{}_{V^*}\langle A(s,X^{\varepsilon,\phi^\varepsilon}(s),\mathcal{L}_{X^{\varepsilon}(s)})-A(s,\bar{X}^{\phi^\varepsilon}(s),\mathcal{L}_{X^0(s)}),Y^{\varepsilon}(s)\rangle_{V}ds
\nonumber\\
\!\!\!\!\!\!\!\!&&~~~+2\int_0^t\langle\big[B(s,X^{\varepsilon,\phi^\varepsilon}(s),\mathcal{L}_{X^{\varepsilon}(s)})-B(s,\bar{X}^{\phi^\varepsilon}(s),\mathcal{L}_{X^0(s)})\big]\phi^\varepsilon(s),Y^{\varepsilon}(s)\rangle_{H}ds
\nonumber\\
\!\!\!\!\!\!\!\!&&~~~+\varepsilon\int_0^t
\|B(s,X^{\varepsilon,\phi^\varepsilon}(s),\mathcal{L}_{X^{\varepsilon}(s)})\|_{L_2(U,H)}^2ds
\nonumber\\
\!\!\!\!\!\!\!\!&&~~~
+2\sqrt{\varepsilon}\int_0^t\langle B(s,X^{\varepsilon,\phi^\varepsilon}(s),\mathcal{L}_{X^{\varepsilon}(s)})dW(s),Y^{\varepsilon}(s)\rangle_H
\nonumber\\
=:~~\!\!\!\!\!\!\!\!&&\sum_{i=1}^{4}\mathcal{J}_i(t).
\end{eqnarray}
Now we estimate the terms $\mathcal{J}_i(t)$, $i=1,2,3,4$, one by one.

First, for $\mathcal{J}_1(t), \mathcal{J}_2(t)$, by $(\mathbf{H2})$ we have
\begin{eqnarray*}
\!\!\!\!\!\!\!\!&&\mathcal{J}_1(t)+\mathcal{J}_2(t)
\nonumber\\
\leq~~\!\!\!\!\!\!\!\!&&2\int_0^t{}_{V^*}\langle A(s,X^{\varepsilon,\phi^\varepsilon}(s),\mathcal{L}_{X^{\varepsilon}(s)})-A(s,\bar{X}^{\phi^\varepsilon}(s),\mathcal{L}_{X^0(s)}),Y^{\varepsilon}(s)\rangle_{V}ds
\nonumber\\
\!\!\!\!\!\!\!\!&&~~~
+\int_0^t\|B(s,X^{\varepsilon,\phi^\varepsilon}(s),\mathcal{L}_{X^{\varepsilon}(s)})-B(s,\bar{X}^{\phi^\varepsilon}(s),\mathcal{L}_{X^0(s)})\|_{L_2(U,H)}^2ds
\nonumber\\
\!\!\!\!\!\!\!\!&&~~~
+\int_0^t\|\phi^\varepsilon(s)\|_U^2\|Y^{\varepsilon}(s)\|_H^2ds
\nonumber\\
\leq~~\!\!\!\!\!\!\!\!&&\int_0^t\big(C+\rho(\bar{X}^{\phi^\varepsilon}(s))+C\|X^0(s)\|_H^{\theta}\big)\|Y^{\varepsilon}(s)\|_H^2ds
\nonumber\\
\!\!\!\!\!\!\!\!&&~~~+C\int_0^t(1+\|X^0(s)\|_H^{\theta})\mathbb{E}\|X^{\varepsilon}(s)-X^0(s)\|_H^2ds+\int_0^t\|\phi^\varepsilon(s)\|_U^2\|Y^{\varepsilon}(s)\|_H^2ds
\nonumber\\
\leq~~\!\!\!\!\!\!\!\!&&C\int_0^t\big(1+\rho(\bar{X}^{\phi^\varepsilon}(s))+\|X^0(s)\|_H^{\theta}\big)\|Y^{\varepsilon}(s)\|_H^2ds
\nonumber\\
\!\!\!\!\!\!\!\!&&~~~+
\int_0^t\|\phi^\varepsilon(s)\|_U^2\|Y^{\varepsilon}(s)\|_H^2ds+C_T\varepsilon,
\end{eqnarray*}
where we used Lemma \ref{lem6} in the last step.

Then applying Gronwall's lemma to (\ref{es40}) we obtain that
\begin{eqnarray}\label{es49}
\|Y^{\varepsilon}(t)\|_H^2\leq~~\!\!\!\!\!\!\!\!&&\Big\{C_T\varepsilon+C\int_0^t\big(1+\rho(\bar{X}^{\phi^\varepsilon}(s))+\|X^0(s)\|_H^{\theta}\big)\|Y^{\varepsilon}(s)\|_H^2ds      +\mathcal{J}_3(t)+\mathcal{J}_4(t)\Big\}
\nonumber\\
\!\!\!\!\!\!\!\!&&~~~~~\cdot
\exp\Big\{\int_0^t\|\phi^\varepsilon(s)\|_U^2ds\Big\}.
\end{eqnarray}
Using a similar argument as the proof of (\ref{es33}), we know that there exists a constant $C_{T,M,x}>0$ such that
\begin{equation*}
\sup_{\phi\in \mathcal{A}_M}\Big\{\mathbb{E}\Big[\sup_{t\in[0,T]}\|\bar{X}^{\phi}(t)\|_{H}^2\Big]+\mathbb{E}\int_0^T\|\bar{X}^{\phi}(t)\|_{V}^{\alpha}dt\Big\}\leq C_{T,M,x},
\end{equation*}
\begin{equation}\label{es43}
\sup_{\phi\in \mathcal{A}_M}\Big\{\mathbb{E}\Big[\sup_{t\in[0,T]}\|X^{\varepsilon,\phi}(t)\|_{H}^2\Big]+\mathbb{E}\int_0^T\|X^{\varepsilon,\phi}(t)\|_{V}^{\alpha}dt\Big\}\leq C_{T,M,x}.
\end{equation}
Meanwhile, for $\mathcal{J}_3(t)$,  by (\ref{es42}) we have
\begin{eqnarray}\label{es44}
\mathcal{J}_3(t)\leq~~\!\!\!\!\!\!\!\!&& C\varepsilon\int_0^t(1+\|X^{\varepsilon,\phi^\varepsilon}(s)\|_H^2+\mathbb{E}\|X^{\varepsilon}(s)\|_H^2)ds
\nonumber\\
\leq~~\!\!\!\!\!\!\!\!&&C_{T}\varepsilon+C\varepsilon\int_0^t\|X^{\varepsilon,\phi^\varepsilon}(s)\|_H^2ds.
\end{eqnarray}
Now we define the following stopping time
\begin{eqnarray*}
\tau_R:=~~\!\!\!\!\!\!\!\!&&\inf\Big\{t\in[0,T]:\|\bar{X}^{\phi^\varepsilon}(t)\|_H+\|X^{\varepsilon,\phi^\varepsilon}(t)\|_H
\nonumber\\
\!\!\!\!\!\!\!\!&&~~~~~~~~~~+\int_0^t\big(\|\bar{X}^{\phi^\varepsilon}(s)\|_V^{\alpha}+\|X^{\varepsilon,\phi^\varepsilon}(s)\|_V^{\alpha}\big)ds\geq R\Big\}\wedge T,~R>0.
\end{eqnarray*}
As for $\mathcal{J}_4(t)$, by Burkholder-Davis-Gundy's inequality we have
\begin{eqnarray}\label{es46}
\!\!\!\!\!\!\!\!&&\mathbb{E}\Big[\sup_{t\in[0,T]}\mathcal{J}_4(t\wedge\tau_R)\Big]
\nonumber\\
\leq~~\!\!\!\!\!\!\!\!&&
 C\sqrt{\varepsilon}\mathbb{E}\Big[\int_0^{T\wedge\tau_R}\|B(t,X^{\varepsilon,\phi^\varepsilon}(t),\mathcal{L}_{X^{\varepsilon}(t)})\|_{L_2(U,H)}^2\|Y^\varepsilon(t)\|_H^2dt\Big]^{\frac{1}{2}}
\nonumber\\
\leq~~\!\!\!\!\!\!\!\!&&\frac{1}{2}\mathbb{E}\Big[\sup_{t\in[0,T\wedge\tau_R]}\|Y^\varepsilon(t)\|_H^2\Big]+C_T\varepsilon\mathbb{E}\Big[\int_0^{T\wedge\tau_R}(1+\|X^{\varepsilon,\phi^\varepsilon}(t)\|_H^2+\mathbb{E}\|X^{\varepsilon}(t)\|_H^2)dt\Big]
\nonumber\\
\leq~~\!\!\!\!\!\!\!\!&&\frac{1}{2}\mathbb{E}\Big[\sup_{t\in[0,T\wedge\tau_R]}\|Y^\varepsilon(t)\|_H^2\Big]+C_T\varepsilon,
\end{eqnarray}
where we used (\ref{es42}) and (\ref{es43}) in the last step.

Combining (\ref{es36}), (\ref{es49}), (\ref{es44}) with the definition of $\tau_R$, we infer that
\begin{eqnarray}\label{es47}
\|Y^{\varepsilon}(t\wedge\tau_R)\|_H^2
\leq~~\!\!\!\!\!\!\!\!&&\Big\{C_T\varepsilon+C_R\int_0^{t\wedge\tau_R}\|Y^{\varepsilon}(s)\|_H^2ds+C\varepsilon\int_0^{t\wedge\tau_R}\|X^{\varepsilon,\phi^\varepsilon}(s)\|_H^2ds
\nonumber\\
\!\!\!\!\!\!\!\!&&~~~
+\mathcal{J}_4(t\wedge\tau_R)\Big\}
\cdot
\exp\Big\{\int_0^{t\wedge\tau_R}\|\phi^\varepsilon(s)\|_U^2ds\Big\}.
\end{eqnarray}
Taking supremum and expectation on both sides of (\ref{es47}), by (\ref{es43}), (\ref{es46}) and the definition of $\mathcal{A}_M$ we have
\begin{equation*}
\mathbb{E}\Big[\sup_{t\in[0,T]}\|Y^{\varepsilon}(t\wedge\tau_R)\|_H^2\Big]
\leq C_{T,M}\varepsilon+C_{R,M}\mathbb{E}\int_0^{T}\|Y^{\varepsilon}(t\wedge\tau_R)\|_H^2dt.
\end{equation*}
Then Gronwall's lemma implies that
\begin{equation*}
\mathbb{E}\Big[\sup_{t\in[0,T\wedge\tau_R]}\|Y^{\varepsilon}(t)\|_H^2\Big]\leq C_{T,M}e^{C_{R,M}T}\varepsilon.
\end{equation*}
Consequently, we deduce that for any $\delta>0$,
\begin{eqnarray}\label{es50}
\mathbb{P}\Big(\sup_{t\in[0,T]}\|Y^{\varepsilon}(t)\|_H>\delta\Big)
=~~ \!\!\!\!\!\!\!\!&&\mathbb{P}\Big(\big\{\sup_{t\in[0,T]}\|Y^{\varepsilon}(t)\|_H>\delta\big\}\cap\{\tau_R\geq T\}\Big)
\nonumber\\
\!\!\!\!\!\!\!\!&&
+\mathbb{P}\Big(\big\{\sup_{t\in[0,T]}\|Y^{\varepsilon}(t)\|_H>\delta\big\}\cap\{\tau_R< T\}\Big)
\nonumber\\
\leq~~\!\!\!\!\!\!\!\!&&\frac{\mathbb{E}\Big[\sup_{t\in[0,T\wedge\tau_R]}\|Y^{\varepsilon}(t)\|_H^2\Big]}{\delta^2}
\nonumber\\
\!\!\!\!\!\!\!\!&&~~
+\frac{\mathbb{E}\Big[\sup_{t\in[0,T]}\big(\|\bar{X}^{\phi^\varepsilon}(t)\|_H+\|X^{\varepsilon,\phi^\varepsilon}(t)\|_H\big)\Big]}{R}
\nonumber\\
\!\!\!\!\!\!\!\!&&
+\frac{\mathbb{E}\Big[\int_0^T\big(\|\bar{X}^{\phi^\varepsilon}(t)\|_V^{\alpha}+\|X^{\varepsilon,\phi^\varepsilon}(t)\|_V^{\alpha}\big)dt\Big]}{R}
\nonumber\\
\leq~~\!\!\!\!\!\!\!\!&&
\frac{C_{T,M}e^{C_{R,M}T}}{\delta^2}\varepsilon+\frac{C_{T,M}}{R}.
\end{eqnarray}
Taking $\limsup_{\varepsilon\to0}$ and then $\lim_{R\to \infty}$ on both sides of (\ref{es50}), we conclude the desired result.
\end{proof}

\begin{proposition}\label{pro8}
Suppose that all assumptions in Theorem \ref{th3} hold. Let $\{\phi^n: n\in\mathbb{N}\}\subset S_M$ for any $M<\infty$ such that $\phi^n$ converges to element $\phi$ in $S_M$ as $n\to\infty$, then
$$\lim_{n\to\infty}\sup_{t\in[0,T]}\|\bar{X}^{\phi^n}(t)-\bar{X}^{\phi}(t)\|_H=0.$$
\end{proposition}
\begin{proof}
Let  $\bar{X}^{\phi^n}$ be the solution of (\ref{eqsk}) with $\phi^n\in S_M$ instead of $\phi$, i.e.,
\begin{equation*}
\frac{d\bar{X}^{\phi^n}(t)}{dt}=A(t,\bar{X}^{\phi^n}(t),\mathcal{L}_{X^0(t)})dt+B(t,\bar{X}^{\phi^n}(t),\mathcal{L}_{X^0(t)})\phi^n(t),~\bar{X}^{\phi^n}(0)=x\in H.
\end{equation*}
Then we define $Z^{n}(t)=\bar{X}^{\phi^n}(t)-\bar{X}^{\phi}(t)$ and it solves
\begin{equation*}
\left\{ \begin{aligned}
\frac{dZ^{n}(t)}{dt}=&A(t,\bar{X}^{\phi^n}(t),\mathcal{L}_{X^0(t)})-A(t,\bar{X}^{\phi}(t),\mathcal{L}_{X^0(t)})
\\
&+B(t,\bar{X}^{\phi^n}(t),\mathcal{L}_{X^0(t)})\phi^n(t)-B(t,\bar{X}^{\phi}(t),\mathcal{L}_{X^0(t)})\phi(t),\\
Z^{n}(0)=&~0.
\end{aligned} \right.
\end{equation*}
Applying the integration by parts formula we have
\begin{eqnarray*}
\|Z^{n}(t)\|_H^2
=~~\!\!\!\!\!\!\!\!&&2\int_0^t{}_{V^*}\langle A(s,\bar{X}^{\phi^n}(s),\mathcal{L}_{X^0(s)})-A(s,\bar{X}^{\phi}(s),\mathcal{L}_{X^0(s)}),Z^{n}(s)\rangle_Vds
\nonumber\\
\!\!\!\!\!\!\!\!&&~~~+2\int_0^t\langle \big(B(s,\bar{X}^{\phi^n}(s),\mathcal{L}_{X^0(s)})-B(s,\bar{X}^{\phi}(s),\mathcal{L}_{X^0(s)})\big)\phi^n(s),Z^{n}(s)\rangle_Hds
\nonumber\\
\!\!\!\!\!\!\!\!&&~~~+2\int_0^t\langle B(s,\bar{X}^{\phi}(s),\mathcal{L}_{X^0(s)})(\phi^n(s)-\phi(s)),Z^{n}(s)\rangle_Hds
\nonumber\\
=:~~\!\!\!\!\!\!\!\!&&\sum_{i=1}^3\mathcal{H}_i(t).
\end{eqnarray*}
Now the terms  $\mathcal{H}_i$, $i=1,2,3$, will be estimated one by one. Firstly, by $(\mathbf{H2})$ we have
\begin{eqnarray}\label{21}
\!\!\!\!\!\!\!\!&&\mathcal{H}_1(t)+\mathcal{H}_2(t)
\nonumber\\
\leq~~\!\!\!\!\!\!\!\!&&\int_0^t2{}_{V^*}\langle A(s,\bar{X}^{\phi^n}(s),\mathcal{L}_{X^0(s)})-A(s,\bar{X}^{\phi}(s),\mathcal{L}_{X^0(s)}),Z^{n}(s)\rangle_V
\nonumber\\
\!\!\!\!\!\!\!\!&&~~~+\|B(s,\bar{X}^{\phi^n}(s),\mathcal{L}_{X^0(s)})-B(s,\bar{X}^{\phi}(s),\mathcal{L}_{X^0(s)})\|_{L_2(U,H)}^2ds
\nonumber\\
\!\!\!\!\!\!\!\!&&~~~
+\int_0^t\|\phi^n(s)\|_U^2\|Z^{n}(s)\|_H^2ds
\nonumber\\
\leq~~\!\!\!\!\!\!\!\!&&\int_0^t(C+\rho(\bar{X}^{\phi}(s))+C\|X^0(s)\|^{\theta})\|Z^{n}(s)\|_H^2ds+C\int_0^t\|\phi^n(s)\|_U^2\|Z^{n}(s)\|_H^2ds.
\end{eqnarray}
Note  that $\phi^n\in S_M$, by (\ref{es22}), (\ref{es51}) and Gronwall's lemma we have
\begin{equation}\label{27}
\sup_{t\in[0,T]}\|Z^{n}(t)\|_H^2\leq C_{T,M}\Big[\sup_{t\in[0,T]}|\mathcal{H}_3(t)|\Big].
\end{equation}
Note also that
\begin{equation}
\sup_{t\in[0,T]}|\mathcal{H}_3(t)|\leq\sum_{i=1}^5\widetilde{\mathcal{H}}_i(n),
\end{equation}
where
\begin{eqnarray*}
\widetilde{\mathcal{H}}_1(n):=~~\!\!\!\!\!\!\!\!&&\sup_{t\in[0,T]}\Big|\int_0^{t}\langle B(s,\bar{X}^{\phi}(s),\mathcal{L}_{X^0(s)})(\phi^n(s)-\phi(s)),Z^n(s)-Z^n(s(\Delta))\rangle_{H}ds\Big|,
\nonumber\\
\widetilde{\mathcal{H}}_2(n):=~~\!\!\!\!\!\!\!\!&&\sup_{t\in[0,T]}\Big|\int_0^{t}\langle \big(B(s,\bar{X}^{\phi}(s),\mathcal{L}_{X^0(s)})-B(s(\Delta),\bar{X}^{\phi}(s),\mathcal{L}_{X^0(s)})\big)
\nonumber\\
\!\!\!\!\!\!\!\!&&~~~~~~~~~~~~~~
\cdot(\phi^n(s)-\phi(s)),Z^n(s(\Delta))\rangle_{H}ds\Big|,
\nonumber\\
\widetilde{\mathcal{H}}_3(n):=~~\!\!\!\!\!\!\!\!&&\sup_{t\in[0,T]}\Big|\int_0^{t}\langle \big(B(s(\Delta),\bar{X}^{\phi}(s),\mathcal{L}_{X^0(s)})-B(s(\Delta),\bar{X}^{\phi}(s(\Delta)),\mathcal{L}_{X^0(s(\Delta)))})\big)
\nonumber\\
\!\!\!\!\!\!\!\!&&~~~~~~~~~~~~~~\cdot(\phi^n(s)-\phi(s)),Z^n(s(\Delta))\rangle_{H}ds\Big|,
\nonumber\\
\widetilde{\mathcal{H}}_4(n):=~~\!\!\!\!\!\!\!\!&&\sup_{t\in[0,T]}\Big|\int_{t(\Delta)}^{t}\langle B(s(\Delta),\bar{X}^{\phi}(s(\Delta)),\mathcal{L}_{X^0(s(\Delta))})(\phi^n(s)-\phi(s)),Z^n(s(\Delta))\rangle_{H}ds\Big|,
\nonumber\\
\widetilde{\mathcal{H}}_5(n):=~~\!\!\!\!\!\!\!\!&&\sum_{k=0}^{[T/\Delta]-1}\Big|\langle B(k\Delta,\bar{X}^{\phi}(k\Delta),\mathcal{L}_{X^0(k\Delta)})\int_{k\Delta}^{(k+1)\Delta}(\phi^n(s)-\phi(s))ds,Z^n(k\Delta)\rangle_{H}\Big|.
\end{eqnarray*}
For $\widetilde{\mathcal{H}}_1(n)$, in view of Lemma \ref{lem7}, it is easy to show that
\begin{eqnarray*}
\widetilde{\mathcal{H}}_1(n)
\leq~~\!\!\!\!\!\!\!\!&&\int_0^T\|B(t,\bar{X}^{\phi}(t),\mathcal{L}_{X^0(t)})\|_{L_2(U,H)}\|\phi^n(t)-\phi(t)\|_U\|Z^n(t)-Z^n(t(\Delta))\|_Hdt
\nonumber\\
\leq~~\!\!\!\!\!\!\!\!&&C\Big\{\int_0^T\big(1+\|\bar{X}^{\phi}(t)\|_H^2+\|X^0(t)\|_H^2\big)\|\phi^n(t)-\phi(t)\|_U^2dt\Big\}^{\frac{1}{2}}
\nonumber\\
\!\!\!\!\!\!\!\!&&~~~\cdot\Big\{\int_0^T(\|\bar{X}^{\phi^n}(t)-\bar{X}^{\phi^n}(t(\Delta))\|_H^2+\|\bar{X}^{\phi}(t)-\bar{X}^{\phi}(t(\Delta)) \|_H^2)dt\Big\}^{\frac{1}{2}}
\nonumber\\
\leq~~\!\!\!\!\!\!\!\!&&C_{T,M}\Delta^{\frac{1}{2}}\Big\{\big(1+\sup_{\phi\in S_M}\{\sup_{t\in[0,T]}\|\bar{X}^{\phi}(t)\|_H^2\}+\sup_{t\in[0,T]}\|X^0(t)\|_H^2\big)
\nonumber\\
\!\!\!\!\!\!\!\!&&~~~\cdot
\int_0^T\|\phi^n(t)-\phi(t)\|_U^2dt\Big\}^{\frac{1}{2}}
\nonumber\\
\leq~~\!\!\!\!\!\!\!\!&&C_{T,M}\Delta^{\frac{1}{2}},
\end{eqnarray*}
where we used (\ref{es33}) and the fact that $\phi^n,\phi\in S_M$ in the last step.

For $\widetilde{\mathcal{H}}_2(n)$,  by $(\mathbf{H5})$ we have
\begin{eqnarray*}
\widetilde{\mathcal{H}}_2(n)
\leq~~\!\!\!\!\!\!\!\!&&\int_0^T\|B(t,\bar{X}^{\phi}(t),\mathcal{L}_{X^0(t)})-B(t(\Delta),\bar{X}^{\phi}(t),\mathcal{L}_{X^0(t)})\|_{L_2(U,H)}
\nonumber\\
\!\!\!\!\!\!\!\!&&~~~\cdot
\|\phi^n(t)-\phi(t)\|_U\|Z^n(t(\Delta))\|_Hdt
\nonumber\\
\leq~~\!\!\!\!\!\!\!\!&&C\sqrt{\sup_{\phi\in S_M}\{\sup_{t\in[0,T]}\|\bar{X}^{\phi}(t)\|_H^2\}}
\nonumber\\
\!\!\!\!\!\!\!\!&&~~~\cdot
\int_0^T\Delta^\gamma\Big(1+\|\bar{X}^{\phi}(t)\|_H+\|X^0(t)\|_H\Big)\|\phi^n(t)-\phi(t)\|_Udt
\nonumber\\
\leq~~\!\!\!\!\!\!\!\!&&C_{T,M}\Delta^{\gamma}.
\end{eqnarray*}
As for  $\widetilde{\mathcal{H}}_3(n)$, by H\"{o}lder's inequality we have
\begin{eqnarray*}
\widetilde{\mathcal{H}}_3(n)
\leq~~\!\!\!\!\!\!\!\!&&C\int_0^T\Big(\|\bar{X}^{\phi}(t)-\bar{X}^{\phi}(t(\Delta))\|_H+\mathbb{W}_{2,H}(\mathcal{L}_{X^0(t)},\mathcal{L}_{X^0(t(\Delta))})\Big)
\nonumber\\
\!\!\!\!\!\!\!\!&&~~~\cdot
\|\phi^n(t)-\phi(t)\|_U\|Z^n(t(\Delta))\|_Hdt
\nonumber\\
\leq~~\!\!\!\!\!\!\!\!&&C\Big\{\int_0^T\|Z^n(t(\Delta))\|_H^2\|\bar{X}^{\phi^n}(t)-\bar{X}^{\phi^n}(t(\Delta))\|_H^2dt\Big\}^{\frac{1}{2}}\Big\{\int_0^T\|\phi^n(t)-\phi(t)\|_U^2dt\Big\}^{\frac{1}{2}}
\nonumber\\
\!\!\!\!\!\!\!\!&&~~~+C\Big\{\int_0^T\|Z^n(t(\Delta))\|_H^2\mathbb{W}_{2,H}(\mathcal{L}_{X^0(t)},\mathcal{L}_{X^0(t(\Delta))})^2dt\Big\}^{\frac{1}{2}}\Big\{\int_0^T\|\phi^n(t)-\phi(t)\|_U^2dt\Big\}^{\frac{1}{2}}
\nonumber\\
\leq~~\!\!\!\!\!\!\!\!&&C_M\sqrt{\sup_{\phi\in S_M}\{\sup_{t\in[0,T]}\|\bar{X}^{\phi}(t)\|_H^2\}}\Big\{\int_0^T\|\bar{X}^{\phi^n}(t)-\bar{X}^{\phi^n}(t(\Delta))\|_H^2dt\Big\}^{\frac{1}{2}}
\nonumber\\
\!\!\!\!\!\!\!\!&&~~~+C_M\sqrt{\sup_{\phi\in S_M}\{\sup_{t\in[0,T]}\|\bar{X}^{\phi}(t)\|_H^2\}}\Big\{\int_0^T\|X^0(t)-X^0(t(\Delta))\|_H^2dt\Big\}^{\frac{1}{2}}.
\end{eqnarray*}
Since $X^0\in C([0,T];H)$, without loss of generality, we suppose that $\|X^0_t-X^0_{t(\Delta)}\|_H\leq\Delta$ for any $t\in[0,T]$. Then by Lemma \ref{lem7} we obtain that
$\widetilde{\mathcal{H}}_3(n)\leq C_{T,M}\Delta^{\frac{1}{2}}$.

By H\"{o}lder's inequality  we also deduce that
\begin{eqnarray*}
\widetilde{\mathcal{H}}_4(n)
\leq~~\!\!\!\!\!\!\!\!&&\sqrt{\sup_{\phi\in S_M}\{\sup_{t\in[0,T]}\|\bar{X}^{\phi}(t)\|_H^2\}}
\nonumber\\
\!\!\!\!\!\!\!\!&&~~~\cdot\Big\{\sup_{t\in[0,T]}\Big|\int_{t(\Delta)}^{t} \|B(s(\Delta),\bar{X}^{\phi}(s(\Delta)),\mathcal{L}_{X^0(t(\Delta))})\|_{L_2(U,H)}\|\phi^n(s)-\phi(s)\|_Uds\Big|\Big\}
\nonumber\\
\leq~~\!\!\!\!\!\!\!\!&&\Delta^{\frac{1}{2}}\sqrt{\sup_{\phi\in S_M}\{\sup_{t\in[0,T]}\|\bar{X}^{\phi}(t)\|_H^2\}}
\nonumber\\
\!\!\!\!\!\!\!\!&&~~~\cdot
\Big\{\sup_{t\in[0,T]}\Big|\int_{t(\Delta)}^{t} \|B(s(\Delta),\bar{X}^{\phi}(s(\Delta)),\mathcal{L}_{X^0(t(\Delta))})\|_{L_2(U,H)}^2\|\phi^n(s)-\phi(s)\|_U^2ds\Big|\Big\}^{\frac{1}{2}}
\nonumber\\
\leq~~\!\!\!\!\!\!\!\!&&C_{T,M}\Delta^{\frac{1}{2}}.
\end{eqnarray*}
Note that $\phi^n\to\phi$ in the weak topology of $S_M$, therefore for any $a,b\in[0,T]$,~$a<b$, the integral
$$\int_a^b\phi(s)^n ds\to\int_a^b\phi(s)ds~\text{weakly in}~U,~\text{as}~n\to\infty.$$
As for $\widetilde{\mathcal{H}}_5(n)$, since $B(k\Delta,\bar{X}^{\phi}_{k\Delta},\mu^0_{k\Delta})$ is a compact operator, then for any fixed $k$, the sequence $$B(k\Delta,\bar{X}^{\phi}(k\Delta),\mathcal{L}_{X^0(k\Delta)})\int_{k\Delta}^{(k+1)\Delta}(\phi^n(s)-\phi(s))ds\to 0~\text{strongly in}~H,~\text{as}~n\to\infty.$$
Hence, due to the boundedness of $Z^n(k\Delta)$, we infer that
\begin{equation}\label{28}
\lim_{n\to\infty}\widetilde{\mathcal{H}}_5(n)=0.
\end{equation}
Finally, combining (\ref{27})-(\ref{28}) and taking $\Delta\to0$, one can conclude that
$$\lim_{n\to\infty}\sup_{t\in[0,T]}\|Z^{n}(t)\|_H=0.$$
Now the proof is complete.
\end{proof}

\vspace{1mm}
Now it is easy  to complete the proof of  Theorem \ref{th3}.

\vspace{1mm}
\textbf{Proof of Theorem \ref{th3}:}  Combining Propositions \ref{pro7} and \ref{pro8},  it is clear that Lemma \ref{lem8} implies  $\{X^{\varepsilon}\}_{\varepsilon>0}$ satisfies the large deviation principle in $C([0,T]; H)$ with a good rate function $I$ defined in (\ref{rf}).

\section{Applications in finite-dimensional case}\label{sec6}
\subsection{McKean-Vlasov SDEs for granular media equation}
The main aim of this part is to use our main results to investigate the following nonlinear equation
\begin{equation}\label{6.3}
\partial_tf_t=\Delta f_t+{\rm{div}}\Big\{f_t\nabla V+f_t\nabla(W\ast f_t)\Big\},
\end{equation}
where  $(W\ast f_t)(x):=\int_{\mathbb{R}^d}W(x-z)f_t(z)dz$, the confinement potential $V:\mathbb{R}^d \rightarrow \mathbb{R}$ and the interaction potential
$W:\mathbb{R}^d \rightarrow \mathbb{R}$ satisfy the following conditions.
\begin{enumerate}

\item [$(\mathbf{A_v})$]\label{AV1} There exists $C,k>0$ such that $\nabla V(\cdot):\mathbb{R}^d\to\mathbb{R}^d$ is continuous and for $x,y\in\mathbb{R}^d$,
\begin{eqnarray*}
 \!\!\!\!\!\!\!\!&&\langle \nabla V(x),x\rangle \geq -C(1+|x|^2),
 \nonumber\\
 \vspace{2mm}
\!\!\!\!\!\!\!\!&& \langle\nabla V(x)-\nabla V(y),x-y\rangle\geq-C|x-y|^2,\label{es55}
 \\
 \vspace{2mm}
\!\!\!\!\!\!\!\!&& |\nabla V(x)|\leq C(1+|x|^k).\nonumber
\end{eqnarray*}

\item [$(\mathbf{A_w})$]\label{AW1} There exists $C,k>0$ such that for  $x,y\in\mathbb{R}^d$,
\begin{eqnarray}
\!\!\!\!\!\!\!\!&&
|\nabla W(x)-\nabla W(y)|\leq C(1+|x|^k+|y|^k)|x-y|, \label{es56}
\\
\!\!\!\!\!\!\!\!&&
|\nabla W(x)|\leq C(1+|x|).\nonumber
\end{eqnarray}
\end{enumerate}

Such nonlinear equation originally arises in the modelling of granular media. It has been studied e.g. in  \cite{CMV,M2} and references therein under various assumptions
on the potentials $V$ and $W$ via  analytical and probabilistic approach.

By applying  It\^{o}'s formula, we observe that the distribution of solution to the following equation
\begin{equation}\label{eqg}
dX(t)=-\nabla V(X(t))-\nabla (W\ast\mu(t))(X(t))dt+ \sqrt{2} dB(t)
\end{equation}
is in fact a weak solution (in PDE sense) to (\ref{6.3}), where $B(t)$ is a $d$-dimensional Wiener process and  $\mu(t)$ is the distribution of $X(t)$.
Therefore, we are able  to use the probabilistic approach to investigate (\ref{6.3}).

\begin{theorem}\label{th4}
Suppose that  $(\mathbf{A_v})$ and $(\mathbf{A_w})$ hold. Then for any $X(0)\in L^r(\Omega;\mathbb{R}^d)$ with $r>2k\vee 2$, (\ref{eqg}) has a unique strong/weak solution in the sense of Definition \ref{de1} provided satisfying
\begin{equation*}
\mathbb{E}\Big[\sup_{t\in[0,T]}|X(t)|^{r}\Big]<\infty.
\end{equation*}
\end{theorem}
\begin{remark}\label{re2}
As applications, the result in Theorem \ref{th4} can be applied to the potentials
\begin{equation}\label{pot1}
\nabla W(x)=-\frac{\beta Kx}{(1+|x|^2)^\gamma},\;\;\nabla V(x)=\beta(x^3-x),~\gamma\geq 0,
\end{equation}
where $\beta = \frac{1}{\kappa T}>0$ ($\kappa$ is the Boltzmann constant) is the inverse temperature. We call it ferromagnetic or anti-ferromagnetic respectively when $K>0$ or $K<0$. Note that (\ref{pot1}) is the extension of the classical Curie-Weiss mean-field  model ($\gamma=0$ in (\ref{pot1})) which satisfies the globally Lipschitz condition, whereas  the potentials in (\ref{pot1}) are only  locally Lipschitz (i.e. (\ref{es56})) when $\gamma>0$.

\end{remark}

\begin{proof}
It suffices to check that all conditions of Theorem \ref{th1} hold.  It is clear that  $(\mathbf{A1})$, $(\mathbf{A3})$ and $(\mathbf{A4})$ hold,  and for  $(\mathbf{A2})$, we  can calculate that
\begin{eqnarray*}
\!\!\!\!\!\!\!\!&&\langle b(x,\mu)-b(y,\nu),x-y\rangle
\nonumber\\
\leq~~\!\!\!\!\!\!\!\!&& C|x-y|^2+C|x-y|\int_{\mathbb{R}^d\times \mathbb{R}^d}\big(1+|y-z|^k+|y-z'|^k\big)|z-z'|\pi(dz,dz').
\end{eqnarray*}
Then by H\"{o}lder's inequality and taking infimum w.r.t.~all coupling of $\mu,\nu$ on both sides of the above inequality, we have
\begin{eqnarray*}
\!\!\!\!\!\!\!\!&&~~\langle b(x,\mu)-b(y,\nu),x-y\rangle
\nonumber\\
\leq~~\!\!\!\!\!\!\!\!&& C|x-y|^2
+C|x-y|\Big(\int_{\mathbb{R}^d\times \mathbb{R}^d}\big(1+|y-z|^k+|y-z'|^k\big)^2\pi(dz,dz')\Big)^{\frac{1}{2}}\mathbb{W}_2(\mu,\nu)\nonumber\\
\leq~~\!\!\!\!\!\!\!\!&& C\big(1+|y|^{2k}+\mu(|\cdot|^{2k})+\nu(|\cdot|^{2k})\big)|x-y|^2+\mathbb{W}_2(\mu,\nu)^2.
\end{eqnarray*}
Similarly, for  $(\mathbf{A2}''')$ we deduce that
\begin{eqnarray*}
\!\!\!\!\!\!\!\!&&\int_{\mathbb{R}^d\times \mathbb{R}^d}\langle b(x,\mu)-b(y,\nu),x-y\rangle\pi(dx,dy)
\nonumber\\
\leq~~\!\!\!\!\!\!\!\!&& C\int_{\mathbb{R}^d\times \mathbb{R}^d}|x-y|^2\pi(dx,dy)+C\int_{\mathbb{R}^d\times \mathbb{R}^d}\Bigg\{|x-y|\Big(\int_{\mathbb{R}^d\times \mathbb{R}^d}\big(1+|y-z|^k+|y-z'|^k\big)^2\pi(dz,dz')\Big)^{\frac{1}{2}}
\nonumber\\
\!\!\!\!\!\!\!\!&&~~\cdot
\Big(\int_{\mathbb{R}^d\times \mathbb{R}^d}|z-z'|^2\pi(dz,dz')\Big)^{\frac{1}{2}}\Bigg\}\pi(dx,dy)
\nonumber\\
\leq~~\!\!\!\!\!\!\!\!&&C\big(1+\mu(|\cdot|^{2k})+\nu(|\cdot|^{2k})\big)\int_{\mathbb{R}^d\times \mathbb{R}^d}|x-y|^2\pi(dx,dy).
\end{eqnarray*}
We complete the proof.
\end{proof}

Now we
consider the following MVSDE with small noise,
\begin{equation*}\label{essm1}
dX^\varepsilon(t)=-\nabla V(X^\varepsilon(t))-\nabla (W\ast\mu^\varepsilon(t))(X(t))dt+\sqrt{2\varepsilon}dB(t),~X^\varepsilon(0)=x\in\mathbb{R}^d,
\end{equation*}
where $\mu^\varepsilon(t)$ is the distribution of $X^\varepsilon(t)$.

\begin{theorem}
Suppose that  $(\mathbf{A_v})$ and $(\mathbf{A_w})$ hold. Then $\{X^\varepsilon\}_{\varepsilon>0}$
satisfies the LDP in $C([0,T]; \mathbb{R}^d)$ with the
good rate function $I$ given by (\ref{rf}).
\end{theorem}

\subsection{McKean-Vlasov SDEs for plasma type model}
Plasma model is an important model in mathematical physics,  which has the following form
\begin{equation}\label{6.4}
\partial_t f_t=\nabla\cdot\Big\{ \nabla f_t+f_t \nabla V+f_t (\nabla_xW\circledast f_t)\Big\},
\end{equation}
where $(\nabla_xW\circledast f_t)(x):=\int_{\mathbb{R}^d}\nabla_xW(x,z)f_t(z)dz$, the confinement potential $V:\mathbb{R}^d \rightarrow \mathbb{R}$ and the interaction functional
$W:\mathbb{R}^d \times\mathbb{R}^d\rightarrow \mathbb{R}$ satisfy the following conditions.
\begin{enumerate}

\item [$(\mathbf{A_v})$]\label{AV2} There exists $C,k>0$ such that $\nabla V(\cdot):\mathbb{R}^d\to\mathbb{R}^d$ is continuous and for any $x,y\in\mathbb{R}^d$,
\begin{eqnarray*}
 \!\!\!\!\!\!\!\!&&
 \langle \nabla V(x),x\rangle \geq -C(1+|x|^2),
 \nonumber\\
 \!\!\!\!\!\!\!\!&&\langle\nabla V(x)-\nabla V(y),x-y\rangle\geq-C|x-y|^2,
  \nonumber\\
 \!\!\!\!\!\!\!\!&&|\nabla V(x)|\leq C(1+|x|^k).\nonumber
\end{eqnarray*}

\item [$(\mathbf{A_w})$]\label{AW2} There exists $C,k>0$ such that $\nabla_x W(\cdot,z):\mathbb{R}^d\to \mathbb{R}^d$ is continuous and for any $x,y,z,z'\in\mathbb{R}^d$,
\begin{eqnarray*}
\!\!\!\!\!\!\!\!&&\langle \nabla_x W(x,z),x\rangle \geq -C(1+|x|^2+|z|^2),
 \nonumber\\
\!\!\!\!\!\!\!\!&&\langle\nabla_xW(x,z)-\nabla_x W(y,z),x-y\rangle\geq -C(1+|x|^k+|y|^k+|z|^k)|x-y|^2,\label{es60}
\\
\!\!\!\!\!\!\!\!&&|\nabla_x W(x,z)-\nabla_x W(x,z')|\leq C(1+|x|^k+|z|^k+|z'|^k)|z-z'|,
\nonumber\\
\!\!\!\!\!\!\!\!&&|\nabla_x W(x,z)|\leq C(1+|x|^k+|z|^k).\nonumber
\end{eqnarray*}
\end{enumerate}

The equivalent  probabilistic  version of (\ref{6.4}) is the following MVSDE
\begin{equation}\label{eqg2}
dX(t)=-\nabla V(X(t))-\big(\nabla_xW\circledast\mu(t)\big)(X(t))dt+\sqrt{2}dB(t),
\end{equation}
where $\mu(t)$ is the distribution of $X(t)$.

%

\begin{theorem}\label{th11}
Suppose that  $(\mathbf{A_v})$ and $(\mathbf{A_w})$ hold. Then for any $X(0)\in L^r(\Omega;\mathbb{R}^d)$ with $r>2k\vee 2$, (\ref{eqg2}) has a  unique strong/weak solution in the sense of Definition \ref{de1} provided satisfying
\begin{equation*}
\mathbb{E}\Big[\sup_{t\in[0,T]}|X(t)|^{r}\Big]<\infty.
\end{equation*}
\end{theorem}
\begin{proof}
The proof is similar to Theorem  \ref{th4}, so we omit the details here.
\end{proof}
\begin{remark}\label{re4} In particular, Theorem \ref{th11} can be applied to the interaction potential $W$ satisfying
$$\nabla_x W(x,z)=\frac{z}{(1+|x|^2)^\gamma},~\gamma\geq 0,$$
which is from the Cucker-Smale model (cf.~\cite{CS}).
\end{remark}

Now we consider the following MVSDE with small noise,
\begin{equation*}
dX^\varepsilon(t)=-\nabla V(X^\varepsilon (t))-\big(\nabla_xW\circledast\mu^\varepsilon(t)\big)(X^\varepsilon(t))dt+\sqrt{2\varepsilon}dB(t),~X^\varepsilon(0)=x\in\mathbb{R}^d,
\end{equation*}
where $\mu^\varepsilon(t)$ is the distribution of $X^\varepsilon(t)$.

\begin{theorem}
Suppose that  $(\mathbf{A_v})$ and $(\mathbf{A_w})$ hold.
Then $\{X^\varepsilon\}_{\varepsilon>0}$
satisfies the LDP in $C([0,T]; \mathbb{R}^d)$ with the
good rate function $I$ given by (\ref{rf}).
\end{theorem}

\subsection{McKean-Vlasov SDEs for kinetic equation}
Kinetic equation describes the motion of the particles containing the state variable $x$ and velocity variable $v$, where the derivative w.r.t.~$x$  is the velocity $v$ and the derivative w.r.t.~$v$ is the acceleration. Some of the models of swarming introduced in \cite{CS} and \cite{DCBC} do not belong to the globally bounded Lipschitz class due to their growth at infinitely leading to an interaction kernel which is only locally Lipschitz.

In this work, we consider the following kinetic PDE (also called Vlasov-Fokker-Planck equation) on $\mathbb{R}^d \times\mathbb{R}^d$
\begin{equation}\label{eq17}
\partial_t f_t+v\cdot\nabla_xf_t-\big(\nabla U(x)+\nabla W\ast_x \pi f_t\big)\cdot \nabla_{v}f_t=\Delta_{v}f_t+\nabla_{v}\cdot(vf_t)
\end{equation}
where $\ast_x$ is a convolution operator acting on variable $x$,
$\pi f_t(x):=\int_{\mathbb{R}^d}f_t(x,w)dw$ is the macroscopic density in the space of position $x\in\mathbb{R}^d$, $U$ and $W$ satisfy the following conditions.
\begin{enumerate}

\item [$(\mathbf{A_u})$]\label{AU}
There exists $C>0$ such that for $x,y\in\mathbb{R}^d$,
\begin{equation*}
 |\nabla U(x)-\nabla U(y)|\leq C|x-y|,\label{es61}
\end{equation*}

\item [$(\mathbf{A_w})$]\label{AW3}
There exists $C,k>0$ such that for $x,y\in\mathbb{R}^d$,
\begin{eqnarray*}
\!\!\!\!\!\!\!\!&&
 |\nabla W(x)-\nabla W(y)|\leq C(1+|x|^k+|y|^k)|x-y|,\label{es62}
 \\
 \!\!\!\!\!\!\!\!&&
 |\nabla W(x)|\leq C(1+|x|).\nonumber
\end{eqnarray*}
\end{enumerate}

We are interested in investigating the kinetic PDE (\ref{eq17})  through MVSDEs, that is, we consider  the following  equation
\begin{equation}\label{6.22}
\begin{cases}
\ dX(t)=V(t)dt\\
\ dV(t)=\sqrt{2}dB(t)-V(t)dt-\Big[\nabla U(X(t))+\int_{\mathbb{R}^d} \nabla W(X(t)-y)\mu(t)(dy)\Big]dt,
\end{cases}
\end{equation}
where $\mu(t)$ is the law of $X(t)$.

\begin{theorem}\label{th5}
Assume that $(\mathbf{A_u})$ and $(\mathbf{A_w})$ hold. Then for any $X(0)\in L^r(\Omega;\mathbb{R}^d)$ with $r>2k\vee 2$, (\ref{6.22}) has a unique strong/weak solution in the sense of Definition \ref{de1} provided satisfying
\begin{equation*}
\mathbb{E}\Big[\sup_{t\in[0,T]}|X(t)|^{r}\Big]<\infty.
\end{equation*}
\end{theorem}

\begin{remark}\label{re3}
 The typical application of Theorem \ref{th5} is the so-called D'Orsogna et al. model (cf.~e.g.~\cite{DCBC}). More precisely, we choose $U=0$ and $W$ is  a smooth radial potential given by
$$W(x):=-C_1e^{-|x|^2/l_1^2}+C_2e^{-|x|^2/l_2^2},$$
where $C_1,C_2$ and $l_1,l_2$ are the strengths and the typical lengths of attraction
and repulsion, respectively.
\end{remark}

\begin{proof}
We only prove $(\mathbf{A2})$, $(\mathbf{A3})$ and $(\mathbf{A2}''')$ in Theorem \ref{th1}. For any $x=(x_1,x_2)^{*}\in\mathbb{R}^d \times \mathbb{R}^d$ and $y=(y_1,y_2)^{*}\in\mathbb{R}^d \times \mathbb{R}^d$,
\begin{eqnarray*}
\!\!\!\!\!\!\!\!&&~~\langle b(x,\mu)-b(y,\nu),x-y\rangle
\nonumber\\
\leq~~\!\!\!\!\!\!\!\!&& C|x-y|^2-\langle \nabla U(x_1)-\nabla U(y_1),x_2-y_2\rangle\nonumber\\
\!\!\!\!\!\!\!\!&&~~~-\langle \int_{\mathbb{R}^d} \nabla W(x_1-z)\mu(dz)-\int_{\mathbb{R}^d} \nabla W(y_1-z')\nu(dz'),x_2-y_2\rangle\nonumber\\
\leq~~\!\!\!\!\!\!\!\!&& C|x-y|^2-\int_{\mathbb{R}^d} \langle \nabla W(x_1-z)-\nabla W(y_1-z),x_2-y_2\rangle\mu(dz)\nonumber\\
\!\!\!\!\!\!\!\!&&~~~-\int_{\mathbb{R}^d\times\mathbb{R}^d} \langle \nabla W(y_1-z)-\nabla W(y_1-z'),x_2-y_2\rangle\pi(dz,dz').
\end{eqnarray*}
Then by H\"{o}lder's inequality and taking infimum w.r.t.~the coupling of $\mu,\nu$ on both sides of the above inequality, we obtain
\begin{eqnarray*}
\!\!\!\!\!\!\!\!&&\langle b(x,\mu)-b(y,\nu),x-y\rangle\nonumber\\
\leq~~\!\!\!\!\!\!\!\!&&  C|x-y|^2+C|x_2-y_2|\big(1+|y_1|^{2k}+\nu(|\cdot|^{2k})+\mu(|\cdot|^{2k})\big)^{\frac{1}{2}}\mathbb{W}_2(\mu,\nu)\nonumber\\
\leq~~\!\!\!\!\!\!\!\!&& C\big(1+|y|^{2k}+\nu(|\cdot|^{2k})+\mu(|\cdot|^{2k})\big)|x-y|^2+\mathbb{W}_2(\mu,\nu)^2,
\end{eqnarray*}
which yields $(\mathbf{A2})$. Similar to the proof of Theorem \ref{th4}, we can show $(\mathbf{A2}''')$ holds.
As for $(\mathbf{A3})$, it follows that
\begin{align*}
\langle b(x,\mu),x\rangle&=\Big\langle \begin{pmatrix}
x_2\nonumber\\
-x_2-[\nabla U(x_1)+\int \nabla W(x_1-z)\mu(dz)]
\end{pmatrix},
\begin{pmatrix}
x_1\nonumber\\
x_2
\end{pmatrix}\Big\rangle\nonumber\\
&=\langle x_1,x_2\rangle-\langle x_2,x_2\rangle-\Big\langle \nabla U(x_1)+\int \nabla W(x_1-z)\mu(dz),x_2\Big\rangle\nonumber\\
&\leq C|x|^2+|x_2|\int_{\mathbb{R}^d}\nabla W(x_1-z)\mu(dz)\nonumber\\
&\leq C\big(1+|x|^2+\mu(|\cdot|^2)\big).
\end{align*}
Now  the proof is complete.
\end{proof}

\vspace{1mm}
For the LDP, we consider the following MVSDE with small noise
\begin{equation*}
\begin{cases}
&dX^\varepsilon(t)=V^\varepsilon(t)dt\\
&dV^\varepsilon(t)=\sqrt{2\varepsilon}dB(t)-V^\varepsilon(t)dt\\
&~~~~~~~~~~~~~~~-\Big[\nabla U(X^\varepsilon(t))+\int_{\mathbb{R}^d} \nabla W(X^\varepsilon(t)-y)\mu(t)(dy)\Big]dt,
\end{cases}
\end{equation*}
with initial value $(x,v)\in\mathbb{R}^d\times \mathbb{R}^d$.

\begin{theorem}
Suppose that  $(\mathbf{A_u})$ and $(\mathbf{A_w})$ hold.
Then $\{(X^\varepsilon,V^\varepsilon)\}_{\varepsilon>0}$
satisfies the LDP in $C([0,T]; \mathbb{R}^d\times\mathbb{R}^d)$ with the
good rate function $I$ given by (\ref{rf}).
\end{theorem}

\subsection{Further examples}
In addition to the examples above, in this part we also present some examples to illustrate the assumptions $(\mathbf{A2}')$ and $(\mathbf{A2}'')$ in Theorem \ref{th1}.

\vspace{1mm}
(i)  For the assumption $(\mathbf{A2}'')$,
we consider the following MVSDEs
\begin{equation}\label{exa}
dX(t)=\int_{\mathbb{R}}\tilde{b}(X(t),z)\mu(t)(dz)dt+\int_{\mathbb{R}}K_0\sin(X(t)-z)\mu(t)(dz)dW(t),
\end{equation}
where $K_0>0$, $\mu(t):=\mathcal{L}_{X(t)}$, $\tilde{b}:\mathbb{R}^2\rightarrow \mathbb{R}$ is bounded and satisfies for any $x,x',y,y'\in \mathbb{R}$,
$$|\tilde{b}(x,y)-\tilde{b}(x',y')|\leq (1+|x|+|x'|)(|x-x'|+|y-y'|).$$

%

\begin{theorem}
For any $\mathbb{E}e^{5|X(0)|}<\infty$, (\ref{exa}) has a unique strong/weak solution in the sense of Definition \ref{de1}  provided satisfying
$$\sup\limits_{t\in[0,T]}\mathbb{E}e^{5|X(t)|}<\infty.$$
\end{theorem}
\begin{proof}
The conditions $(\mathbf{A1})$, $(\mathbf{A3})$ and $(\mathbf{A4})$ in Theorem \ref{th1} hold obviously. We only need to prove $(\mathbf{A2}'')$.

Let $b(x,\mu):=\int_{\mathbb{R}}\tilde{b}(x,z)\mu(dz),~\sigma(x,\mu):=\int_{\mathbb{R}}\sin(x-z)\mu(dz)$.
It is easy to see that
\begin{align}\label{es64}
&|b(x,\mu)-b(y,\nu)|+\|\sigma(x,\mu)-\sigma(y,\nu)\|\nonumber\\
\leq&\int_{\mathbb{R}}|\tilde{b}(x,z)-\tilde{b}(y,z )|\mu(dz)+\int_{\mathbb{R}\times\mathbb{R}}|\tilde{b}(y,z)-\tilde{b}(y,z' )|\pi(dz,dz')\nonumber\\
&+\int_{\mathbb{R}}|\sin(x-z)-\sin(y-z)|\mu(dz)
\nonumber\\
&+\int_{\mathbb{R}\times\mathbb{R}}|\sin(x-z)-\sin(x-z')|\pi(dz,dz')\nonumber\\
\leq& (2+|x|+|y|)|x-y|+(C+2|y|)\int_{\mathbb{R}\times\mathbb{R}}|z-z'|\pi(dz,dz').
\end{align}
Then  we observe for any $|x|\vee |y|\leq R,~\mu,\nu\in \mathcal{P}_2(\mathbb{R}^d)$,
\begin{align*}
2\langle b(x,\mu)-b(y,\nu),x-y\rangle+\|\sigma(x)-\sigma(y)\|^2
\leq (C+4R)\big(|x-y|^2+\mathbb{W}_2(\mu,\nu)^2\big).
\end{align*}
Moreover, for any initial value $X(0)$ satisfying $\mathbb{E}e^{5|X(0)|}<\infty$, taking $f(x)=5x$ we can get that
$\sup\limits_{t\in[0,T]}\mathbb{E}e^{5|X(t)|}<\infty.$
Indeed, by It\^{o}'s formula we deduce that
\begin{align*}
&d{e}^{(1+|X(t)|^2)^{\frac{1}{2}}}\\
\leq & \frac{1}{2} {e}^{(1+|X(t)|^2)^{\frac{1}{2}}}\{2|b(X(t),\mathcal{L}_{X(t)})|+(1+|X(t)|^2)^{-\frac{1}{2}}\|\sigma(X(t),\mathcal{L}_{X(t)})\|^2\}dt\\
&\;-\frac{1}{2}{e}^{(1+|X(t)|^2)^{\frac{1}{2}}}{(1+|X(t)|^2)^{-\frac{1}{2}}\|\sigma(X(t),\mathcal{L}_{X(t)})\|^2}dt\\
&\;+{e}^{(1+|X(t)|^2)^{\frac{1}{2}}}(1+|X(t)|^2)^{-\frac{1}{2}}\langle X(t),\sigma(X(t),\mathcal{L}_{X(t)})dW(t)\rangle\\
\leq& C{e}^{(1+|X(t)|^2)^{\frac{1}{2}}}\{|b|_{\infty}+\|\sigma\|_{\infty}\}dt
\\
&+e^{(1+|X(t)|^2)^{\frac{1}{2}}}(1+|X(t)|^2)^{-\frac{1}{2}}\langle X(t),\sigma(X(t),\mathcal{L}_{X(t)})dW(t)\rangle.
\end{align*}
By a straightforward calculation, we have
$$\sup\limits_{t\in[0,T]}\mathbb{E}e^{5|X(t)|}<\sup\limits_{t\in[0,T]}\mathbb{E}e^{5(1+|X(t)|^2)^{\frac{1}{2}}}<\infty.$$
Thus,  $(\mathbf{A2}'')$  holds.
\end{proof}
\begin{remark}
We mention that the interaction potential $F(u):=K_0\sin u$ acts as an attractive force between particles  and the parameter $K_0 > 0$ modulates the strength of the force, which is referred as the Kuramoto model (or also the mean field classical XY model) that was originally from systems of chemical and biological oscillators, see e.g. \cite[(1.5)]{ABKO} and references therein.
\end{remark}

(ii) For the self-containedness of this work, we also give an example to illustrate  the application of $(\mathbf{A2}')$, which is inspired by \cite[Example 1.1]{RTW}. More precisely,  we consider the following MVSDEs on $\mathbb{R}$
\begin{equation}\label{exa2}
dX(t)=h\Big(\int_{\mathbb{R}}F(x)\mu(t)(dx)\Big)X(t)dt+h\Big(\int_{\mathbb{R}}F(x)\mu(t)(dx)\Big)X(t)dW(t),
\end{equation}
where $\mu(t):=\mathcal{L}_{X(t)}$, $h:\mathbb{R}\to\mathbb{R}$ is bounded and  Lipschitz continuous, $F$ is given by (constant $N>0$)
\begin{equation*}
F(x):=\left\{ \begin{aligned}
&|x|,~~|x|\leq N;
\\
&N,~~~~|x|> N.
\end{aligned} \right.
\end{equation*}

Then we have the following result.
\begin{theorem}
For any $X(0)\in L^r(\Omega;\mathbb{R})$ with $r>2$, (\ref{exa2}) has a unique strong/weak solution in the sense of Definition \ref{de1}  provided satisfying
\begin{equation*}
\mathbb{E}\Big[\sup_{t\in[0,T]}|X(t)|^{r}\Big]<\infty.
\end{equation*}
\end{theorem}

\begin{proof}
Following a similar argument as in \cite[Example 1.1]{RTW}, we can see that the coefficients of (\ref{exa2}) satisfy $(\mathbf{A1})$-$(\mathbf{A4})$ and $(\mathbf{A2}')$, then the result follows from Theorem \ref{th1}. Here we present the proof of $(\mathbf{A2}')$ for the readers' convenience.

Let
$$b(x,\mu)=\sigma(x,\mu):=h\Big(\int_{\mathbb{R}}F(x)\mu(dx)\Big)x.$$
First, for any $R\geq N$, $\xi,\eta\in \mathcal{C}_T$, $\mu,\nu\in \mathcal{P}_{2,T}$ and $t\in[0,T\wedge\tau_R^{\xi}\wedge\tau_R^{\eta}]$,
\begin{equation}\label{eqc4}
|b(\xi(t),\mu(t))-b(\eta(t),\nu(t))|\leq C_R\big(|\xi(t)-\eta(t)|+|\mu(t)(F(\cdot))-\nu(t)(F(\cdot))|\big).
\end{equation}
Denote
$$\|\xi-\eta\|_{\tau_R}:=\sup_{t\in[0,T\wedge\tau_R^{\xi}\wedge\tau_R^{\eta}]}|\xi(t)-\eta(t)|.$$
When $R\geq N$ and $t\in[0,T\wedge\tau_R^{\xi}\wedge\tau_R^{\eta}]$, we deduce that
\begin{equation*}
|F(\xi(t))-F(\eta(t))|\left\{ \begin{aligned}
&=||\xi(t)|-|\eta(t)||\leq \|\xi-\eta\|_{\tau_R},~|\xi(t)|\leq N,~|\eta(t)|\leq N;
\\
&=|N-|\eta(t)||\leq \|\xi-\eta\|_{\tau_R},~~~~~~|\xi(t)|> N,~|\eta(t)|\leq N;
\\
&=||\xi(t)|-N|\leq \|\xi-\eta\|_{\tau_R},~~~~~~|\xi(t)|\leq N,~|\eta(t)|> N;
\\
&=0\leq \|\xi-\eta\|_{\tau_R},~~~~~~~~~~~~~~~~~~~~~~~|\xi(t)|> N,~|\eta(t)|> N.
\end{aligned} \right.
\end{equation*}
Therefore, it follows that
$$|\mu(t)(F(\cdot))-\nu(t)(F(\cdot))|\leq \mathbb{W}_{2,T.R}(\mu_t,\nu_t),$$
which together with (\ref{eqc4}) implies $(\mathbf{A2}')$ holds.
\end{proof}

\section{Applications in infinite-dimensional case} \label{example1}
We denote by $\Lambda\subseteq\mathbb{R}^d$ an open bounded domain with a smooth boundary.  Let
$C_0^\infty(\Lambda, \mathbb{R}^d)$ be the space of all infinitely differentiable functions from $\Lambda$ to $\mathbb{R}^d$ with compact support.  For any $p\ge 1$, let $L^p(\Lambda, \mathbb{R}^d)$ be the vector valued $L^p$-space equipped with the norm $\|\cdot\|_{L^p}$.
For any integer $m>0$, let $W_0^{m,p}(\Lambda, \mathbb{R}^d)$ be the classical Sobolev space defined on $\Lambda$
and taking values in $\mathbb{R}^d$ with the (equivalent) norm
$$ \|u\|_{W^{m,p}} = \left( \sum_{0\le |\alpha|\le m} \int_{\Lambda} |D^\alpha u|^pd x \right)^\frac{1}{p}.$$

%

In what follows, we apply our main results to a class of nonlinear SPDEs perturbed  by interaction external force, typically,  the stochastic porous media equation, stochastic 2D hydrodynamical systems as examples.  Besides the aforementioned examples, our  well-posedness and LDP results are also applicable to the models derived in \cite{CKS,BKK,ES1}, and stochastic Burgers equations, stochastic power law fluid equation, see e.g.~\cite{LR2,LR13,LR1,LTZ}. In order to keep down the length of the paper, here we only give the following two class of examples in details.
\subsection{McKean-Vlasov stochastic 2D hydrodynamical systems}
Now we apply our general framework established in Theorems \ref{th2} and \ref{th3} to  McKean-Vlasov stochastic 2D hydrodynamical type systems, which include a large class of mathematical models arise in fluid dynamics (cf. e.g. \cite{CM}).

Let $H$ be a separable Hilbert space equipped with norm $\|\cdot\|_H$ and $L$ be an (unbounded) positive linear self-adjoint operator on $H$. Define
$V=\mathcal{D}(L^\frac{1}{2})$ and the associated norm $\|v\|_V=\|L^\frac{1}{2}v\|_H$ for any $v\in V$. Let $V^*$ be the dual space of $V$ with respect to the scalar product
$(\cdot,\cdot)$ on $H$.
Then we can consider a Gelfand triple $V \subset H \subset V^*$. Let us denote by $\langle{u},v\rangle$ the dualization between $u\in V$ and $v\in V^*$.
There exists an orthonormal basis $\{e_k\}_{k\geq 1}$ on $H$ of eigenfunctions of $L$ and the increasing eigenvalue sequence $0<\lambda_1\leq\lambda_2\leq...\leq\lambda_n\leq...\uparrow\infty.$

Let $F:V\times V \to V^*$ be a continuous map satisfying the following conditions.

\begin{enumerate}
\item[(C1)] $F: V \times V \to V^*$ is a continuous bilinear map.

\item [(C2)] For all $u_i \in V, i=1,2,3$,
\begin{equation*}
\langle F(u_1,u_2),u_3 \rangle = -\langle F(u_1,u_3),u_2\rangle,~~\langle F(u_1,u_2),u_2 \rangle=0.
\end{equation*}

\item[(C3)] There exists a Banach space $\mathcal{H}$ such that
\\
(i) $V \subset \mathcal{H} \subset H;$
\\
(ii) there exists a constant $a_0>0$ such that
\begin{equation*}
\|u\|_\mathcal{H}^2 \le a_0 \|u\|_H\|u\|_V~~\text{for all} ~ u\in V;
\end{equation*}
(iii) for every $\eta>0$ there exists a constant $C_\eta>0$ such that
\begin{equation*}
|\langle F(u_1,u_2),u_3 \rangle| \leq \eta\|u_3\|_V^2+C_\eta\|u_1\|_{\mathcal{H}}^2\|u_2\|_{\mathcal{H}}^2
~~\text{for all}~ u_i\in V,  i=1,2,3.
\end{equation*}
\end{enumerate}
For simplicity, we denote $F(u)=F(u,u)$.  Now we consider
the following McKean-Vlasov stochastic 2D hydrodynamical type systems
\begin{equation}\label{hys}
dX(t)+\big[L X(t)+F(X(t))\big]dt=\Phi(t,X(t),\mathcal{L}_{X(t)})dt+BdW(t).
\end{equation}

\begin{theorem}\label{2Dhys}
Assume that the embedding $V\subset H$ is compact, $B\in L_2(U,H)$, $F$ satisfies (C1)-(C3) and $\Phi$ satisfies $(\mathbf{H1})$-$(\mathbf{H4})$ and $(\mathbf{H2}')$. Then for any $X(0)\in L^p(\Omega,\mathcal{F}(0),\mathbb{P};H)$ with $p> 4\vee\theta$, (\ref{hys}) has a unique solution in the sense of Definition \ref{de2} provided satisfying
\begin{equation*}
\mathbb{E}\Big[\sup_{t\in[0,T]}\|X(t)\|_H^p\Big]+\mathbb{E}\int_0^T\|X(t)\|_V^{2}dt+\mathbb{E}\int_0^T\|X(t)\|_H^{p-2}\|X(t)\|_V^{2}dt<\infty.
\end{equation*}
\end{theorem}

\begin{proof}
It is easy to prove that  $(\mathbf{H1})$-$(\mathbf{H4})$ hold for $\tilde{A}(u):=-Lu-F(u,u)$ with $\alpha=\beta=2$,  we omit the details here (see. e.g. \cite{CM} or \cite{LR2}).
\end{proof}
\begin{remark}
For a concrete example of $\Phi$ one can see (\ref{exa2}) with $H$ replacing $\mathbb{R}$, we do not repeat here.
\end{remark}

Next, we  consider the following McKean-Vlasov stochastic hydrodynamical type systems with small Gaussian noise
\begin{equation}\label{hys1}
dX^\varepsilon(t)=-\Big[L X^\varepsilon(t)+F(X^\varepsilon(t))+\Phi(t,X^\varepsilon(t),\mathcal{L}_{X^\varepsilon(t)})\Big]dt
+\sqrt{\varepsilon}BdW(t),
\end{equation}
with the initial value $X^\varepsilon(0)=x\in H$.

\begin{theorem}
Assume that the conditions in Theorem \ref{2Dhys} hold. Then $\{X^\varepsilon\}_{\varepsilon>0}$
satisfies the LDP in $C([0,T]; H)$ with the
good rate function $I$ given by (\ref{rf}).
\end{theorem}

\begin{remark}\label{remark6.2}
(i) As stated in  \cite{CM}, the results obtained in this subsection are applicable to various hydrodynamical models with distribution dependent coefficients, e.g. 2D Navier-Stokes equation, magneto-hydrodynamic equations,  Boussinesq equations, etc.

(ii) It is interesting to investigate the mean field limit problem of following $N$-interacting 2D hydrodynamical systems
\begin{equation*}
dX^{N,i}(t)=-\Big[LX^{N,i}(t)+F(X^{N,i}(t))\Big]dt+\frac{1}{N}\sum_{j=1}^N\big(X^{N,i}(t)-X^{N,j}(t)\big)dt+dW^i(t),
\end{equation*}
where the interaction kernel $K(u,v):=u-v$ is called the ``Stokes drug force'' in the dynamics of fluid, which is  proportional to the relative velocity of particles. This type model could be used to describe the dynamics of interaction between large number of particles in fluids, it allows us to consider micro-organisms like bacteria who can ``swim" in the fluid and one interesting subject is to study  its microscopic limit. Such propagation of chaos problem for the $N$-interacting 2D hydrodynamical systems will be investigated in the future work.
\end{remark}

Besides the semilinear SPDEs, our main results  are also applicable to a class of McKean-Vlasov quasilinear SPDEs.

\subsection{McKean-Vlasov stochastic porous media equation}
Denote by $(E,\mathcal{M},\textbf{m})$ a separable probability space and $(L,\mathcal{D}(L))$ a negative definite linear self-adjoint map defined on
$(L^2(\textbf{m}),\langle\cdot,\cdot\rangle)$, which has discrete spectrum with eigenvalues
$$0>-\lambda_1\geq-\lambda_2\geq\cdots\rightarrow-\infty.$$
Let $H$ be the topological dual space of $\mathcal{D}(\sqrt{-L})$, which is equipped with the inner product
$$\langle u,v\rangle_{H}:=\int_E\big(\sqrt{-L}u(\xi)\big)\cdot\big(\sqrt{-L}v(\xi)\big)d\xi,~u,v\in H,$$
then identify $L^2(\textbf{m})$ with its dual, we can get the continuous and dense embedding
$$\mathcal{D}(\sqrt{-L})\subseteq L^2(\textbf{m})\subseteq H.$$
Assume that $L^{-1}$ is continuous in $V:=L^{r+1}(\textbf{m})$, where $r>1$ is a fixed number. Then we have a presentation of its dual space $V^*$ by the following embedding
$$V\subset H\simeq \mathcal{D}(\sqrt{-L})\subset V^*,$$
where $\simeq$ is understood through $\sqrt{-L}$.

Now  we consider the following McKean-Vlasov stochastic porous media equation:
\begin{equation}\label{PME}
dX(t)=\Big[L\Psi(X(t))+\Phi(t,X(t),\mathcal{L}_{X(t)})\Big]dt+BdW(t),
\end{equation}
where $W(t)$ is a cylindrical Wiener process defined
on a probability space $(\Omega,\mathcal {F},\{\mathcal{F}(t)\}_{t\in[0,T]},\mathbb{P})$ and taking values in a sparable Hilbert space $U$, $\Psi:\mathbb{R}\rightarrow\mathbb{R}$
is a continuous and measurable map such that there are some constants $\delta>0$ and $K$,
\begin{eqnarray}
&&|\Psi(s)|\leq K(1+|s|^r),~~s\in \mathbb{R};
\label{PME3}\\
&&-\langle\Psi(u)-\Psi(v),u-v\rangle\leq-\delta\|u-v\|_{V}^{r+1},~~u,v\in V.\label{PME4}
\end{eqnarray}

\begin{theorem}\label{main result PME2}
Suppose that the embedding $V\subset H$ is compact, $\Psi$ satisfy the conditions (\ref{PME3})-(\ref{PME4}) and $\Phi$ satisfies $(\mathbf{H1})$-$(\mathbf{H4})$ and $(\mathbf{H2}'')$. Then for any $X(0)\in L^p(\Omega,\mathcal{F}(0),\mathbb{P};H)$ with $p>2r\vee \theta$,
(\ref{PME}) admits a unique solution in the sense of Definition \ref{de2}  provided satisfying
\begin{equation*}
\mathbb{E}\Big[\sup_{t\in[0,T]}\|X(t)\|_H^p\Big]+\mathbb{E}\int_0^T\|X(t)\|_V^{r+1}dt+\mathbb{E}\int_0^T\|X(t)\|_H^{p-2}\|X(t)\|_V^{r+1}dt<\infty.
\end{equation*}
\end{theorem}

\begin{proof}
By a standard argument  we can show that the map $\tilde{A}:=L\Psi$ satisfy $(\mathbf{H1})$-$(\mathbf{H4})$ for $\rho\equiv0,~\beta=0,~\alpha=r+1$, we refer to \cite[Example 4.1.11]{LR1} for the detailed proof.
 Therefore, the conclusion is a direct consequence of Theorem \ref{th2}.
\end{proof}

\vspace{3mm}
Next, we consider  the following McKean-Vlasov stochastic porous media equation with small Gaussian noise
\begin{equation}\label{PME1}
dX^\varepsilon(t)=\Big[L\Psi(X^\varepsilon(t))+\Phi(t,X^\varepsilon(t),\mathcal{L}_{X^\varepsilon(t)})\Big]dt+\sqrt{\varepsilon}BdW(t),
\end{equation}
where the initial value $X^\varepsilon(0)=x\in H$.

\vspace{1mm}
As a consequence of Theorem \ref{th3}, we can obtain the LDP result for  (\ref{PME1}).
\begin{theorem}
Assume that the conditions in Theorem \ref{main result PME2} hold. Then $\{X^\varepsilon\}_{\varepsilon>0}$
satisfies the LDP in $C([0,T]; H)$ with the
good rate function $I$ given by (\ref{rf}).
\end{theorem}

\begin{remark}
A typical example is  $L=\Delta$, the Laplace operator on a smooth bounded domain in a
complete Riemannian manifold with Dirichlet boundary and
$$\Psi(s):=|s|^{r-1}s,~s\in\mathbb{R},r>1.$$
 The well-posedness and LDP result for McKean-Vlasov stochastic porous media equation have been established in the recent works \cite{HE,HL1,HLL3}, however, under our new framework constructed in Section \ref{secin},  one can easily extend the results of \cite{HE,HL1,HLL3} to more general cases.


\end{remark}

\begin{appendix}
\section*{}\label{secapp} 
We first recall the following   criterion of tightness (cf.~Theorems 4.10 and 4.11 in \cite{KS1}).
\begin{lemma}\label{lem1} (Tightness)
For any $T>0$, the family $\{X^{(n)}\}_{n\in\mathbb{N}}$ is tight in $C([0,T];\mathbb{R}^d)$ if  the following two conditions hold.

(i) There exists  a constant $r>0$ such that
\begin{equation*}
\sup_{n\in\mathbb{N}}\mathbb{E}|X^{(n)}(0)|^r<\infty.
\end{equation*}

(ii) There exist constants $r,\delta>0$ such that for any $0\leq s,t\leq T$,
\begin{equation*}
\sup_{n\in\mathbb{N}}\mathbb{E}|X^{(n)}(t)-X^{(n)}(s)|^r\leq C_T|t-s|^{1+\delta}.
\end{equation*}

\end{lemma}

The Skorokhod theorem allows for the representation of the limit measure of a weakly convergent probability measures sequence on a metric space as the distribution of a pointwise convergent random variables sequence defined on a common probability space.

\begin{lemma}\label{lem22}
(Skorokhod representation theorem)
Let $E$ be a separable metric space. For an arbitrary sequence of probability measures $\{\mu_n\}_{n\geq1}$ on $\mathcal{B}(E)$ weakly convergence to a probability measure
$\mu$,  then there exists a probability space $(\tilde{\Omega},\tilde{\mathcal{F}},\tilde{\mathbb{P}})$ and a sequence of random variables $\tilde{X}_n,\tilde{X}$ such that they coincide in distribution, under $\tilde{\mathbb{P}}$, with $\mu_n,\mu$, respectively. Furthermore,  $\tilde{X}_n\to \tilde{X}$, $\tilde{\mathbb{P}}$-a.s., as $n\to\infty$.
\end{lemma}

Now we recall the following Vitali's convergence theorem, which is often used to prove the $L^p$-convergence.
\begin{lemma} (Vitali's convergence theorem)\label{lem9}
Let $p\geq1$, the random variables sequence $\{X_n\}_{n\geq1}\subset L^p(\Omega)$, then the following statement are equivalent:

\vspace{2mm}
(i) $X_n\to X$ in $L^p(\Omega)$-sense  as $n\to\infty$.

(ii) $X_n\to X$ in probability as $n\to\infty$ and $|X_n|^p$ is  uniformly integrable;

(iii) $X_n\to X$ in probability and $\mathbb{E}|X_n|^p\to\mathbb{E}|X|^p$ as $n\to\infty$ .
\end{lemma}

The following modified version of Yamada-Watanabe theorem is established recently in \cite[Lemma 2.1]{HW1}, which plays an important role in proving the existence of strong solution from the weak solution.
\begin{lemma}\label{lem2}(Modified Yamada-Watanabe theorem)
Assume that the distribution dependent SDE
\begin{equation}\label{eq13}
dX(t)=b(t,X(t),\mathcal{L}_{X(t)})dt+\sigma(t,X(t),\mathcal{L}_{X(t)})dW(t)
\end{equation}
has a weak solution $\{X(t)\}_{t\in[0,T]}$ under the probability measure $\mathbb{P}$, and let $\mu(t)=\mathcal{L}_{X(t)}|_{\mathbb{P}}$, $t\in[0,T]$.
If the SDE
$$dX(t)=b(t,X(t),\mu(t))dt+\sigma(t,X(t),\mu(t))dW(t)$$
has strong uniqueness for some initial value $X(0)$ with $\mathcal{L}_{X(0)}=\mu(0)$, then (\ref{eq13}) has a strong solution starting at $X(0)$. Moreover, if (\ref{eq13}) has strong uniqueness for any initial value with $\mathcal{L}_{X(0)}=\mu(0)$, then it is weakly well-posed for the initial distribution $\mu(0)$.
\end{lemma}
\end{appendix}
%
%
%


\begin{funding}
S. Hu is supported by the Deutsche Forschungsgemeinschaft (DFG, German Research Foundation)-IRTG 2235-Project number 282638148.     W. Liu is supported by NSFC (No. 12171208, 12090011, 12090010, 11831014) and the PAPD of Jiangsu Higher Education Institutions.
\end{funding}

\vspace{5mm}
\noindent\textbf{Acknowledgements} {The authors are grateful to the referees whose constructive comments and suggestions have helped to greatly improve the quality of this paper.

Wei Liu is the corresponding author.}

\end{document}